\documentclass[a4paper, twoside, 11pt, english]{article}

\usepackage{microtype}
\usepackage{youngtab}
\usepackage{amsmath, amsthm, amsfonts, amssymb}
\usepackage[T1]{fontenc}
\usepackage{babel}
\usepackage{times, euler, euscript}
\usepackage{fancyhdr, titlesec, url, enumerate}
\usepackage{xy4}
\usepackage{tikz}
\usepackage{hyperref}
\usepackage{ifthen}
\usepackage{comment}
\usepackage{array}
\usepackage{multirow}
\usepackage{cancel}
\usepackage{float}
\usepackage{time}
\CompileMatrices

\usetikzlibrary{calc}

\DeclareMathAlphabet{\mathscr}{T1}{pzc}{m}{it}

\newcount\hh
\newcount\mm
\mm=\time
\hh=\time
\divide\hh by 60
\divide\mm by 60
\multiply\mm by 60
\mm=-\mm
\advance\mm by \time
\def\hhmm{\number\hh:\ifnum\mm<10{}0\fi\number\mm}

\titleformat{\section}[block]{\scshape\filcenter\Large}{\thesection.}{.5em}{}
\titleformat{\subsection}[block]{\bfseries\filcenter\large}{\thesubsection.}{.5em}{\medskip}
\titleformat{\subsubsection}[runin]{\bfseries}{\thesubsubsection.}{.5em}{}[.]
\titlespacing{\subsubsection}{0pt}{10pt}{.5em}

\newtheoremstyle{ntheorem}%
	{\topsep}{\topsep}{\itshape}{0pt}{\bfseries}{.}{.5em}%
	{\thmnumber{#2.\hspace{.5em}}\thmname{#1}\thmnote{ (#3)}}
	
\newtheoremstyle{ndefinition}%
	{\topsep}{\topsep}{\normalfont}{0pt}{\bfseries}{.}{.5em}%
	{\thmnumber{#2.\hspace{.5em}}\thmname{#1}\thmnote{ (#3)}}
	
\newtheoremstyle{nremark}%
	{\topsep}{\topsep}{\normalfont}{0pt}{\itshape}{.}{.5em}%
	{\thmnumber{}\thmname{#1}\thmnote{ (#3)}}

\theoremstyle{ntheorem}
  	\newtheorem{theorem}[subsubsection]{Theorem}
  	\newtheorem{proposition}[subsubsection]{Proposition}
	\newtheorem{lemma}[subsubsection]{Lemma}
  	\newtheorem{corollary}[subsubsection]{Corollary}

\theoremstyle{ndefinition}

	\newtheorem{example}[subsubsection]{Example}
	\newtheorem{remark}[subsubsection]{Remark}

\fancypagestyle{plain}{
\fancyhf{}\fancyfoot[LC,RC]{}
\fancyfoot[LE,RO]{$\thepage$}

}

\newcommand{\pdf}[1]{\texorpdfstring{$#1$}{1}}

\UseTips
\SelectTips{eu}{11}

\newdir{ >}{{}*!/-10pt/@{>}}
\newdir{ -}{{}*!/-10pt/@{}}
\newdir{> }{{}*!/+10pt/@{>}}

\makeatletter

\xyletcsnamecsname@{dir4{}}{dir{}}
\xydefcsname@{dir4{-}}{\line@ \quadruple@\xydashh@}
\xydefcsname@{dir4{.}}{\point@ \quadruple@\xydashh@}
\xydefcsname@{dir4{~}}{\squiggle@ \quadruple@\xybsqlh@}
\xydefcsname@{dir4{>}}{\Tttip@}
\xydefcsname@{dir4{<}}{\reverseDirection@\Tttip@}

\xydef@\quadruple@#1{%
	\edef\Drop@@{%
		\dimen@=#1\relax
		\dimen@=.5\dimen@
		\A@=-\sinDirection\dimen@
		\B@=\cosDirection\dimen@
		\setboxz@h{%
			\setbox2=\hbox{\kern3\A@\raise3\B@\copy\z@}%
			\dp2=\z@ \ht2=\z@ \wd2=\z@ \box2
			\setbox2=\hbox{\kern\A@\raise\B@\copy\z@}%
			\dp2=\z@ \ht2=\z@ \wd2=\z@ \box2
			\setbox2=\hbox{\kern-\A@\raise-\B@\copy\z@}%
			\dp2=\z@ \ht2=\z@ \wd2=\z@ \box2
			\setbox2=\hbox{\kern-3\A@\raise-3\B@ \noexpand\boxz@}%
			\dp2=\z@ \ht2=\z@ \wd2=\z@ \box2
		}%
		\ht\z@=\z@ \dp\z@=\z@ \wd\z@=\z@ \noexpand\styledboxz@
	}%
}

\xydef@\Tttip@{\kern2pt \vrule height2pt depth2pt width\z@
	\Tttip@@ \kern2pt \egroup
	\U@c=0pt \D@c=0pt \L@c=0pt \R@c=0pt \Edge@c={\circleEdge}%
	\def\Leftness@{.5}\def\Upness@{.5}%
	\def\Drop@@{\styledboxz@}\def\Connect@@{\straight@{\dottedSpread@\jot}}}
	
\xydef@\Tttip@@{%
	\dimen@=.25\dimen@
 	\B@=\cosDirection\dimen@
	\setboxz@h\bgroup\reverseDirection@\line@ \wdz@=\z@ \ht\z@=\z@ \dp\z@=\z@
	{\vDirection@(1,-1)\xydashl@ \xyatipfont\char\DirectionChar}%
	{\vDirection@(1,+1)\xydashl@ \xybtipfont\char\DirectionChar}%
}

\xydef@\ar@form{
	\ifx \space@\next \expandafter\DN@\space{\xyFN@\ar@form}%
	\else\ifx ^\next \DN@ ^{\xyFN@\ar@style}\edef\arvariant@@{\string^}%
	\else\ifx _\next \DN@ _{\xyFN@\ar@style}\edef\arvariant@@{\string_}%
	\else\ifx 0\next \DN@ 0{\xyFN@\ar@style}\def\arvariant@@{0}%
	\else\ifx 1\next \DN@ 1{\xyFN@\ar@style}\def\arvariant@@{1}%
	\else\ifx 2\next \DN@ 2{\xyFN@\ar@style}\def\arvariant@@{2}%
	\else\ifx 3\next \DN@ 3{\xyFN@\ar@style}\def\arvariant@@{3}%
	\else\ifx 4\next \DN@ 4{\xyFN@\ar@style}\def\arvariant@@{4}%
	\else\ifx \bgroup\next \let\next@=\ar@style
	\else\ifx [\next \DN@[##1]{\ar@modifiers{[##1]}}
	\else\ifx *\next \DN@ *{\ar@modifiers}%
	\else\addLT@\ifx\next \let\next@=\ar@slide
	\else\ifx /\next \let\next@=\ar@curveslash
	\else\ifx (\next \let\next@=\ar@curveinout 
	\else\addRQ@\ifx\next \addRQ@\DN@{\ar@curve@}%
	\else\addLQ@\ifx\next \addLQ@\DN@{\xyFN@\ar@curve}%
	\else\addDASH@\ifx\next \addDASH@\DN@{\defarstem@-\xyFN@\ar@}%
	\else\addEQ@\ifx\next \addEQ@\DN@{\def\arvariant@@{2}\defarstem@-\xyFN@\ar@}%
	\else\addDOT@\ifx\next \addDOT@\DN@{\defarstem@.\xyFN@\ar@}%
	\else\ifx :\next \DN@:{\def\arvariant@@{2}\defarstem@.\xyFN@\ar@}%
	\else\ifx ~\next \DN@~{\defarstem@~\xyFN@\ar@}%
	\else\ifx !\next \DN@!{\dasharstem@\xyFN@\ar@}%
	\else\ifx ?\next \DN@?{\ar@upsidedown\xyFN@\ar@}%
	\else \let\next@=\ar@error
	\fi\fi\fi\fi\fi\fi\fi\fi\fi\fi\fi\fi\fi\fi\fi\fi\fi\fi\fi\fi\fi\fi\fi \next@}

\makeatother

\newcommand{\fl}{\rightarrow}
\newcommand{\fll}{\longrightarrow}
\newcommand{\ofl}[1]{\overset{\displaystyle #1}{\fll}}

\newcommand{\dfl}{\Rightarrow}
\newcommand{\dfll}{\Longrightarrow}
\newcommand{\odfl}[1]{\overset{\displaystyle #1}{\dfll}}

\newcommand{\tfl}{\Rrightarrow}
\newcommand{\qfl}{\xymatrix@1@C=10pt{\ar@4 [r] &}}

\newcommand{\cl}[1]{\overline{#1}}
\newcommand{\rep}[1]{\widehat{#1}}

\newcommand{\tck}[1]{#1^{\top}}

\renewcommand{\phi}{\varphi}
\renewcommand{\epsilon}{\varepsilon}

\newcommand{\M}{\mathbf{M}}

\def\catego#1{{\bf{\sf #1}}}

\newcommand{\Cat}{\catego{Cat}}

\newcommand{\ifthen}[2]{\ifthenelse{#1}{#2}{}}

\newcommand{\typedeuxbase}[3]{
\begin{tikzpicture}[
	scale = 0.5, every node/.style = { inner sep = 0mm },
	onecell/.style = { text height = 1.5ex },
	]
	
\node [onecell] (u) at (0,0) {#1} ;
\node [onecell] (v) at (1,0) {#2} ;

\ifthen{\equal{#3}{1}}
	{
		\node (uv) at ($ (u.north) + (0.5,0.2) $) {} ;
		\draw ($ (u.north) + (0.1,-0.1) $) -- (uv.south) -- ($ (v.north) + (-0.1,-0.1) $) ;
	}
\ifthen{\equal{#3}{0}}
	{
		\node (uv) at ($ (u.north) + (0.5,0) $) {$\scriptstyle\times$} ;
	}
\ifthen{\equal{#3}{01}}
	{
		\node (uv) at ($ (u.north) + (0.5,0) $) {$\scriptstyle\times 1$} ;
	}
\ifthen{\equal{#3}{02}}
	{
		\node (uv) at ($ (u.north) + (0.5,0) $) {$\scriptstyle\times 2$} ;
	}
\end{tikzpicture}
}

\newcommand{\typedeux}[1]{\typedeuxbase{$u$}{$v$}{#1}}

\newcommand{\typetroisbasep}[4]{
\begin{tikzpicture}[
	scale = 0.5, every node/.style = { inner sep = 0mm },
	onecell/.style = { text height = 1ex },
	]

\node [onecell] (u) at (0,0) {#1} ;
\node [onecell] (v) at (1,0) {#2} ;
\node [onecell] (w) at (2,0) {#3} ;

\ifthen{\equal{#4}{00}}
	{
                       \node (uv) at ($ (u.north) + (0.5,0.2) $) {$\scriptstyle\times$} ;
                       \node (vw) at ($ (v.north) + (0.5,0.2) $) {$\scriptstyle\times$} ;
	}

\ifthen{\equal{#4}{11}}
	{
                       \node (uv) at ($ (u.north) + (0.5,0.2) $) {$\scriptstyle\times 1$} ;
                       \node (vw) at ($ (v.north) + (0.5,0.2) $) {$\scriptstyle\times 1$} ;
	}
\ifthen{\equal{#4}{12}}
	{
                       \node (uv) at ($ (u.north) + (0.5,0.2) $) {$\scriptstyle\times 1$} ;
                       \node (vw) at ($ (v.north) + (0.5,0.2) $) {$\scriptstyle\times 2$} ;
             }
\ifthen{\equal{#4}{21}}
	{
                       \node (uv) at ($ (u.north) + (0.5,0.2) $) {$\scriptstyle\times 2$} ;
                       \node (vw) at ($ (v.north) + (0.5,0.2) $) {$\scriptstyle\times 1$} ;
             }
\ifthen{\equal{#4}{22}}
	{
                       \node (uv) at ($ (u.north) + (0.5,0.2) $) {$\scriptstyle\times 2$} ;
                       \node (vw) at ($ (v.north) + (0.5,0.2) $) {$\scriptstyle\times 2$} ;
             }
\end{tikzpicture}
}

\newcommand{\tableau}[1]{\texttt{#1}}

\renewcommand{\leq}{\leqslant}
\renewcommand{\geq}{\geqslant}

\def\col{\mathrm{col}}
\DeclareMathOperator{\Path}{Path}
\DeclareMathOperator{\Colo}{Col}
\DeclareMathOperator{\PreColo}{PreCol}
\DeclareMathOperator{\Knuth}{Knuth}
\def\Knuthc{\Knuth^{\mathrm{c}}}
\def\Knuthcc{\Knuth^{\mathrm{cc}}}
\def\Knuthkb{\Knuth^{\mathrm{KB}}}
\def\Knuthpath{\Knuth^{\mathrm{path}}}
\DeclareMathOperator{\CC}{C}
\DeclareMathOperator{\PC}{PC}
\DeclareMathOperator{\CPC}{CPC}
\DeclareMathOperator{\Crys}{Crys}

\newcommand{\odfll}[1]{\overset{\displaystyle #1}{\dfll}}

\def\P{\mathbf{P}}

\def\lnds{\ell^{\mathrm{nds}}} 
\def\len{\ell}

\def\ordrecoldeglex{\preccurlyeq_{\text{deglex}}}
\def\ordrecolrev{\preccurlyeq_{\text{rev}}}
\def\ordreweight{\prec_{\text{weight}}}
\def\ordrenweight{\preccurlyeq_{\text{n-weight}}}
\def\ordernweight{\prec_{\text{n-weight}}}

\DeclareSymbolFont{greekletters}{OML}{cmr}{m}{it}
\DeclareMathSymbol{\varrho}{\mathalpha}{greekletters}{"25}
\DeclareMathSymbol{\varsigma}{\mathalpha}{greekletters}{"26}

\newcount\cols
{\catcode`,=\active\catcode`|=\active
 \gdef\Young#1{\hbox{$\vcenter
 {\mathcode`,="8000\mathcode`|="8000
  \def,{\global\advance\cols by 1 &}%
  \def|{\cr
        \multispan{\the\cols}\hrulefill\cr
        &\global\cols=2 }%
  \offinterlineskip\everycr{}\tabskip=0pt
  \dimen0=\ht\strutbox \advance\dimen0 by \dp\strutbox
  \halign
   {\vrule height \ht\strutbox depth \dp\strutbox##
    &&\hbox to \dimen0{\hss$##$\hss}\vrule\cr
    \noalign{\hrule}&\global\cols=2 #1\crcr
    \multispan{\the\cols}\hrulefill\cr%
   }
 }$}}
\gdef\Skew(#1:#2){\hbox{$\vcenter
{\mathcode`,="8000\mathcode`|="8000
  \dimen0=\ht\strutbox \advance\dimen0 by \dp\strutbox
  \def\boxbeg{\vbox
    \bgroup\hrule\kern-0.4pt\hbox to\dimen0\bgroup\strut\vrule\hss$}%
  \def\boxend{$\hss\egroup\hrule\egroup}%
  \def,{\boxend\boxbeg}%
  \def|##1:{\boxend\vrule\egroup\nointerlineskip\kern-0.4pt
    \moveright##1\dimen0\hbox\bgroup\boxbeg}%
  \def\\##1\\##2:{\boxend\vrule\egroup\nointerlineskip\kern-0.4pt
    \kern ##1\dimen0\moveright##2\dimen0\hbox\bgroup\boxbeg}%
  \moveright#1\dimen0\hbox\bgroup\boxbeg#2\boxend\vrule\egroup
 }$}}
}

\definecolor{vert}{rgb}{0,0.45,0}

\begin{document}
\thispagestyle{empty}

\begin{center}

\begin{Large}
\begin{uppercase}
{Knuth's coherent presentations of}
\end{uppercase}

\vskip+5pt

\begin{uppercase}
{plactic monoids of type A}
\end{uppercase}
\end{Large}

\bigskip\hrule height 1.5pt \bigskip

\begin{large}\begin{uppercase}
{Nohra Hage \qquad Philippe Malbos}
\end{uppercase}\end{large}

\vspace*{\stretch{2}}

\begin{small}\begin{minipage}{14cm}
\noindent\textbf{Abstract --}
We construct finite coherent presentations of plactic monoids of type~A.
Such coherent presentations express a system of generators and relations for the monoid extended in a coherent way to give a family of generators of the relations amongst the relations. Such extended presentations are used for representations of monoids, in particular, it is a way to describe actions of monoids on categories. Moreover, a coherent presentation provides the first step in the computation of a categorical cofibrant replacement of a monoid.
Our construction is based on a rewriting method introduced by Squier that computes a coherent presentation from a convergent one. We compute a finite coherent presentation of a plactic monoid from its column presentation that is known to be finite and convergent. Finally, we show how to reduce this coherent presentation to a Tietze equivalent one having Knuth's generators.

\smallskip\noindent\textbf{M.S.C. 2010 -- 20M05, 18D05, 68Q42, 05E10.} 
\end{minipage}\end{small}

\vspace*{\stretch{1}}

\rule{0.35\textwidth}{0.5pt}

\vspace*{\stretch{2}}

\begin{small}\begin{minipage}{12cm}
\renewcommand{\contentsname}{}
\setcounter{tocdepth}{2}
\tableofcontents
\end{minipage}
\end{small}
\end{center}

\vspace*{\stretch{4}}

\hfill \begin{small}---\;\;\today\;\;---\end{small}

\vspace*{\stretch{1}}

\begin{minipage}{15cm}
\noindent \rule{0.3\textwidth}{0.5pt}

\begin{footnotesize}
\qquad This work is partially supported by the French National Research Agency, ANR-13-BS02-0005-02.
\end{footnotesize}
\end{minipage}

\newpage

\section{Introduction}

\subsection*{Coherent presentations of plactic monoids}

\subsubsection*{Plactic monoids}

The structure of plactic monoids appeared in the combinatorial study of Young tableaux by Schensted \cite{Schensted61} and Knuth \cite{Knuth70}. The \emph{plactic monoid} of rank $n>0$ is the monoid, denoted by $\P_n$, generated by the finite set $\{1,\ldots,n\}$ and subject to the \emph{Knuth relations}:
\begin{align*}
zxy &= xzy \quad \text{for all $1\leq x\leq y < z \leq n$},\\
yzx &= yxz \quad \text{for all $1\leq x<y\leq z\leq n$}.
\end{align*}
For instance, the monoid $\P_2$ is generated by $1$ and $2$ and submitted to the relations~$211 = 121$ and $221 = 212$.
The Knuth presentation of the monoid $\P_3$ has $3$ generators and $8$ relations.
Lascoux and Sch\"{u}tzenberger used the plactic monoid in order to prove the Littlewood-Richardson rule for the decomposition of  tensor products of irreducible modules over the Lie algebra of~$n$ by~$n$ matrices,~\cite{Schutzenberger77, LascouxSchutzenberger81}. The structure of plactic monoids has several applications in algebraic combinatorics and representation theory~\cite{LascouxSchutsenberger78,LascouxSchutzenberger81,LascouxLeclercThibon95,Fulton97} and several works have generalised the notion of tableaux to classical Lie algebras~{\cite{Berelle86,Sundaram90,KashiwaraNakashima94,Littelmann96, Sheats99}}.

\subsubsection*{Syzygies of Knuth's relations}

The aim of this work is to give an algorithmic method for the syzygy problem of finding all independent irreducible algebraic relations amongst the Knuth relations and some other presentations of the plactic monoids in type~$A$.
A \emph{$2$-syzygy} for a presentation of a monoid is a relation amongst relations. For instance, using the Knuth relations there are two ways to prove the equality~$2211 = 2121$ in the monoid $\P_2$, either by applying the first Knuth relation $211=121$ or the second relation $221=212$. This two equalities are related by a syzygy. 
Starting with a monoid presentation, we would like to compute all syzygies for this presentation and in particular to compute a family of generators for the syzygies.
For instance, we will prove that in rank $2$ the two Knuth relations form a unique generating syzzygy for the Knuth relations. 
For rank greater than $3$, the syzygies problem for the Knuth presentation is difficult due to the combinatorial complexity of the relations.
In commutative algebra, the theory of Gr\"obner bases gives algorithms to compute bases for linear syzygies. By a similar method, the syzygy problem for presentation of monoids can be algorithmically solved using \emph{convergent rewriting systems}.

\subsubsection*{Rewriting and plactic monoids}

Study presentations from a rewriting approach consists in the orientation of the relations, then called \emph{reduction rules}. For instance, the relations of the monoid $\P_2$ can be oriented with respect to the lexicographic order as follows
\[
\eta_{1,1,2}:211 \dfl 121
\qquad
\epsilon_{1,2,2} : 221 \dfl 212.
\]
In a monoid presented by a rewriting system, two words are equal if they are related by a zig-zag sequence of applications of reductions rules.
A rewriting system is convergent if the reduction relation induced by the rules is well-founded and it satisfies the \emph{confluence property}. This means that any reductions starting on a same word can be extended to end on a same reduced word. Recently plactic monoids were investigated by rewriting methods~\cite{KubatOkninski14, BokutChenChenLi15, CainGrayMalheiro15, Hage14, CainGrayMalheiro14}.

\subsubsection*{Coherent presentations}

In this paper, we give a categorical description of $2$-syzygies of presentations of the monoid~$\P_n$ using \emph{coherent presentations}. Such a presentation extends the notion of a presentation of the monoid by globular homotopy generators taking into account the relations amongst the relations of the monoid.
We compute a coherent presentation of the monoid $\P_n$ using the homotopical completion procedure introduced in~\cite{GuiraudMalbosMimram13, GaussentGuiraudMalbos14}. Such a procedure extends the Knuth-Bendix completion procedure, \cite{KnuthBendix70}, by keeping track of homotopy generators created when adding rules during the completion. Its correctness is based on the Squier theorem,~\cite{Squier94}, which states that a convergent presentation of a monoid extended by the homotopy generators defined by the confluence diagrams induced by \emph{critical branchings} forms a coherent convergent presentation. The notion of critical branching describes the overlapping of two rules on a same word.
For instance, the Knuth presentation of the monoid $\P_2$ is convergent. It can be extended into a coherent presentation with a unique globular homotopy generator described by the following $3$-cell corresponding to the unique critical branching of the presentation between the rules $\eta_{1,1,2}$ and $\epsilon_{1,2,2}$:
\[
\xymatrix@!C@C=3em{
2211
  \ar@2@/^3ex/ [r] ^{2\eta_{1,1,2}}="src"
  \ar@2@/_3ex/ [r] _{\epsilon_{1,2,2}1}="tgt"
& 2121
  \ar@3 "src"!<0pt,-15pt>;"tgt"!<0pt,15pt>
}
\]
The Knuth presentation of the monoid $\P_3$ is not convergent, but it can be completed by adding $3$ relations to get a presentation with $27$ $3$-cells corresponding to the $27$ critical branchings. For the monoid~$\P_4$ we have $4$ $1$-cells and $20$ $2$-cells, for $\P_5$ we have $5$ $1$-cells and $40$ $2$-cells and for $\P_6$ we have $6$ $1$-cells and $70$ $2$-cells. However, in the last three cases, the completion is infinite and another approach is necessary to compute a finite generating family for syzygies of the Knuth presentation. 

\subsubsection*{The column presentation}

Kubat and Okni\'nski showed in \cite{KubatOkninski14} that for rank $n>3$, a finite convergent presentation of the monoid~$\P_n$ cannot be obtained by completion of the Knuth presentation with the deglex order. Then Bokut, Chen, Chen and Li in~\cite{BokutChenChenLi15} and Cain, Gray and Malheiro in \cite{CainGrayMalheiro15} constructed with independent methods a finite convergent presentation by adding column generators to the Knuth presentation.
The monoid $\P_n$ corresponds to the representations of the general Lie algebra of~$n$ by~$n$ matrices which is of type~$A$, and now called the plactic monoid of type~$A$,~\cite{DateJimboMiwa90,LascouxLeclercThibon95}. The classification of finite dimensional complex semisimple Lie algebras in classical types~$A$,~$B$,~$C$,~$D$ and in exceptional ones allows the existence of plactic monoids of the same types. Theses monoids can be defined by a case-by-case analysis using the Kashiwara theory of crystal bases~\cite{Kashiwara91,
KashiwaraNakashima94,Kashiwara94,Baker00,Lecouvey02,Lecouvey03} or in a general way using Littelmann path model~\cite{Littelmann96}.
Using the Kashiwara theory of crystal bases, the first author constructed in~\cite{Hage14} a finite and convergent presentation for plactic monoids of type~$C$. Similar presentations for plactic monoids of type~$B$,~$C$,~$D$ and~$G_2$ were obtained by Cain, Gray and Malheiro in~\cite{CainGrayMalheiro14}. Recently, finite convergent presentations of plactic monoids for any type was also obtained by the first author using the Littelmann path model,~\cite{Hage15}.
However, on the one hand, the proof given in~\cite{CainGrayMalheiro15} does not give explicitly the critical branchings of the presentation which does not permit to use the homotopical completion procedure. On the other hand, the construction in~\cite{BokutChenChenLi15} gave an explicit description of the critical branchings of the presentation, but this does not allow to get explicitly the relations amongst the relations, and in particular it is difficult to reduce such a presentation.

\subsubsection*{The Knuth coherent presentation}

We construct a coherent presentation of the monoid $\P_n$ that extends the Knuth presentation in two steps. The first step consists in giving an explicit description of the critical branchings of  the column presentation.
The column presentation of the plactic monoid has one generator~$c_u$ for each column $u$, that is, a word $u=x_p\ldots x_1$ such that $x_p>\ldots >x_1$.
Given two columns~$u$ and~$v$, using the Schensted algorithm, we compute the Schensted tableau $P(uv)$ associated to the word $uv$. One proves that the planar representation of the tableau $P(uv)$ contains at most two columns. If the planar representation is not the tableau obtained as the concatenation of the two columns~$u$ and~$v$, one defines a rule $\alpha_{u,v} : c_uc_v \dfl c_wc_{w'}$ where $w$ and $w'$ are respectively the left and right columns (with one of them possibly empty).
We show that the column presentation can be extended into a \emph{coherent column presentation} whose any $3$-cell has at most an hexagonal form.
For instance, the column presentation for the monoid~$\P_2$ has generators $c_1$, $c_2$, $c_{21}$, with the rules $\alpha_{2,1} : c_2c_1\dfl c_{21}$, $\alpha_{1,21} : c_1c_{21} \dfl c_{21}c_1$ and~$\alpha_{2,21} : c_2c_{21} \dfl c_{21}c_2$.
This presentation has only one critical branching:
\[
\xymatrix @C=3em @R=0.6em {
& c_{21}c_{21}
     \ar@3 []!<-14pt,-10pt>;[dd]!<-14pt,10pt>
\\
c_2c_1c_{21}
    \ar@2@/^/ [ur] ^-{\alpha_{2,1}c_{21}}
    \ar@2@/_/ [dr] _-{c_2\alpha_{1,21}}
\\
& c_2c_{21}c_1
    \ar@2 [r] _-{\alpha_{2,21}c_1}
& c_{21}c_2c_1
    \ar@2 [luu] _-{c_{21}\alpha_{2,1}}
}
\]
and thus the $3$-cell of the extended coherent presentation is reduced to this $3$-cell defined by this confluence diagram.
Note that for column presentations of the monoids~$\P_3$, $\P_4$ and $\P_5$ we count respectively $7$, $15$ and $31$ generators, $22$, $115$ and $531$ relations, $42$, $621$ and $6893$ $3$-cells.

The second step aimed at to reduce the coherent column presentation using Tietze transformations that coherently eliminates redundant column generators and defining relations to the \emph{Knuth coherent presentation} giving syzygies of the  Knuth presentation. For instance, if we apply this Tietze transformation on the column coherent presentation of the monoid $\P_2$, we prove that the Knuth coherent presentation of~$\P_2$ on the generators $c_1,c_2$ and the relations $\eta_{1,1,2}$, $\epsilon_{1,2,2}$ has a unique generating $3$-cell $2\eta_{1,1,2} \tfl \epsilon_{1,2,2}1$ described above.

\subsection*{Organisation and main results of the article}

\subsubsection*{Two-dimensional rewriting}

In this work, we use the polygraphical description of string rewriting systems. The polygraphic notions are briefly recalled in Section \ref{Subsection:PresentationBy2Polygraph} and we refer the reader to \cite{GuiraudMalbos14} for a deeper presentation.
A \emph{$2$-polygraph} is a data made of a directed graph $(\Sigma_0,\Sigma_1)$ and a globular extension~$\Sigma_2$ of the free monoid $\Sigma_1^\ast$ on $\Sigma_1$.
The \emph{monoid presented by $\Sigma$} is the quotient of the free monoid $\Sigma_1^\ast$ by the congruence generated by the $2$-cells of the free $2$-category $\Sigma_2^\ast$.
A \emph{rewriting step} of a $2$-polygraph~$\Sigma$ is a $2$-cell in the $2$-category~$\Sigma_2^*$ generated by $\Sigma$ and with shape
\[
\xymatrix@C=3.5em{
\bullet
	\ar [r] ^-{w} 
& \bullet
	\ar@/^3ex/ [r] ^-{u} ^{}="src"
	\ar@/_3ex/ [r] _-{v} ^{}="tgt"
	\ar@2 "src"!<0pt,-10pt>;"tgt"!<0pt,10pt> ^-{\beta}
& \bullet
	\ar [r] ^-{w'}
& \bullet
}
\]
where $\beta$ is a $2$-cell of~$\Sigma_2$ and~$w$ and~$w'$ are $1$-cells of~$\Sigma_1^*$. 
A \emph{rewriting sequence} is a finite or infinite sequence 
of rewriting steps.
The $2$-polygraph~$\Sigma$ \emph{terminates} if it has no infinite rewriting sequence.
A \emph{branching} of the $2$-polygraph~$\Sigma$ is a non ordered pair~$(f,g)$ of $2$-cells of~$\Sigma_2^*$ with a common source as in 
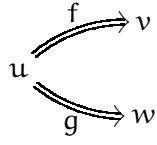
\[
\xymatrix @R=0.6em @C=3em {
& {v}
\\
{u}
	\ar@2@/^/ [ur] ^-{f}
	\ar@2@/_/ [dr] _-{g}
\\
& {w}
}
\]
It is \emph{local} if~$f$ and~$g$ are rewriting steps, \emph{aspherical} if $f=g$ and \emph{Peiffer} when it is of the form $(hu_2,u_1k)$ for rewriting steps $h$ and $k$ with $s_1(h)=u_1$ and $s_1(k)=u_2$. The \emph{overlapping} branchings are the remaining local branchings.
A minimal overlapping local branching is a \emph{critical branching}.
A $2$-polygraph~$\Sigma$ is \emph{confluent} if for all branching $(f,g)$ there exist $2$-cells~$f'$ and~$g'$ in~$\Sigma_2^*$ as in the following diagram: 
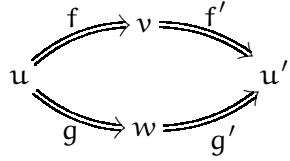
\[
\xymatrix @R=0.6em @C=3em{
& {v}
	\ar@2@/^/ [dr] ^-{f'}
\\
{u}
	\ar@2@/^/ [ur] ^-{f}
	\ar@2@/_/ [dr] _-{g}
&& {u'}
\\
& {w}
	\ar@2@/_/ [ur] _-{g'}
}
\]
A $2$-polygraph~$\Sigma$ is \emph{convergent} if it terminates and it is confluent. 

\subsubsection*{Plactic monoids}

In Section~\ref{placticmonoids}, we recall the definition and properties of plactic monoids. We refer the reader to~\cite{Lothaire02} and~\cite{Fulton97} for a full introduction.
The \emph{Knuth $2$-polygraph of rank $n>0$} is the 
$2$-polygraph~$\Knuth_2(n)$ whose set of $1$-cells is $\{1,\ldots,n\}$ and the set of $2$-cells is 
\[
\{
zxy \odfl{\eta_{x,y,z}} xzy \;|\; 1\leq x\leq y < z \leq n
\}
\,\cup\,
\{
yzx \odfl{\epsilon_{x,y,z}} yxz \;|\; 1\leq x<y\leq z\leq n
\}.
\]
The $2$-cells of $\Knuth_2(n)$ correspond to the Knuth relations oriented with respect to the lexicographic order and the monoid presented by the $2$-polygraph $\Knuth_2(n)$ is the monoid $\P_n$,~{\cite[Theorem 6]{Knuth70}}.

\subsubsection*{Pre-column presentation}

In~\ref{Section:Pre-columnPresentation}, we introduce the pre-column presentation. Consider the set $\col(n)$ of non-empty columns on the set $\{1,\ldots,n\}$.
One adds to the presentation~$\Knuth_{2}(n)$ one superfluous generator $c_{u}$ for any~$u$ in $\col(n)$. We denote by
$\Colo_1(n)$ the set of \emph{column generators} $c_{u}$ for any  $u$ in~$\col(n)$ and by 
\[
\gamma_u : c_{x_{p}}\ldots c_{x_{1}} \dfll  c_u
\]
the defining relation for the column generators $u=x_{p} \ldots x_{1}$ in $\col(n)$ of length greater than $2$. 
In the free monoid $\Colo_1(n)^\ast$, the  Knuth relations can  be written  in the following form
\[
c_{z}c_{x}c_{y} \odfll{\eta_{x,y,z}^c} c_{x}c_{z}c_{y}
\;\text{for}\;
1 \leq x\leq y<z  \leq n,
\quad\text{and}\quad
c_{y}c_{z}c_{x}\odfll{\epsilon_{x,y,z}^c}c_{y}c_{x}c_{z}
\;\text{for}\;
1 \leq x< y\leq z \leq n.
\]
The $2$-polygraph $\Knuthcc_2(n)$ whose $1$-cells are columns and $2$-cells are the defining relations for columns generators and  the Knuth relations $\eta_{x,y,z}^c$ and $\epsilon_{x,y,z}^c$ is a presentation of the monoid $\P_n$.
In \ref{Subsubsection:Pre-columnPresentation}, we give an other presentation of the column generators. One defines the $2$-polygraph~$\PreColo_2(n)$ with column generators and the set of $2$-cells is 
\begin{eqnarray*}
\big\{
c_{x}c_{zy}\odfll{\alpha'_{x,zy}}c_{zx}c_{y}
\;|\; 1 \leq x\leq y<z \leq n
\big\}
\,\cup\,  \hspace{7cm}\\
\hspace{0.5cm}
\big\{c_{y}c_{zx}\odfll{\alpha'_{y,zx}}c_{yx}c_{z}
\;|\; 1 \leq x< y\leq z \leq n
\big\}
\,\cup\,
\big\{
c_{x}c_{u} \odfl{\alpha'_{x,u}} c_{xu}
\;|\; xu\in \col(n) \;\;\text{and}\;\; 1\leq x \leq n
\big\},
\end{eqnarray*}
where the $2$-cells $\alpha'_{x,zy}$ and $\alpha'_{y,zx}$ correspond respectively to the Knuth relations $\eta_{x,y,z}^c$ and~$\epsilon_{x,y,z}^c$.
We prove in Proposition~\ref{Proposition:PresentationPreColo2} that the $2$-polygraph $\PreColo_2(n)$ is a presentation of the monoid $\P_n$, then called the \emph{pre-column presentation of~$\P_n$}.

\subsubsection*{Column presentation}

In \ref{Section:ColumnPresentation}, we recall the column presentation introduced in~\cite{CainGrayMalheiro15}.
Given columns $u$ and $v$, if the planar representation of the Schensted tableau $P(uv)$ is not the tableau obtained as the concatenation of the two columns $u$ and $v$, we will denote $\typedeux{0}$. In this case, the tableau $P(uv)$ contains at most two columns and we will denote~$\typedeux{01}$ if the tableau $P(uv)$ has one column and we will denote~$\typedeux{02}$ if the tableau~$P(uv)$ has two columns.
When $\typedeux{0}$, we define a $2$-cell
\[
\alpha_{u,v} :  c_uc_v \dfl c_{w}c_{w'}
\]
where $w=uv$ and $c_{w'}=1$, if $\typedeux{01}$, and $w$ and $w'$ are respectively the left and right columns of the tableau $P(uv)$, if $\typedeux{02}$.
The $2$-polygraph~$\Colo_2(n)$ whose set of $1$-cells is $\Colo_1(n)$ and the $2$-cells are the~$\alpha_{u,v}$  is a finite convergent presentation of the monoid $\P_n$, called the \emph{column presentation} of the monoid~$\P_n$.
The proof given in~\cite{CainGrayMalheiro15} for the convergence of $\Colo_2(n)$ consists in showing that the $2$-polygraph $\Colo_2(n)$ has the unique normal form property.  
The construction in Section~\ref{SubSection:CoherentColumnPresentation} gives an other proof of the confluence of the $2$-polygraph~$\Colo_2(n)$ by showing the confluence of all the critical branchings of the column presentation. 

\subsubsection*{Coherent column presentation}

In Section \ref{Section:CoherentColumnPresentation}, we recall the notion of coherent presentation of a monoid.
A \emph{$(3,1)$-polygraph} is a pair $(\Sigma_2,\Sigma_3)$ made of a $2$-polygraph $\Sigma_2$ and a globular extension $\Sigma_3$ of the free $(2,1)$-category $\tck{\Sigma}_2$.
A \emph{coherent presentation} of a monoid $\M$ is a $(3,1)$-polygraph whose underlying $2$-polygraph is a presentation of the monoid $\M$ and such that, for every $2$-sphere $\gamma$ of $\tck{\Sigma}_2$, there exists a $3$-cell in $\tck{\Sigma}_3$ with boundary $\gamma$.
Using the homotopical completion procedure from \cite{GaussentGuiraudMalbos14}, we extend the $2$-polygraph $\Colo_2(n)$ into a coherent presentation~$\Colo_3(n)$ of the monoid $\P_n$. In particular, we explicit  all the $3$-cells $\mathcal{X}_{u,v,t}$ given by the confluence diagrams of the critical branchings and having the following hexagonal form
\[
\xymatrix @!C @C=2.3em @R=0.6em {
&
c_{e} c_{e'}c_t
   \ar@2[r] ^{c_{e}\alpha_{e',t}}
     \ar@3 []!<30pt,-8pt>;[dd]!<30pt,8pt> ^{\mathcal{X}_{u,v,t}}
&
c_{e} c_{b}c_{b'}
  \ar@2[dr] ^{\alpha_{e,b}c_{b'}}
\\
c_uc_vc_t
  \ar@2[ur] ^{\alpha_{u,v}c_t}
  \ar@2[dr] _{c_u\alpha_{v,t}}
&&&
c_{a}c_{d} c_{b'}
\\
&
c_uc_{w}c_{w'}
  \ar@2[r] _{\alpha_{u,w}c_{w'}}
&
c_{a}c_{a'}c_{w'}
 \ar@2[ur] _{c_{a}\alpha_{a',w'}}
}
\]
for any columns $u$, $v$ and $t$ such that \typetroisbasep{$u$}{$v$}{$t$}{00}. This shows the first main result of this article:

\begin{quote}
{\bf Theorem \ref{MainTheoremA}.}
\emph{
For $n>0$, the $(3,1)$-polygraph $\Colo_3(n)$ is a coherent presentation of the monoid $\P_n$.
}
\end{quote}

The extended presentation $\Colo_3(n)$ is called the \emph{column coherent presentation} of the monoid~$\P_n$.

\subsubsection*{Pre-column coherent presentation}

In Section \ref{Section:ReductionCoherentPresentation}, using the homotopical reduction procedure given in~{\cite[2.3.3]{GaussentGuiraudMalbos14}, we reduce the coherent presentation~$\Colo_3(n)$ into a smaller coherent presentation of the monoid~$\P_n$. We proceed in three steps. Firstly, we apply a homotopical reduction  on the $(3,1)$-polygraph $\Colo_3(n)$ with a collapsible part defined by some of the generating triple confluences of the $2$-polygraph $\Colo_2(n)$. In this way, we reduce the coherent presentation~$\Colo_3(n)$ of the monoid~$\P_n$ into the coherent presentation~$\overline{\Colo}_3(n)$ of~$\P_n$, whose underlying $2$-polygraph is $\Colo_2(n)$ and the $3$-cells~$\mathcal{X}_{u,v,t}$ are those of $\Colo_3(n)$, but with~$u$ is of length~$1$.
Then we reduce the coherent presentation~$\overline{\Colo}_3(n)$ into a coherent presentation $\PreColo_3(n)$ obtained from $\PreColo_2(n)$ by adjunction of  the $3$-cell $R_{\Gamma_3}(C'_{x,v,t})$ where
\[
\xymatrix @C=3em @R=0.6em {
& {c_{xv}c_{t}}
\ar@3 []!<0pt,-10pt>;[dd]!<0pt,+10pt> ^{C'_{x,v,t}}
\\
{c_{x}c_{v}c_{t}}
	\ar@2@/^/ [ur] ^-{\alpha_{x,v}c_{t}}
	\ar@2@/_/ [dr] _-{c_{x}\alpha_{v,t}}
\\
& {c_{x}c_{w}c_{w'}}
	\ar@2 [r] _-{\alpha_{x,w}c_{w'}}
& {c_{xv}c_{z_{l}\ldots z_{q+1}}c_{w'}}
	\ar@2 [luu] _-{c_{xv}\alpha_{z_{l}\ldots z_{q+1}, w'}}
}
\]
with \typetroisbasep{$x$}{$v$}{$t$}{12}, and the $3$-cell $R_{\Gamma_3}(D_{x,v,t})$ where
\[
\xymatrix @!C @C=2.3em @R=0.6em {
&
c_{e} c_{e'}c_t
  \ar@2[r] ^{c_{e}\alpha_{e',t}}
     \ar@3 []!<40pt,-10pt>;[dd]!<40pt,10pt> ^{D_{x,v,t}} 
&
c_{e} c_{b}c_{b'}
  \ar@2[dr] ^{\alpha_{e,b}c_{b'}}
\\
c_xc_vc_t
  \ar@2[ur] ^{\alpha_{x,v}c_t}
  \ar@2[dr] _{c_x\alpha_{v,t}}
&&&
c_{a}c_{d} c_{b'}
\\
&
c_xc_{w}c_{w'}
  \ar@2[r] _{\alpha_{x,w}c_{w'}}
&
c_{a}c_{a'}c_{w'}
 \ar@2[ur] _{c_{a}\alpha_{a',w'}}
}
\]
with \typetroisbasep{$x$}{$v$}{$t$}{22} and where the homotopical reduction $R_{\Gamma_3}$ eliminates a collapsible part $\Gamma_3$ of $\overline{\Colo}_3(n)$. In this way, we prove that

\begin{quote}
{\bf Theorem \ref{Theorem:PreColo3Coherent}.}
\emph{
For $n > 0$, the $(3,1)$-polygraph $\PreColo_3(n)$ is a coherent presentation of the monoid~$\P_n$.
}
\end{quote}

\noindent For instance, the coherent presentation $\Colo_3(2)$ has only one $3$-cell
\[
\xymatrix @C=3em @R=0.6em {
& {c_{21}c_{21}}
\ar@3 []!<-10pt,-10pt>;[dd]!<-10pt,+10pt> ^{C'_{2,1,21}}
\\
{c_2c_1c_{21}}
	\ar@2@/^/ [ur] ^-{\alpha_{2,1}c_{21}}
	\ar@2@/_/ [dr] _-{c_{2}\alpha_{1,21}}
\\
& {c_2c_{21}c_1}
	\ar@2 [r] _-{\alpha_{2,21}c_1}
& {c_{21}c_2c_1}
	\ar@2 [luu] _-{c_{21}\alpha_{2,1}}
}
\] 
In this case, the $(3,1)$-polygraphs $\PreColo_3(2)$ and $\Colo_3(2)$ coincide.
We give in~\ref{SubSection:Computations} the values of number of cells of the $(3,1)$-polygraphs~$\overline{\Colo}_3(n)$ and~$\PreColo_3(n)$  for plactic monoids of rank~$n\leq 10$. 

\subsubsection*{Knuth's coherent presentation}

In a final step, we reduce in \ref{Section:KnuthCoherentPresentation} the coherent presentation $\PreColo_3(n)$ into a coherent presentation of the monoid $\P_n$ whose underlying $2$-polygraph is $\Knuth_2(n)$. We define an extended presentation $\Knuth_{3}(n)$ of the monoid $\P_{n}$ obtained from $\Knuth_{2}(n)$ by adjunction of  the following set of $3$-cells
\[
\{\,
\mathcal{R}(C'_{x,v,t})\;\;\big|\;\;\typetroisbasep{$x$}{$v$}{$t$}{12}
\,\}
\;\cup\;
\{\,
\mathcal{R}(D_{x,v,t})\;\;\big|\;\;\typetroisbasep{$x$}{$v$}{$t$}{22}
\,\},
\]
where $\mathcal{R} :\tck{\overline{\Colo}_3(n)}\fl\tck{\Knuthcc_3(n)}$ is a Tietze transformation constructed throughout Section~\ref{Section:ReductionCoherentPresentation}.
We obtain our main result:

\begin{quote}
{\bf Theorem \ref{KnuthcoherentTheorem}.}
\emph{
For $n > 0$, the $(3,1)$-polygraph $\Knuth_3(n)$ is a coherent presentation of the monoid~$\P_n$.
}
\end{quote}

For instance, the Knuth coherent presentation of the monoid \pdf{\P_2} has generators $c_1$ and $c_2$ subject to the Knuth relations $\eta_{1,1,2}^c : c_2c_1c_1 \dfl c_1c_2c_1$ and $\epsilon_{1,2,2}^c : c_2c_2c_1 \dfl c_2c_1c_2$ and the following~$3$-cell
\[
\xymatrix@!C@C=3em{
c_2c_2c_1c_1
  \ar@2@/^3ex/ [r] ^{2\eta_{1,1,2}^c}="src"
  \ar@2@/_3ex/ [r] _{\epsilon_{1,2,2}^{c}1}="tgt"
& c_2c_1c_2c_1
    \ar@3 "src"!<-5pt,-15pt>;"tgt"!<-5pt,15pt>  ^-{\, C''}
}
\]
Note that the Knuth coherent presentation of the monoid $\P_2$ corresponds to the coherent presentation that one can compute directly using the fact that the $2$-polygraph $\Knuth_2(2)$ is convergent.

\subsubsection*{Coherence and Lakshmibai-Seshadri's paths}

The plactic monoid admits a description in terms of the Kashiwara theory of crystal bases,~\cite{DateJimboMiwa90, Kashiwara91, KashiwaraNakashima94, Kashiwara94}, and the Littelmann path model,~\cite{Littelmann96}. In a last part of the paper, we compute a coherent presentation of the monoid~$\P_n$ using these two approaches. 
In Section~\ref{CoherenceLakshmibaiSeshadriPaths}, we construct a convergent presentation of the monoid~$\P_n$ using the notions of tableaux and Yamanouchi paths. In this case, the plactic congruence~$\sim_{\textrm{path(n)}}$ is defined in terms of a crystal isomorphism,~see~\ref{Equivalenceonpaths}. 
We recall in~\ref{Paths and crystal graphs}, the notion of paths from~\cite{Littelmann96}. Consider~$\mathbb{R}^{n}$ with its canonical basis~$(\epsilon_1,\ldots, \epsilon_n)$. A \emph{path} is a piecewise linear continuous map
$\pi: [0,1]\fll X\otimes_{\mathbb{Z}}\mathbb{R}$ where $X$ the lattice $\mathbb{Z}\epsilon_1\oplus\ldots\oplus\mathbb{Z}\epsilon_n$. Denote by $\pi_{\epsilon_i}$ the path that connects the origin with~$\epsilon_i$ by a straight line.

\subsubsection*{$2$-polygraph of crystals}

In~\ref{Subsection:TableauxLSpaths}, we recall  the notion of tableaux, Lakshmibai-Seshadri's paths and Yamanouchi's paths from~\cite{Littelmann94,Littelmann96,Lothaire02,Shimozono05}.
Using these notions, we define in~\ref{2polygraphofcrystals} the \emph{$2$-polygraph of crystals} as follows. Consider the $2$-polygraph~$\Crys^{0}_2(n)$ whose $1$-cells are $\pi_{\epsilon_1},\ldots \pi_{\epsilon_n}$ and whose $2$-cells are of the form~$\vartheta_{\pi_w}:\pi_{w} \dfl Y(\pi_{w})$, where~$\pi_w$ is a non-Yamanouchi path tableau and~$Y(\pi_{w})$ is its corresponding Yamanouchi path tableau such that~\mbox{$\pi_{w}(1) = Y(\pi_{w})(1)$}. 
For~$k\geq 0$, we define the $2$-polygraph~$\Crys^{k}_2(n)$ whose $1$-cells are $\pi_{\epsilon_1},\ldots \pi_{\epsilon_n}$ and whose $2$-cells are of the form $\vartheta_{\pi_w}^{\alpha_{j_{k}}}:f_{\alpha_{j_{k}}}( \pi_{w})\dfl 
f_{\alpha_{j_{k}}}(Y(\pi_{w}))$, where~$\pi_{w}$ and~$Y(\pi_{w})$ are respectively non-Yamanouchi and Yamanouchi paths tableaux such that~\mbox{$\pi_{w}(1) = Y(\pi_{w})(1)$} and where~$f_{\alpha_{j_{k}}}$ is a composite of roots operators~$f_{\alpha_i}$, defined in~\ref{Rootoperators}.
We define the \emph{$2$-polygraph of crystals}~$\Crys_2(n)$ as the union
$\underset{i\geq 0}{\cup}\Crys^{i}_2(n)$. We prove that

\begin{quote}
{\bf Theorem \ref{Crystaltheorem}.}
\emph{For $n>0$, the $2$-polygraph~$\Crys_2(n)$ is a convergent presentation of the monoid~$\P_n$.
}
\end{quote}

\subsubsection*{Path coherent presentation}

Finally, we end Section~\ref{CoherenceLakshmibaiSeshadriPaths} by constructing a coherent presentation of the monoid~$\P_n$ in terms of Lakshmibai-Seshadri's paths. We consider the $1$-polygraph~$\Path_1(n)$  with only one $0$-cell and whose $1$-cells are all Lakshmibai-Seshadri's paths.
For each pair $(\pi_u, \pi_{v})$ in $\Path_1(n)$  such that $\pi_u\star\pi_{v}$ is not a tableau, we define a $2$-cell~
$\alpha_{\pi_u,\pi_{v}} : \pi_u\star\pi_{v}\dfl \pi_w\star\pi_{w'}$, where $\pi_w\star\pi_{w'}$ is the unique tableau such  that $\pi_u\star\pi_{v}\sim_{\textrm{path(n)}} \pi_w\star\pi_{w'}$.
The \emph{$2$-polygraph of paths}, denoted by $\Path_2(n)$, is the  $1$-polygraph $\Path_1(n)$ extended by the set of $2$-cells  $\alpha_{\pi_u,\pi_{v}}$, where  $\pi_u$ and $\pi_{v}$ are in $\Path_1(n)$  such that~$\pi_u\star\pi_{v}$ is not a tableau.
Then we consider  the extended presentation~$\Path_3(n)$  of the monoid~$\P_n$ obtained from~$\Path_2(n)$ by adjunction of the following $3$-cell 
\[
\xymatrix @!C @C=2.3em @R=0.6em {
&
\pi_{e}\star \pi_{e'}\star\pi_t
  \ar@2[r] ^{\pi_{e}\alpha_{\pi_{e'},\pi_t}}
     \ar@3 []!<45pt,-8pt>;[dd]!<45pt,8pt> ^{}
&
\pi_{e}\star \pi_{b}\star \pi_{b'}
  \ar@2[dr] ^{\alpha_{\pi_e,\pi_b}\pi_{b'}}
\\
\pi_u\star \pi_v\star \pi_t
  \ar@2[ur] ^{\alpha_{\pi_u,\pi_v}\pi_t}
  \ar@2[dr] _{\pi_u\alpha_{\pi_v,\pi_t}}
&&&
\pi_{a}\star\pi_{d}\star \pi_{b'}
\\
&
\pi_u\star\pi_{w}\star\pi_{w'}
  \ar@2[r] _{\alpha_{\pi_u,\pi_w}\pi_{w'}}
&
\pi_{a}\star \pi_{a'}\star\pi_{w'}
 \ar@2[ur] _{\pi_{a}\alpha_{\pi_{a'},\pi_{w'}}}
}	
\]
where the paths $\pi_u\star\pi_v$ and~$\pi_v\star\pi_t$ are not tableaux. As a consequence of Theorem~\ref{MainTheoremA}, we deduce that the $(3,1)$-polygraph $\Path_3(n)$ is a coherent presentation of the monoid~$\P_n$.

\subsubsection*{Applications and perspectives}

In~\cite{GaussentGuiraudMalbos14}, a description of the category of actions of a monoid on categories is given in terms of coherent presentations.
Using this description, Theorem~\ref{KnuthcoherentTheorem} allows to present actions of plactic monoids on categories as follows. The category $\mathrm{Act}(\P_n)$ of actions of the monoid~$\P_n$ on categories is equivalent to the category of $2$-functors from the $(2,1)$-category $\tck{\Knuth_2(n)}$ to the category~$\Cat$ of categories, that sends the $3$-cells of $\Knuth_3(n)$ to commutative diagrams in~$\Cat$. 
One potential application is the actions of the plactic monoid in the category of finite dimensional representations of the general Lie algebra or in the category~$\mathcal{O}$ of finite and infinite dimensional representations of the general Lie algebra,~\cite{Humphreys08}.

Following~\cite{GuiraudMalbos12advances}, a convergent presentation of a monoid can be extended into a polygraphic resolution of the monoid, that is, a cofibrant replacement of the monoid in the category of $(\infty,1)$-categories.
The column presentation $\Colo_2(n)$ of the monoid~$\P_n$ can then be extended into a polygraphic resolution whose $n$-cells, for every $n\geq 3$, are indexed by $(n-1)$-fold branching of $\Colo_2(n)$.
We can explicit the $4$-cells of this resolution, which correspond to the confluence diagrams induced by critical triple branchings with source $c_{u}c_{v}c_tc_e$ for all columns~$u$, $v$, $t$ and~$e$ such that~\typedeuxbase{$u$}{$v$}{0}, \typedeuxbase{$v$}{$t$}{0} and~\typedeuxbase{$t$}{$e$}{0}. These $4$-cells have a permutohedral form. More generally, one may conjecture that the generating $n$-cells of the resolution have the form of the permutohedron of order $n$ corresponding to a confluence diagram of $(n-1)$ overlapping reductions. This construction should generalise the construction of the Anick resolution for the monoid~$\P_n$ starting with the column presentation, given by Lopatkin in~\cite{Lopatkin16}.

Finally, by extending our construction to plactic monoids of other types, we expect other potential applications in representation theory. In particular, our approach for plactic monoids of type~$A$ could be applied for computations of finite coherent presentations for the plactic monoids of type~$B$,~$C$ and~$D$. The column rules for the type~$A$ are defined by the Schensted insertion algorithm. For the other types, the column rules are defined by Lecouvey's insertion algorithm,~\cite{Lecouvey02,Lecouvey03}, as was shown in~\cite{Hage14}. We expect that the syzygies for the classical types have an hexagonal form as shown for type~A. Finally, this question is more difficult for the exceptional types that we are not able to give a conjectural form.

\pagebreak

\section{Column presentation of plactic monoids}
\label{Section:ColumnPresentationofplacticmonoids}

In this work, rewriting methods are presented in the language of polygraphs, that we recall in this section. We refer the reader to \cite{GuiraudMalbos12advances} and \cite{GuiraudMalbos14} for a deeper presentation.

\subsection{Presentations of monoids by two-dimensional polygraphs}
\label{Subsection:PresentationBy2Polygraph}

\subsubsection{Two-dimensional polygraphs}

A \emph{$1$-polygraph} is a directed graph
\[
\xymatrix @C=4.2em @R=2.8em{
\Sigma_0
& \Sigma_1
	\ar@<.45ex> [l] ^{t_0}
	\ar@<-.45ex> [l] _{s_0}
}
\]
given by a set $\Sigma_0$ of \emph{$0$-cells}, a set $\Sigma_1$ of \emph{$1$-cells} together with two maps $s_0$ and $t_0$ sending a $1$-cell~$x$ on its \emph{source} $s_0(x)$ and its \emph{target} $t_0(x)$.
We will denote by $\Sigma_1^\ast$ the free category generated by the \mbox{$1$-polygraph}~$(\Sigma_0,\Sigma_1)$. Its set of $0$-cells is $\Sigma_0$ and for any $0$-cells $p$ and $q$, the $1$-cells of the hom-set~$\Sigma_1^\ast(p,q)$ are paths from $p$ to $q$ in the $1$-polygraph $(\Sigma_0,\Sigma_1)$. The composition is the concatenation of paths and the identity on a $0$-cell $p$ is the empty path with source and target $p$.   
A \emph{globular extension} of the free category $\Sigma_1^\ast$ is a set $\Sigma_2$ equipped with two maps 
\[
\xymatrix @C=4.2em @R=2.8em{
\Sigma_1^\ast
& \Sigma_2
	\ar@<.45ex> [l] ^{t_1}
	\ar@<-.45ex> [l] _{s_1}
}
\]
such that, for every $\beta$ in $\Sigma_2$, the pair $(s_1(\beta),t_1(\beta))$ is a \emph{$1$-sphere} in the category $\Sigma_1^\ast$, that is, 
\[
s_0s_1(\beta)=s_0t_1(\beta)
\quad\text{and}\quad
t_0s_1(\beta)=t_0t_1(\beta).
\] 
An element of the globular extension $\Sigma_2$ can be represented by a $2$-cell with the following globular shape
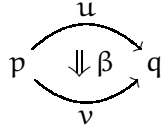
\[
\xymatrix @C=3.5em {
p
	\ar @/^3ex/ [r] ^-{u} _-{}="src"
	\ar @/_3ex/ [r] _-{v} ^-{}="tgt"
	\ar@2 "src"!<-2.5pt,-10pt>;"tgt"!<-2.5pt,10pt> ^-*+{\beta}
&q
}
\]
that relates parallel $1$-cells $u$ and $v$ of $\Sigma_1^\ast$.
A \emph{$2$-polygraph} $\Sigma$ is a triple $(\Sigma_0,\Sigma_1,\Sigma_2)$, where $(\Sigma_0,\Sigma_1)$ is a $1$-polygraph and $\Sigma_2$ is a globular extension of the free category $\Sigma_1^\ast$.
The elements of~$\Sigma_2$ are called the \emph{$2$-cells} of the $2$-polygraph~$\Sigma$, or the \emph{rewriting rules} defined by $\Sigma$.
If there is no possible confusion,~$\Sigma_2$ will denote the set of $2$-cells of the $2$-polygraph $\Sigma$ or the $2$-polygraph itself.
A \emph{$2$-category} is a category enriched in categories. When two $1$-cells, or $2$-cells, $f$ and~$g$ of a $2$-category are $i$-composable, for~\mbox{$i=0,1$}, that is~$t_i(f) = s_i(g)$, we denote by $f\star_i g$ their $i$-composite. A \emph{$(2,1)$-category} is a category enriched in groupoid, that is a $2$-category whose $2$-cells are invertible for the $1$-composition. We will denote by $\Sigma_2^\ast$ (resp. $\tck{\Sigma}_2$) the $2$-category (resp. $(2,1)$-category) freely generated by the $2$-polygraph $\Sigma$. We refer the reader to {\cite[Section 2.4.]{GuiraudMalbos14}} for expended definitions of $2$-categories and constructions of the $2$-categories~$\Sigma_2^\ast$ and~$\Sigma_2^\top$.

In this article, we deal with rewriting in monoids, that is, categories with only one $0$-cell, so that the set~$\Sigma_0$ is reduced to a set with exactly one element denoted $\bullet$. In this case, the $1$-polygraph $(\Sigma_0,\Sigma_1)$ will be identified to a set $\Sigma_1$ and $\Sigma_1^\ast$ will be identified to the free monoid on $\Sigma_1$.

\subsubsection{Presentations of monoids by \pdf{2}-polygraphs}

The \emph{monoid presented by a $2$-polygraph $\Sigma$}, denoted by $\cl{\Sigma}$, is defined as the quotient of the free monoid~$\Sigma_1^\ast$ by the relations $s_1(\beta)=t_1(\beta)$, for every $2$-cell $\beta$ of~$\Sigma_2^\ast$.
A presentation of a monoid $\M$ is a $2$-polygraph whose presented monoid is isomorphic to~$\M$. Two $2$-polygraphs are \emph{Tietze equivalent} if they present isomorphic monoids.

\subsubsection{Tietze transformations of \pdf{2}-polygraphs}
\label{Subsubsection:TietzeTransformations2Polygraphs}

A $2$-cell~$\beta$ of a $2$-polygraph~$\Sigma$ is \emph{collapsible}, if~$t_1(\beta)$ is a~$1$-cell of $\Sigma_1$ and the $1$-cell $s_1(\beta)$ does not contain $t_1(\beta)$.
The target of a collapsible $2$-cell is a \emph{redundant} $1$-cell.
Tietze transformations were introduced in group theory in order to transform a presentation of a group into a presentation of the same group by adding or removing generators and rules,~\cite{Tietze08}. This notion can be defined for $2$-polygraphs.
Recall from~{\cite[2.1.1.]{GaussentGuiraudMalbos14}}, that an \emph{elementary Tietze transformation} of a $2$-polygraph $\Sigma$ is a $2$-functor with domain $\tck{\Sigma}_2$ that belongs to one of the following four transformations:
\begin{enumerate}[{\bf i)}]
\item adjunction
$\iota_{\beta}^1 : \tck{\Sigma}_2 \fl \tck{\Sigma}_2[x](\beta)$
of a redundant $1$-cell $x$ with its collapsible $2$-cell $\beta$:
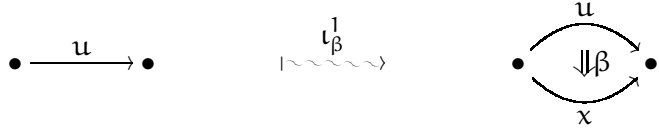
\[
\xymatrix@C=3.5em{
\bullet
	\ar [r] ^-{u} 
& \bullet
&
   \ar@{|~>} [r] ^{\iota_{\beta}^1}
&
& \bullet
	\ar@/^3ex/ [r] ^-{u} ^{}="src"
	\ar@/_3ex/ [r] _-{x} ^{}="tgt"
	\ar@2 "src"!<0pt,-10pt>;"tgt"!<0pt,10pt> ^-{\beta}
& \bullet
}
\]
\item elimination 
$\pi_\beta : \tck{\Sigma}_2 \fl 
\tck{(\Sigma_1\setminus\{x\},\Sigma_2\setminus\{\beta\})}$ of a redundant $1$-cell $x$ with its collapsible $2$-cell $\beta$:
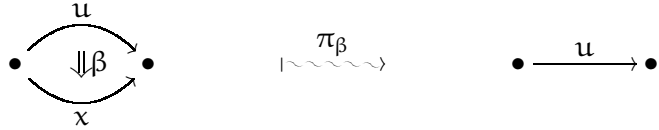
\[
\xymatrix@C=3.5em{
\bullet
	\ar@/^3ex/ [r] ^-{u} ^{}="src"
	\ar@/_3ex/ [r] _-{x} ^{}="tgt"
	\ar@2 "src"!<0pt,-10pt>;"tgt"!<0pt,10pt> ^-{\beta}
& \bullet
&
   \ar@{|~>} [r] ^{\pi_\beta}
&
&
\bullet
	\ar [r] ^-{u} 
& \bullet
}
\]
which maps $x$ to $u$ and the $2$-cell $\beta$ to $1_u$ and being identity on the others cells,
\item adjunction $\iota_\beta : \tck{\Sigma}_2 \fl \tck{\Sigma}_2(\beta)$ of a redundant $2$-cell $\beta$:
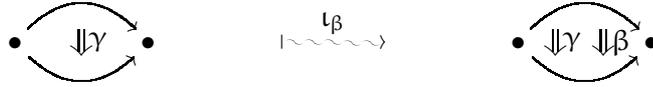
\[
\xymatrix@C=3.5em{
\bullet
	\ar@/^3ex/ [r] ^-{} ^{}="src"
	\ar@/_3ex/ [r] _-{} ^{}="tgt"
	\ar@2 "src"!<0pt,-10pt>;"tgt"!<0pt,10pt> ^-{\gamma}
& \bullet
&
   \ar@{|~>} [r] ^{\iota_\beta}
&
&
\bullet
	\ar@/^3ex/ [r] ^-{} ^{}="src"
	\ar@/_3ex/ [r] _-{} ^{}="tgt"
	\ar@2 "src"!<-11pt,-10pt>;"tgt"!<-11pt,10pt> ^-{\gamma}
	\ar@2 "src"!<7pt,-10pt>;"tgt"!<7pt,10pt> ^-{\beta}
& \bullet
}
\]
\item elimination $\pi_{(\gamma,\beta)} : \tck{\Sigma}_2 \fl \tck{\Sigma}_2/(\gamma,\beta)$ of a redundant $2$-cell $\beta$:
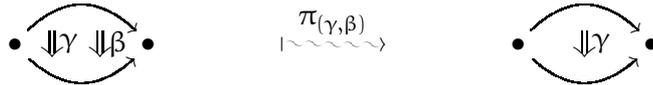
\[
\xymatrix@C=3.5em{
\bullet
	\ar@/^3ex/ [r] ^-{} ^{}="src"
	\ar@/_3ex/ [r] _-{} ^{}="tgt"
	\ar@2 "src"!<-11pt,-10pt>;"tgt"!<-11pt,10pt> ^-{\gamma}
	\ar@2 "src"!<7pt,-10pt>;"tgt"!<7pt,10pt> ^-{\beta}
& \bullet
&
   \ar@{|~>} [r] ^{\pi_{(\gamma,\beta)}}
&
&
\bullet
	\ar@/^3ex/ [r] ^-{} ^{}="src"
	\ar@/_3ex/ [r] _-{} ^{}="tgt"
	\ar@2 "src"!<0pt,-10pt>;"tgt"!<0pt,10pt> ^-{\gamma}
& \bullet
}
\] 
\end{enumerate}
If $\Sigma$ and $\Upsilon$ are $2$-polygraphs, a \emph{Tietze transformation} from $\Sigma$ to $\Upsilon$ is a $2$-functor $F : \tck{\Sigma} \fl \tck{\Upsilon}$ that decomposes into sequence of elementary Tietze transformations.
Two $2$-polygraphs are Tietze equivalent if, and only if, there exists a Tietze transformation between them~{\cite[Theorem 2.1.3.]{GaussentGuiraudMalbos14}}.

\subsubsection{Nielsen transformation}

Recall the notion of Nielsen transformation from~{\cite[2.1.4.]{GaussentGuiraudMalbos14}}.
Given a $2$-polygraph $\Sigma$ and a $2$-cell 
\[
u_1 \odfl{\gamma_1} u \odfl{\gamma} v \odfl{\gamma_2} v_2
\] 
in $\tck{\Sigma}_2$, the \emph{Nielsen transformation} $\kappa_{\gamma \leftarrow \beta}$ is the Tietze transformation that replaces in the $(2,1)$-category~$\tck{\Sigma}_2$ the $2$-cell $\gamma$ by a $2$-cell $\beta : u_1 \dfl v_2$. The transformation $\kappa_{\gamma \leftarrow \beta}$ can be decomposed into the following composition of elementary Tietze transformations:
\[
\tck{\Sigma}_2 \ofl{\iota_\beta} \tck{\Sigma}_2(\beta) 
\overset{\pi_{(\gamma_1^-\star_1\beta\star_1\gamma_2^-,\gamma)}}{\longrightarrow} \tck{\Sigma}_2/(\gamma_1^-\star_1\beta\star_1\gamma_2^-,\gamma).
\]
When $\gamma_2$ is identity, we will denote by $\kappa_{\gamma \leftarrow \beta}^{'}$ the Nielsen transformation which, given a $2$-cell~\mbox{$u_1\odfl{\gamma_1} u\odfl{\gamma} v$} in $\tck{\Sigma}_2$, replaces the $2$-cell $\gamma$ by a $2$-cell $\beta: u_{1}\dfl v$.

\subsubsection{Rewriting sequences}

A \emph{rewriting step} of a $2$-polygraph~$\Sigma$ is a $2$-cell of~$\Sigma_2^*$ with shape
\[
\xymatrix@C=3.5em{
\bullet
	\ar [r] ^-{w} 
& \bullet
	\ar@/^3ex/ [r] ^-{u} ^{}="src"
	\ar@/_3ex/ [r] _-{v} ^{}="tgt"
	\ar@2 "src"!<0pt,-10pt>;"tgt"!<0pt,10pt> ^-{\beta}
& \bullet
	\ar [r] ^-{w'}
& \bullet
}
\]
where $\beta$ is a $2$-cell of~$\Sigma_2$ and~$w$ and~$w'$ are $1$-cells of~$\Sigma_1^*$. A \emph{rewriting sequence} of~$\Sigma$ is a finite or infinite sequence 
\[
\xymatrix@C=1.6em{
{u_1}
	\ar@2 [r] ^-{}
& {u_2}
	\ar@2 [r] ^-{}
& \;\cdots\;
	\ar@2 [r] ^-{}
& {u_n}
	\ar@2 [r] ^-{}
& \cdots
}
\]
of rewriting steps. If~$\Sigma$ has a rewriting sequence from~$u$ to~$v$, we say that \emph{$u$ rewrites into~$v$}. A $1$-cell~$u$ of~$\Sigma_1^*$ is a \emph{normal form} if there is no rewriting step with source~$u$.
The $2$-polygraph~$\Sigma$ \emph{terminates} if it has no infinite rewriting sequence. In that case, every $1$-cell of~$\Sigma_1^*$ has at least one normal form.

\subsubsection{Branchings}

A \emph{branching} of the $2$-polygraph~$\Sigma$ is a non ordered pair~$(f,g)$ of $2$-cells of~$\Sigma_2^*$ with a common source, that is~$s_1(f) = s_1(g)$. 
A branching~$(f,g)$ is \emph{local} if~$f$ and~$g$ are rewriting steps. A branching is \emph{aspherical} if it is of the form~$(f,f)$, for a rewriting step $f$ and \emph{Peiffer} when it is of the form~$(fv,ug)$ for rewriting steps $f$ and $g$ with $s_1(f)=u$ and $s_1(g)=v$. The \emph{overlapping} branchings are the remaining local branchings.
Local branchings are ordered by the order~$\sqsubseteq$ generated by the relations
\[
(f,g) \:\sqsubseteq\: \big( w f w', w g w')
\]
given for any local branching~$(f,g)$ and any possible $1$-cells~$w$ and~$w'$ of the category~$\Sigma_1^*$. An overlapping local branching that is minimal for the order~$\sqsubseteq$ is called a \emph{critical branching}.

\subsubsection{Confluence}

A branching $(f,g)$  is \emph{confluent} if there exist $2$-cells~$f'$ and~$g'$ in~$\Sigma_2^*$, as in the following diagram: 
\[
\xymatrix @R=0.6em @C=3em{
& {v}
	\ar@2@/^/ [dr] ^-{f'}
\\
{u}
	\ar@2@/^/ [ur] ^-{f}
	\ar@2@/_/ [dr] _-{g}
&& {u'}
\\
& {w}
	\ar@2@/_/ [ur] _-{g'}
}
\]
We say that a $2$-polygraph~$\Sigma$ is \emph{confluent} (resp.\ \emph{locally confluent}) if all of its branchings (resp.\ local branchings) are confluent.  If~$\Sigma$ is confluent, every $1$-cell of~$\Sigma^*$ has at most one normal form. The critical branching Lemma, {\cite[Theorem 3.1.5.]{GuiraudMalbos14}}, states that a $2$-polygraph is locally confluent if and only if all its critical branchings are confluent. The Newman Lemma, {\cite[Theorem 3.1.6.]{GuiraudMalbos14}}, states that for terminating $2$-polygraphs, local confluence and confluence are equivalent properties.

\subsubsection{Convergence}

A $2$-polygraph~$\Sigma$ is \emph{convergent} if it terminates and it is confluent. Such a~$\Sigma$ is called a \emph{convergent presentation of any monoid isomorphic to~$\cl{\Sigma}$}. In that case, every $1$-cell~$u$ of~$\Sigma_1^*$ has a unique normal form.

\subsection{Plactic monoids}
\label{placticmonoids}

We recall the definition and properties of plactic monoids. We refer the reader to~\cite{Lothaire02} and~\cite{Fulton97} for a full introduction to the plactic structure and tableaux.

\subsubsection{Rows, columns and tableaux}

For a natural number $n>0$, we denote by $[n]$ the finite \linebreak set~$\{1,2,\ldots,n\}$ totally ordered by $1<2<\ldots<n$.
A \emph{row} is a non-decreasing $1$-cell $x_1\ldots x_k$ in the free monoid $[n]^\ast$, \emph{i.e.}, with $x_{i}\leq x_{i+1}$  for $1\leq i\leq k-1$.
A \emph{column} is a decreasing $1$-cell $x_p\ldots x_1$ in the free monoid~$[n]^\ast$, \emph{i.e.}, with $x_{i+1}>x_{i}$,  for $1\leq i \leq p-1$. We will denote by $\col(n)$ the set of non-empty columns in $[n]^\ast$.
We denote by $\len(w)$ the length of a $1$-cell $w$ and we denote by $\lnds(w)$ the length of the longest non-decreasing subsequence in $w$.

A row  $x_1\ldots x_k$ \emph{dominates} a row $y_1 \ldots y_l$, and we denote $x_1\ldots x_k\vartriangleright y_1 \ldots y_l$, if $k\leq l$ and $x_{i}>y_{i}$, for $1\leq i \leq k$.
Any $1$-cell $w$ in $[n]^\ast$ has a unique decomposition as a product of rows of maximal length~$w=u_{1}\ldots u_{k}$. Such a $1$-cell $w$ is a \emph{(semistandard) tableau}  if $u_{1}\vartriangleright u_{2} \vartriangleright \ldots \vartriangleright u_{k}$. It is usual to write tableaux in a planar form, with the rows placed in order of domination from bottom to top and left-justified as in \cite{Fulton97}. For example, the $1$-cells $13123$ and $23412$ are not tableaux and the $1$-cell~$6745662233461112234$ is a tableau whose planar representation is
\begin{equation}
\label{Example:Tableau}
\young(1112234,223346,4566,67)
\end{equation}

The \emph{column reading} of the planar representation of a tableau $w$ constructs a $1$-cell, denoted by $C(w)$, obtained by reading the planar representation of $w$ column-wise from bottom to top and from left to right. For example, the column reading of the tableau (\ref{Example:Tableau}) is $6421752163163242634$.

\subsubsection{Total orders on columns}

We will denote by $\ordrecoldeglex$ the total order on $\col(n)$ defined by \mbox{$u\ordrecoldeglex v$} if 
\[
\len(u)<\len(v) \quad \text{or} \quad \big(\;\len(u)=\len(v) \text{ and } u<_{lex} v\;\big),
\]
for all $u$ and $v$ in $\col(n)$, where~$<_{\text{lex}}$ denotes the lexicographic order on $[n]^\ast$ induced by the total order on~$[n]$.
We will denote by $\ordrecolrev$ the total order on $\col(n)$ defined by $u\ordrecolrev v$ if
\[
\len(u)>\len(v)\quad  \text{or}\quad  \big(\;\len(u)= \len(v)  \text{ and } u<_{\text{lex}} v\;\big),
\]
for all $u$ and $v$ in $\col(n)$.

\subsubsection{Schensted's algorithm}
\label{Algorithm:Schensted}

The \emph{Schensted algorithm} computes for each $1$-cell $w$ in the free monoid~$[n]^\ast$ a tableau denoted by $P(w)$, called the \emph{Schensted tableau} of $w$ and constructed as follows, \cite{Schensted61}.
Given $w$ a tableau written as a product of rows of maximal length $w=u_1\ldots u_k$ and $y$ in~$[n]$, it computes the tableau $P(wy)$ as follows: 
\begin{enumerate}[{\bf i)}]
\item if $u_{k}y$ is a row, the result is $u_{1}\ldots u_{k}y$ ;
\item if $u_{k}y$ is not a row, then suppose $u_{k}=x_{1}\ldots x_{l}$ with $x_{i}$ in $[n]$ and let $j$ minimal such that $x_{j}>y$, 
then the result is $P(u_{1}\ldots u_{k-1}x_{j})v_{k}$ where 
$v_{k}=x_{1}\ldots x_{j-1}yx_{j+1}\ldots x_{l}$.
\end{enumerate}

Given a $1$-cell $w$, the tableau $P(w)$ is computed by starting with the empty tableau, corresponding to the empty $1$-cell, and iteratively applying the Schensted algorithm. In other words, $P(w)$ is the row reading of the planar representation of the tableau computed by the Schensted algorithm.
The number of columns in $P(w)$ is equal to $\lnds(w)$, \cite{Schensted61}. In particular, if $P(w)$ consists of one column, then the $1$-cell~$w$ is a column.
Finally, note that if $w$ is a tableau, then $P(w)=w$ holds in~$[n]^\ast$. 

\subsubsection{Plactic monoids}
\label{Placticcongruence}

We will denote by $\sim_{\mathrm{plax}(n)}$ the equivalence relation on the free monoid $[n]^\ast$ defined by  $u\sim_{\mathrm{plax}(n)} v$ if $P(u) = P(v)$ in $[n]^\ast$.
The \emph{plactic monoid of rank $n$}, denoted by $\P_n$, is the quotient of the free monoid $[n]^\ast$ by the congruence $\sim_{\mathrm{plax}(n)}$.

\subsubsection{Knuth's $2$-polygraph and the plactic congruence}

The \emph{Knuth $2$-polygraph of rank $n$} is the 
$2$-polygraph, denoted by~$\Knuth_2(n)$, whose set of $1$-cells is $[n]$ and the set of $2$-cells is 
\begin{equation}
\label{KnuthRelations}
\{\;
zxy \odfl{\eta_{x,y,z}} xzy \;|\; 1\leq x\leq y < z \leq n
\;\}\,\cup\,\{\;
yzx \odfl{\epsilon_{x,y,z}} yxz \;|\; 1\leq x<y\leq z\leq n
\;\}.
\end{equation}
These $2$-cells correspond to the \emph{Knuth relations} defined in \cite{Knuth70} with an orientation compatible with the lexicographic order $<_{\mathrm{lex}}$.
The congruence on the free monoid $[n]^\ast$ generated by the $2$-polygraph~$\Knuth_2(n)$ is called the \emph{plactic congruence of rank $n$}.
Knuth showed in~\cite{Knuth70}  that for any $u$ and $v$ in $[n]^\ast$, we have~$u\sim_{\mathrm{plax}(n)} v$ if and only if $u$ and $v$ are equal modulo the plactic congruence.

\begin{proposition}[{\cite[Theorem 6]{Knuth70}}]
The $2$-polygraph $\Knuth_2(n)$ is a presentation of the monoid $\P_n$.
\end{proposition}

Each plactic congruence class contains exactly one tableau. Indeed, any $1$-cell~$w$ in $[n]^\ast$ is equal to its Schensted's tableau in $\P_n$, that is, $w=P(w)$ holds in $\P_n$,~{\cite[Proposition 5.2.3]{Lothaire02}}.
Moreover, a $1$-cell $w$ is equal to the column reading of the planar representation of the tableau $P(w)$, that is, $w=C(P(w))$ holds in $\P_n$~{\cite[Problem~5.2.4]{Lothaire02}}.
Finally, the Knuth relations being homogeneous, we have $\len(P(w)) = \len(w)$, for any $1$-cell $w$ in~$[n]^\ast$.

\subsection{Pre-column presentation}
\label{Section:Pre-columnPresentation}

\subsubsection{Columns as generators}

One adds to the presentation  $\Knuth_{2}(n)$ one superfluous generator $c_{u}$ for any~$u$ in $\col(n)$. Let us denote by
\[
\Colo_1(n)=\big\{c_{u} \; \big| \; u\in\col(n) \big\}
\]
the set of \emph{column generators} of the monoid $\P_n$ and by 
\[
\CC_2(n)
\: = \:
\big\{\xymatrix@C=3em{c_{x_{p}}\ldots c_{x_{1}} \odfll{\gamma_u} c_u} \; \big| \; 
u=x_{p} \ldots x_{1}\in\col(n)\;\;\text{with}\;\len(u)\geq 2 \big\}
\]
the set of the defining relations for the column generators.
In the free monoid $\Colo_1(n)^\ast$, the  Knuth relations~(\ref{KnuthRelations}) can  be written  in the following form
\begin{equation}
\label{KnuthRelationsCol}
\big\{ c_{z}c_{x}c_{y} \odfll{\eta_{x,y,z}^{c}} c_{x}c_{z}c_{y} \;\big|\;  1 \leq x\leq y<z  \leq n \big\}
\;\cup\;
\big\{ c_{y}c_{z}c_{x}\odfll{\epsilon_{x,y,z}^{c}}c_{y}c_{x}c_{z} \;\big|\;  1 \leq x< y\leq z \leq n \big\}.
\end{equation}

Let denote by $\Knuthc_2(n)$ the $2$-polygraph whose set of $1$-cells is $\{c_{1},\ldots,c_{n}\}$ and whose set of $2$-cells is given by~(\ref{KnuthRelationsCol}). 
By definition, this $2$-polygraph is Tietze equivalent to the $2$-polygraph $\Knuth_2(n)$. Indeed, the mapping~$i\mapsto c_i$, for any $i$ in $[n]$, induces an isomorphism between the two presented monoids.
In the sequel, we will identify the $2$-polygraphs  $\Knuthc_2(n)$ and $\Knuth_2(n)$ through this mapping.
Let us define the $2$-polygraph $\Knuthcc_2(n)$, whose $1$-cells are columns and $2$-cells are the defining relations for columns generators and the Knuth relations:
\[
\Knuthcc_2(n) :=  \langle\;\; \Colo_1(n) \;\;|\;\; \CC_2(n) \, \cup \, \Knuthc_2(n)\;\;\rangle.
\]

\begin{proposition}
\label{KGamma2(n)Proposition}
For $n>0$, the $2$-polygraph $\Knuthcc_2(n)$ is a presentation of the monoid $\P_n$.
\end{proposition}

\begin{proof}
We have $\Colo_1(n)=\{\,c_{1},\ldots, c_{n}\,\}\cup \{\,c_{u}\;\big|\; u\in \col(n),\, \len(u)\geq 2 \,\}$, thus in order to prove that the $2$-polygraphs $\Knuthcc_2(n)$ and $\Knuthc_2(n)$ are Tietze equivalent, we add to the $2$-polygraph $\Knuthc_2(n)$ all the column generator $c_u$, for all $u=x_p\ldots x_1$ in $\col(n)$ such that $\len(u)\geq 2$, and the corresponding collapsible $2$-cell : $\gamma_u : c_{x_p}\ldots c_{x_1} \dfl c_u$. We apply successively a Tietze transformation~$\iota^1_{\gamma_u}$, defined in \ref{Subsubsection:TietzeTransformations2Polygraphs}. {\bf i)}, from the bigger column in~$\col(n)$ to the smaller one with respect to the order $\ordrecoldeglex$. 
The composite 
\[
T_{1}= \iota^1_{\gamma_1} \circ \ldots \circ \iota^1_{\gamma_{u_{i}}}\circ \iota^1_{\gamma_{u_{i+1}}}\circ \ldots \circ \iota^1_{\gamma_{n\ldots 1}},
\]
with $u_{i}\ordrecoldeglex u_{i+1}$, defines a Tietze transformation 
\[
T_1 : 
\tck{\Knuthc_2(n)}
\fll 
\tck{\Knuthcc_2(n)},
\]
which proves that~$\Knuthcc_2(n)$ is Tietze equivalent to $\Knuthc_2(n)$, hence Tietze equivalent to $\Knuth_{2}(n)$.
\end{proof}

\subsubsection{Pre-column presentation}
\label{Subsubsection:Pre-columnPresentation}

Let us define the $2$-polygraph~$\PreColo_2(n)$ whose set of $1$-cells is~$\Colo_{1}(n)$ and the set of $2$-cells is
\[
\PreColo_2(n) = \PC_2(n)
\,\cup\,
\big\{
c_{x}c_{u} \odfl{\alpha'_{x,u}} c_{xu}
\;|\; xu\in \col(n) \;\;\text{and}\;\; 1\leq x \leq n
\big\},
\]
where
\[
\PC_2(n) = \big\{\;
c_{x}c_{zy}\odfll{\alpha'_{x,zy}}c_{zx}c_{y}
\;|\; 1 \leq x\leq y<z \leq n
\;\big\}\,\cup\,\big\{\; c_{y}c_{zx}\odfll{\alpha'_{y,zx}}c_{yx}c_{z}
\;|\; 1 \leq x< y\leq z \leq n
\;\big\}.
\]

We will see in Lemma~\ref{Gamma2(n)Lemma} that the $2$-cells $\alpha'_{x,zy}$ and $\alpha'_{y,zx}$ correspond respectively to the Knuth relations $\eta_{x,y,z}$ for~$1 \leq x\leq y<z \leq n$ and $\epsilon_{x,y,z}$ for~$1 \leq x< y\leq z \leq n$. They also correspond to the following Schensted transformations as indicated in the following diagrams:
\[
\small
\xymatrix @C=1.5em @R=0.6em {
{\young(x)\;\raisebox{-0.4cm}{\young(y,z)}}
\ar@2@/^/[rrrr] ^-{\alpha'_{x,zy}}
\ar@{..>} [ddd] _-{C\big(\scalebox{0.7}{\young(x)\;\raisebox{-0.4cm}{\young(y,z)}}\big)}
&&&&{\raisebox{-0.4cm}{\young(xy,z)}}
\ar@{..>}[ddd] ^-{C\big( \scalebox{0.7}{\raisebox{-0.4cm}{\young(xy,z) }}\big)}
\\ 
\\\\
{xzy}
\ar@{..>} [rr] ^-{\scalebox{0.8}{P(xz)}}
&&{\young(xz)\;y}
\ar@{..>} [uuurr] ^-{\scalebox{0.8}{P((xz)y)}}
&&{zxy}
\ar@2@/^5ex/ [llll] ^-{\eta_{x,y,z}} ^{}="tgt"
}
\quad
\xymatrix @C=1.5em @R=0.6em {
{\young(y)\;\raisebox{-0.4cm}{\young(x,z)}}
\ar@2@/^/[rrrr] ^-{\alpha'_{y,zx}}
\ar@{..>}[ddd] _-{C\big(\scalebox{0.7}{\young(y)\;\raisebox{-0.4cm}{\young(x,z)}}\big)}
&&&&{\raisebox{-0.4cm}{\young(xz,y)}}
\ar@{..>}[ddd] ^-{C\big( \scalebox{0.7}{\raisebox{-0.4cm}{\young(xz,y) }}\big)}
\\ 
\\\\
{yzx}
\ar@{..>} [rr] ^-{\scalebox{0.8}{P(yz)}}
\ar@2@/_5ex/ [rrrr] _-{\epsilon_{x,y,z}} ^{}="tgt"
&&{\young(yz)\;x}
\ar@{..>} [uuurr] ^-{\scalebox{0.8}{P((yz)x)}}
&&{yxz}
}
\]

\begin{proposition}
\label{Proposition:PresentationPreColo2}
For $n>0$, the $2$-polygraph $\PreColo_2(n)$ is a presentation of the monoid $\P_n$.
\end{proposition}

The $2$-polygraph $\PreColo_2(n)$ is called the \emph{pre-column presentation of~$\P_n$}.
The proof of Proposition~\ref{Proposition:PresentationPreColo2} is given by the following two lemmas. 

\begin{lemma}
\label{Gamma2(n)Lemma}
The $2$-polygraph 
\[
\CPC_2(n) :=  \langle\;\; \Colo_1(n) \;\;|\;\; \CC_2(n)  \cup \PC_2(n)\;\;\rangle
\]
is Tietze equivalent to the $2$-polygraph $\Knuthcc_2(n)$.
\end{lemma}

\begin{proof}
For $1\leq x\leq y<z \leq n$, consider the following critical branching
\[\xymatrix @C=1.5em @R=0.6em {
& {c_{x}c_{z}c_{y}}
  \ar@2[r] ^-{c_{x}\gamma_{zy}}
& {c_{x}c_{zy}}
\\
{c_{z}c_{x}c_{y}}
	\ar@2@/^/ [ur] ^-{\eta_{x,y,z}^{c}}
	\ar@2@/_/ [dr] _-{\gamma_{zx}c_{y}}
\\
& {c_{zx}c_{y}}
}\]
of the $2$-polygraph $\Knuthcc_2(n)$. Let consider the Tietze transformation 
\[
\kappa_{\eta_{x,y,z}^{c} \leftarrow \alpha'_{x,zy}} : 
\tck{\Knuthcc_2(n)} \fll \tck{\Knuthcc_2(n)}/(\eta_{x,y,z}^{c}\leftarrow\alpha'_{x,zy}),
\]
that substitutes the $2$-cell 
$\alpha'_{x,zy} : c_{x}c_{zy} \dfl c_{zx}c_{y}$ to the $2$-cell $\eta_{x,y,z}^{c}$, for every~$1\leq x\leq y<z \leq n$.
We denote by $T_{\eta\leftarrow \alpha'}$ the successive applications of the Tietze transformation $\kappa_{\eta_{x,y,z}^{c}\leftarrow\alpha'_{x,zy}}$, for every $1\leq x\leq y<z \leq n$, with respect to the lexicographic order on the triples $(x,y,z)$ induced by the total order on $[n]$. 

Similarly, for $1\leq x< y\leq z \leq n$, consider the following critical branching
\[\xymatrix @C=1.5em @R=0,6em {
& {c_{y}c_{x}c_{z}}
  \ar@2[r] ^-{\gamma_{yx}c_{z}}
& {c_{yx}c_{z}}
\\
{c_{y}c_{z}c_{x}}
	\ar@2@/^/ [ur] ^-{\epsilon_{x,y,z}^{c}}
	\ar@2@/_/ [dr] _-{c_{y}\gamma_{zx}}
\\
& {c_{y}c_{zx}}
}
\]
of the $2$-polygraph  $\Knuthcc_2(n)$.
Let consider the Tietze transformation 
\[
\kappa_{\epsilon_{x,y,z}^{c} \leftarrow\alpha'_{y,zx}} : 
\tck{\Knuthcc_2(n)} \fll \tck{\Knuthcc_2(n)}/(\epsilon_{x,y,z}^{c} \leftarrow\alpha'_{y,zx}),
\]
that substitutes the $2$-cell 
$\alpha'_{y,zx} : c_{y}c_{zx} \dfl c_{yx}c_{z}$ to the $2$-cell $\epsilon_{x,y,z}^{c}$, for every $1\leq x < y \leq z \leq n$. 
We denote by $T_{\epsilon\leftarrow \alpha'}$ the successive applications of the Tietze transformation $\kappa_{\epsilon_{x,y,z}^{c}\leftarrow\alpha'_{x,zy}}$, for every $1\leq x < y \leq z \leq n$, with respect to the lexicographic order on the triples $(x,y,z)$ induced by the total order on $[n]$.

Let define the composite $T_{\eta,\epsilon \leftarrow \alpha'}=T_{\eta\leftarrow \alpha'} \circ T_{\epsilon\leftarrow \alpha'}$, this gives us a Tietze transformation:
\[
T_{\eta,\epsilon\leftarrow \alpha'} : \tck{\Knuthcc_2(n)} \fll \tck{\CPC_2(n)}.
\]
In this way, the $2$-polygraphs $\Knuthcc_2(n)$ and $\CPC_2(n)$ are Tietze equivalent.
\end{proof}

The following lemma proves that the $2$-polygraph $\PreColo_2(n)$ is a presentation of the monoid $\P_n$. 

\begin{lemma}
\label{Precolo2nLemma}
The $2$-polygraph $\PreColo_2(n)$ is Tietze equivalent to the $2$-polygraph $\CPC_2(n)$.
\end{lemma}

\begin{proof}
Let $x_{p}\ldots x_{1}$ be a column  with $\len(x_{p}\ldots x_{1})>2$ and  define $\alpha'_{y,x} :=\gamma_{yx} : c_{y}c_{x}\dfl c_{yx},$ for every~$x<y$.
Consider the following critical branching 
\[
\xymatrix @C=1.5em @R=0,6em {
& {c_{x_{p}} c_{x_{p-1}\ldots x_{1}}}
\\
{c_{x_{p}}c_{x_{p-1}}\ldots c_{x_{1}}}
	\ar@2@/^/ [ur] ^-{c_{x_{p}}\gamma_{x_{p-1}\ldots x_{1}}}
	\ar@2@/_/ [dr] _-{\gamma_{x_{p}\ldots x_{1}}}
\\
& {c_{x_{p}\ldots x_{1}}}
}
\]
of the $2$-polygraph  $\CPC_{2}(n)$  and  the following  Tietze transformation 
\[
\kappa_{\gamma_{x_{p}\ldots x_{1}} \leftarrow\alpha'_{x_{p},x_{p-1}\ldots x_{1}}}^{'} : 
\tck{\CPC_2(n)} \fll \tck{\CPC_2(n)}/(\gamma_{x_{p}\ldots x_{1}} \leftarrow\alpha'_{x_{p},x_{p-1}\ldots x_{1}}),
\]
that substitutes the $2$-cell 
\[\alpha'_{x_p,x_{p-1}\ldots x_{1}}: c_{x_{p}}c_{x_{p-1}\ldots x_{1}}\odfl{} c_{x_p\ldots x_1},\]
 to the $2$-cell 
\[\gamma_{x_{p}\ldots x_{1}}: c_{x_p}\ldots c_{x_{1}} \odfl{} c_{x_p\ldots x_1},\] for each column $x_{p}\ldots x_{1}$ such that $\len(x_{p}\ldots x_{1})>2$. 
Starting from the $2$-polygraph $\CPC_2(n)$, we apply successively  the Tietze transformation $\kappa_{\gamma_{x_{p}\ldots x_{1}} \leftarrow\alpha'_{x_{p},x_{p-1}\ldots x_{1}}}^{'}$, for every  column $x_{p}\ldots x_{1}$ such that~$\len(x_{p}\ldots x_{1})>2$,  from the bigger to the smaller one with respect to the total order  $\ordrecoldeglex$.

Let us define the composite 
\[
T_{\gamma \leftarrow \alpha'}= \kappa_{\gamma_{x_3x_2 x_1} \leftarrow\alpha'_{x_{3},x_{2}x_{1}}}^{'} \circ \ldots \circ \kappa_{\gamma_{x_{n}\ldots x_{1}} \leftarrow\alpha'_{x_{n},x_{n-1}\ldots x_{1}}}^{'},
\] 
with $x_{3}x_{2}x_{1}\ordrecoldeglex \ldots \ordrecoldeglex x_{n}\ldots x_{1}$. This gives us a Tietze transformation:
\[
T_{\gamma \leftarrow \alpha'}: \tck{\CPC_2(n)} \fll \tck{\PreColo_2(n)}.
\]
In this way, we prove that $\PreColo_2(n)$ is Tietze equivalent to $\CPC_2(n)$.
\end{proof}

To resume the construction of this section, we have constructed the following Tietze equivalences:
\[
\tck{\Knuth_2(n)}
\ofl{T_1}
\tck{\Knuthcc_2(n)}
\ofl{T_{\eta,\epsilon\leftarrow \alpha'}} 
\tck{\CPC_2(n)}
\ofl{T_{\gamma \leftarrow \alpha'}}
\tck{\PreColo_2(n)}.
\]

\subsection{Column presentation}
\label{Section:ColumnPresentation}

\subsubsection{Notation}

Let $n>0$ be a natural number.
Given columns $u=x_p\ldots x_1$ and $v=y_q\ldots y_1$ in $\col(n)$, we consider the tableau~$P(uv)$.
As observed in~{\cite[Lemma 3.1.]{CainGrayMalheiro15}}, the length $\lnds(uv)$ of the longest non-decreasing subsequence of $uv$ is lower or equal to $2$. Indeed, if $uv$ is a column, necessary its non-decreasing subsequences are each of length equal to one and thus $\lnds(uv) = 1$. Otherwise, if $uv$ is not a column, then $x_1\leq y_q$. Hence all the non-decreasing subsequences of~$uv$ are of length~$2$.
As a consequence, the tableau $P(uv)$ contains at most two columns. We will use graphical notations depending on whether the tableau $P(uv)$ consists in two columns:
\begin{enumerate}[{\bf i)}]
\item we will denote $\typedeux{1}$ if the planar representation of $P(uv)$ is the tableau:

\renewcommand{\arraystretch}{0.6}
\setlength{\tabcolsep}{0.15cm}
\newcolumntype{C}[1]{>{\centering\arraybackslash }b{#1}}
\begin{center}
\begin{tabular}{|c|c|}
\hline
$x_1$ & $y_1$\\
\hline
$\vdots$&$\vdots$\\
\cline{2-2}
& $y_q$\\
\cline{2-2}
\\
\cline{1-1}
$x_p$\\ 
\cline{1-1}
\end{tabular}
\end{center}
that is, $p\geq q$ and $x_i\leq y_i$, for any $i\leq q$,
\item we will denote $\typedeux{0}$ in all the other cases, that is, when $p<q$ or $x_i>y_i$, for some $i\leq q$.
\end{enumerate}

In the case {\bf ii)}, we will denote $\typedeux{01}$ if the tableau $P(uv)$ has one column and we will denote $\typedeux{02}$ if the tableau $P(uv)$ has two columns.

\subsubsection{Column presentation}

For every  columns $u$ and $v$ in $\col(n)$ such that $\typedeux{0}$, we define a $2$-cell
\[
\alpha_{u,v} :  c_uc_v \dfl c_{w}c_{w'}
\]
where 
\begin{enumerate}[{\bf i)}]
\item $w=uv$ and $c_{w'}=1$, if $\typedeux{01}$,
\item $w$ and $w'$ are respectively the left and right columns of the tableau $P(uv)$, if $\typedeux{02}$.
\end{enumerate}

Let us denote by $\Colo_2(n)$ the $2$-polygraph whose set of $1$-cells is $\Colo_1(n)$ and the set of $2$-cells is
\begin{equation}
\label{rulesAlpha}
\Colo_2(n)
\: = \:
\big\{ \;c_{u}c_{v}\odfl{\alpha_{u,v}} c_{w}c_{w'} \; \big| \; u,v\in\col(n)\;\text{and}\;\typedeux{0} \; \big\}.
\end{equation}

Note that the $2$-cells of $\PreColo_2(n)$ correspond to the $2$-cells $\alpha_{u,v}$ of $\Colo_2(n)$, where $\len(u)=1$ and~$\len(v)=2$.
Moreover, we notice that, for any $2$-cells $\alpha_{u,v}:c_uc_v \dfl c_wc_{w'}$ of $\Colo_2(n)$, there exists a $2$-cell in $\PreColo_2(n)^\ast$ with source $c_uc_v$ and target $c_wc_{w'}$.

\subsubsection{Column presentation and Schensted's algorithm}

Let us remark that Schensted's Algorithm~\ref{Algorithm:Schensted} that computes a tableau~$P(w)$ from a $1$-cell $w$ in $[n]^\ast$, corresponds to the leftmost reduction path in~$\Colo_2^\ast(n)$ from the $1$-cell $w$ to its normal form $P(w)$, that is, the reduction paths obtained by applying the rules of~$\Colo_2(n)$ starting from the left.
For example, consider the $1$-cell $w=421532435452$ in~$[5]^{\ast}$. To compute the tableau~$P(w)$, one applies the following successive rules of~$\Colo_2(5)$ starting in each step from the left:
\[
w=\scalebox{0.72}{\young(4)\,\young(2)\,\young(1)\,\young(5)\,\young(3)\,\young(2)\,\young(4)\,\young(3)\,\young(5)\,\young(4)\,\young(5)\,\young(2)}
\;
\odfl{\alpha_{4,2}}
\;
\scalebox{0.72}{ {\raisebox{-0.43cm}{\young(2,4)}}\,\young(1)\,\young(5)\,\young(3)\,\young(2)\,\young(4)\,\young(3)\,\young(5)\,\young(4)\,\young(5)\,\young(2)}
\;
\odfl{\alpha_{42,1}}
\;
\scalebox{0.72}{ \raisebox{-0.85cm}{\young(1,2,4)}\,\young(5)\,\young(3)\,\young(2)\,\young(4)\,\young(3)\,\young(5)\,\young(4)\,\young(5)\,\young(2)}
\]
\[
\scalebox{0.72}{ \raisebox{-0.85cm}{\young(1,2,4)}\,\young(5)\,\young(3)\,\young(2)\,\young(4)\,\young(3)\,\young(5)\,\young(4)\,\young(5)\,\young(2)}
\;
\odfl{\alpha_{5,3}}
(\,\ldots\,)
\odfl{\alpha_{5,4}}
\scalebox{0.72}{ {\raisebox{-0.85cm}{\young(1,2,4)}\,\raisebox{-0.85cm}{\young(2,3,5)}\,\raisebox{-0.43cm}{\young(3,4)}\,\raisebox{-0.43cm}{\young(4,5)}\,\young(5)\,\young(2)}}
\;
\odfl{\alpha_{5,2}}
\;
\scalebox{0.72}{ {\raisebox{-0.85cm}{\young(1,2,4)}}\,\raisebox{-0.85cm}{\young(2,3,5)}\,\raisebox{-0.43cm}{\young(3,4)}\,\raisebox{-0.43cm}{\young(4,5)}\,\raisebox{-0.43cm}{\young(2,5)}}
\;
\odfl{\alpha_{54,52}}
\;
\scalebox{0.72}{ {\raisebox{-0.85cm}{\young(1,2,4)}}\,\raisebox{-0.85cm}{\young(2,3,5)}\,\raisebox{-0.43cm}{\young(3,4)}\,\raisebox{-0.85cm}{\young(2,4,5)}\,\young(5)}
\]
\[
\scalebox{0.72}{ {\raisebox{-0.85cm}{\young(1,2,4)}}\,\raisebox{-0.85cm}{\young(2,3,5)}\,\raisebox{-0.43cm}{\young(3,4)}\,\raisebox{-0.85cm}{\young(2,4,5)}\,\young(5)}
\;
\odfl{\alpha_{43,542}}
\;
\scalebox{0.72}{ {\raisebox{-0.85cm}{\young(1,2,4)}}\,\raisebox{-0.85cm}{\young(2,3,5)}\,\raisebox{-0.85cm}{\young(2,3,4)}\,\raisebox{-0.43cm}{\young(4,5)}\,\young(5)}
\;
\odfl{\alpha_{532,432}}
\;
\scalebox{0.72}{ {\raisebox{-0.85cm}{\young(1,2,4)}}\,\raisebox{-1.3cm}{\young(2,3,4,5)}\,\raisebox{-0.43cm}{\young(2,3)}\,\raisebox{-0.43cm}{\young(4,5)}\,\young(5)}
\;
\odfl{\alpha_{421,5432}} 
\scalebox{0.72}{ \raisebox{-1.3cm}{\young(12245,2335,44,5)}}  =  P(w)
\]
In particular, for any columns~$u$ and~$v$ in~$\col(n)$ such that~$\typedeux{0}$, applying successive rules of~$\Colo_2(n)$ on~$uv$ starting in each step from the left leads to a unique normal form, which is the tableau~$P(uv)$.

\begin{proposition}
\label{normalformproperty}
The $2$-polygraph~$\Colo_2(n)$ has the unique normal form property.
\end{proposition}

\begin{proof}
Consider a $1$-cell~$w$ in~$\Colo_1(n)^{\ast}$ and let~$w'$ and~$w''$ be normal forms of $w$. Proving the unique normal form property consists in showing that the normal forms~$w'$ and~$w''$ are equal. Let~$T'$ (resp.~$T''$) be the planar representation of~$w'$ (resp.~$w''$). Since~$w'$ and~$w''$ are normal forms, they don't contain any subsequences that form sources of $2$-cells in~$\Colo_2(n)$. As a consequence,~$T'$ (resp.~$T''$) is a juxtaposition of columns that form a tableau.  Hence, the normal forms $w'$ and $w''$ are tableaux such that the equality~$w=w'=w''$ holds in the monoid~$\P_n$. Since each congruence contains exactly one tableau~{\cite[Theorem~5.2.5]{Lothaire02}}, we have that $w'=w''$.
\end{proof}

\begin{proposition}
\label{Proposition:ColumnPresentation}
For $n>0$, the $2$-polygraph $\Colo_2(n)$ is a presentation of the monoid $\P_n$.
\end{proposition}

The $2$-polygraph $\Colo_2(n)$ is called the \emph{column presentation} of the monoid $\P_n$. Note that, the set of columns being finite, this $2$-polygraph is finite.

\begin{proof} 
Let us prove that the $2$-polygraph $\Colo_2(n)$ is Tietze equivalent to the $2$-polygraph $\Knuthcc_2(n)$.
Any $2$-cell in~$\Knuthcc_2(n)$ can be deduced from a $2$-cell in $\Colo_2(n)$ as follows. For any $1\leq x \leq y < z \leq n$ (resp. $1 \leq x < y \leq z \leq n$), the $2$-cells $\eta_{x,y,z}^{c}$ (resp. $\epsilon_{x,y,z}^{c}$) can be deduced by the following composition 
\[
\xymatrix @C=2.3em @R=2.25em{
{c_{z}c_{x}c_{y}}
   \ar@{..>}@2 [r] ^-{\eta_{x,y,z}^{c}}
     \ar@2 [d]_-{\alpha_{z,x}c_{y}}
& {c_{x}c_{z}c_{y}}
     \ar@2 [d] ^-{c_{x}\alpha_{z,y}}
\\
{c_{zx}c_{y}}
& {c_{x}c_{zy}}
    \ar@2[l] ^-{\alpha_{x,zy}}
}
\qquad \raisebox{-0.7cm}{\text{(resp.}} \quad
\xymatrix @C=2.3em @R=2.25em{
{c_{y}c_{z}c_{x}}
   \ar@{..>}@2 [r] ^-{\epsilon_{x,y,z}^{c}}
     \ar@2 [d]_-{c_{y}\alpha_{z,x}}
& {c_{y}c_{x}c_{z}}
     \ar@2 [d] ^-{\alpha_{y,x}c_{z}}
\\
{c_{y}c_{zx}}
\ar@2[r] _-{\alpha_{y,zx}}
& {c_{yx}c_{z}}
}
\quad \raisebox{-0.7cm}{\text{).}}
\]
For any column $x_p\ldots x_1$, the $2$-cell $\gamma_{x_p\ldots x_1}$ can be deduced by the following composition
\[
\xymatrix @C=1.5em @R=2.25em{
c_{x_p} \ldots c_{x_1} 
	\ar@{..>}@2[rr] ^{\gamma_{x_p\ldots x_1}}
	\ar@2[d] _-{\alpha_{x_p,x_{p-1}}c_{x_{p-2}} \ldots c_{x_{1}}}
&& c_{x_p \ldots x_1}
\\
c_{x_px_{p-1}}c_{x_{p-2}} \ldots c_{x_1}
	\ar@2[r]
&(\dots)
	\ar@2[r]
& c_{x_p \ldots x_2}c_{x_1}
	\ar@2[u] _-{\alpha_{x_p\ldots x_2,x_1}}
}
\]
As a consequence, if the $1$-cells $w$ and $w'$ in $\Colo_1(n)^\ast$ are equal modulo relations in $\Knuthcc_2(n)$, then they are equal modulo relations in $\Colo_2(n)$.
Conversely, if the $1$-cells $w$ and $w'$ in $\Colo_1(n)^\ast$ are equal modulo relations in $\Colo_2(n)$, by Proposition~\ref{normalformproperty}, they have the same normal form with respect to~$\Colo_2(n)$. Moreover, this normal form is the common tableau of the $1$-cells $w$ and $w'$. It follows that $w$ and $w'$ are in the plactic congruence and hence  they are equal modulo~$\Knuthcc_2(n)$.
\end{proof}

\subsubsection{Termination of the column presentation}

The termination of the $2$-polygraph $\Colo_2(n)$ can be proved using the terminating order $\ll$ defined on $\Colo_1(n)^\ast$ as follows. For $c_{u_i}$ and $c_{v_j}$ in $\Colo_1(n)$, we have
$c_{u_1}\ldots c_{u_k}\ll c_{v_1}\ldots c_{v_l}$, if
\[
\begin{cases}
k<l\quad \text{or}\\ 
k=l\;\;\text{and}\;\; \exists\, i\in\{1,\ldots,k\}\;\;\text{such that}\;\; u_i \ordrecolrev v_i\;\;\text{and}\;\; c_{u_j}=c_{v_j}\;\; \text{for any}\;\; j<i.
\end{cases}
\]
The relation $\ll$ is a well-ordering on $\Colo_1(n)^\ast$, which is compatible with rules in $\Colo_2(n)$ proving the termination~{\cite[Lemma 3.2]{CainGrayMalheiro15}}.
An other method to prove termination of the $2$-polygraph $\Colo_2(n)$ will be given in \ref{Subsection:QuadraticNormalisation}.

\subsubsection{Confluence of the column presentation}
\label{Subsubsection:ConfluenceColumnPresentation}

The column presentation is confluent,~{\cite[Lemma 3.3]{CainGrayMalheiro15}}. The proof given in~\cite{CainGrayMalheiro15} consists in showing that the $2$-polygraph $\Colo_2(n)$ has the unique normal form property. 
Note that our construction in Section~\ref{SubSection:CoherentColumnPresentation} gives an other proof of the confluence of the $2$-polygraph~$\Colo_2(n)$ by showing the confluence of all the critical branchings of the column presentation. 

\subsubsection{Cardinality of the column presentation}

For $m=1$ and $m=2$, let us denote by $\varkappa(n,m)$  the number of $m$-cells of the presentation $\Colo_2(n)$ of the monoid~$\P_n$. We refer the reader to~\ref{SubSection:Computations} for the values of number of cells of the $2$-polygraph $\Colo_2(n)$ for plactic monoids of low-dimensional rank~$n$. 

\begin{proposition}
For $n>0$, we have  $\varkappa(n,1) = 2^{n} - 1$ and 
\[\varkappa(n,2) =  \varkappa(n,1)^{2}-\left( \prod_{1\leq i\leq j\leq n} {\frac{i+j+1}{i+j-1}} - \prod_{1\leq i\leq j\leq n} {\frac{i+j}{i+j-1}}\right).\]
\end{proposition}

\begin{proof}
The number $\varkappa(n,1)$ is the sum of the number of columns of length $k$ for any $1\leq k \leq n$. Moreover, the number of columns of length $k$ is equal to $\dbinom{n}{k}$.  Hence we have $\varkappa(n,1) = \displaystyle\sum_{k=1}^n{\binom{n}{k}}=2^{n} - 1$.

Denote by $S_{n,q}$ the set of all tableaux with at most $q$ columns and with entries in $[n]$. By Gordon~\cite{Gordon83}, we have 
\[|S_{n,q}| =  \prod_{1\leq i\leq j\leq n} {\frac{q+i+j-1}{i+j-1}}.\]
Then, for two columns~$u$ and~$v$ in~$\col(n)$ the number of possibilities of~$\typedeux{0}$ is~$|S_{n,2}| - |S_{n,1}|$. In addition, the number of possibilities of~$\typedeux{0}$ and~$\typedeux{1}$ is~$\varkappa(n,1)^{2}$. Since~$\varkappa(n,2)$ is equal to the number of possibilities of~$\typedeux{0}$, we have~$\varkappa(n,2) = \varkappa(n,1)^{2}-\left( |S_{n,2}|-|S_{n,1}|\right)$.
\end{proof}

\section{Coherent column presentation}
\label{Section:CoherentColumnPresentation}

\subsection{Coherent presentations of monoids}

\subsubsection{$(3,1)$-polygraph}

A \emph{$(3,1)$-polygraph} is a pair $(\Sigma_2,\Sigma_3)$ made of a $2$-polygraph $\Sigma_2$ and a globular extension $\Sigma_3$ of the $(2,1)$-category $\tck{\Sigma}_2$:
\[
\xymatrix @C=4.2em @R=2.8em{
\tck{\Sigma}_2
& \Sigma_3
	\ar@<.45ex> [l] ^{t_2}
	\ar@<-.45ex> [l] _{s_2}
}
.
\]
An element of the globular extension $\Sigma_3$ can be represented by a $3$-cell with the following globular shape
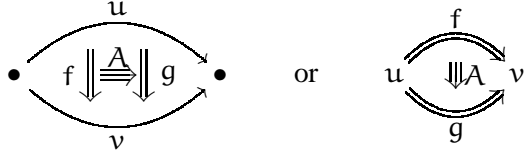
\[
\xymatrix @C=6em {
\bullet
	\ar @/^4ex/ [r] ^-{u} _-{}="src"
	\ar @/_4ex/ [r] _-{v} ^-{}="tgt"
	\ar@2 "src"!<-10pt,-10pt>;"tgt"!<-10pt,10pt> _-*+{f} ^-{}="srcA"
	\ar@2 "src"!<+10pt,-10pt>;"tgt"!<+10pt,10pt> ^-*+{g} _-{}="tgtA"
	\ar@3 "srcA"!<4pt,0pt> ; "tgtA"!<4pt,0pt> ^-{A}
&
\bullet
}
\qquad\text{or}\qquad
\xymatrix@!C@C=3em{
u
	\ar@2@/^3ex/ [r] ^{f} _{}="src"
	\ar@2@/_3ex/ [r] _{g} ^{}="tgt"
& 
v
\ar@3 "src"!<0pt,-10pt>;"tgt"!<0pt,10pt> ^-{A}
}
\]
that relates parallel $2$-cells $f$ and $g$ in the $(2,1)$-category $\tck{\Sigma}_2$. We will denote by $\tck{\Sigma}_3$ the free $(3,1)$-category generated by the $(3,1)$-polygraph $(\Sigma_2,\Sigma_3)$.
A pair $(f,g)$ of $2$-cells of $\tck{\Sigma}_2$ such that  
$s_1(f)=s_1(g)$ and~\mbox{$t_1(f)=t_1(g)$} is called a \emph{$2$-sphere} of $\tck{\Sigma}_2$.

\subsubsection{Coherent presentations of monoids}

An \emph{extended presentation} of a monoid $\M$ is a $(3,1)$-polygraph whose underlying $2$-polygraph is a presentation of the monoid $\M$. 
A \emph{coherent presentation of $\M$} is an extended presentation $\Sigma$ of $\mathbf{M}$ such that the cellular extension $\Sigma_3$ is a \emph{homotopy basis} of the $(2,1)$-category $\tck{\Sigma}_2$, that is, for every $2$-sphere $\gamma$ of $\tck{\Sigma}_2$, there exists a $3$-cell in $\tck{\Sigma}_3$ with boundary $\gamma$.

\subsubsection{Tietze transformations of \pdf{(3,1)}-polygraphs}

We recall the notion of Tietze transformation from~{\cite[Section 2.1]{GaussentGuiraudMalbos14}}. Let~$\Sigma$ be a $(3,1)$-polygraph. A $3$-cell~$A$ of~$\Sigma$ is called \emph{collapsible} if $t_2(A)$ is in $\Sigma_2$ and~$s_2(A)$ is a $2$-cell of the free $(2,1)$-category over $\tck{(\Sigma_2\setminus\{t_2(A)\})}$.
If~$A$ is a collapsible $3$-cell, then its target is called a \emph{redundant} cell. 
An \emph{elementary Tietze transformation} of a $(3,1)$-polygraph $\Sigma$ is a $3$-functor with domain $\tck{\Sigma}_3$ that belongs to one of the following operations:
\begin{enumerate}[{\bf i)}]
\item adjunction $\iota^1_\alpha$ and elimination $\pi_\alpha$ of a $2$-cell $\alpha$ as described in \ref{Subsubsection:TietzeTransformations2Polygraphs},
\item coherent adjunction $\iota^2_A : \tck{\Sigma}_3 \fl \tck{\Sigma}_3(\alpha)(A)$ of a redundant $2$-cell $\alpha$ with its collapsible $3$-cell $A$:
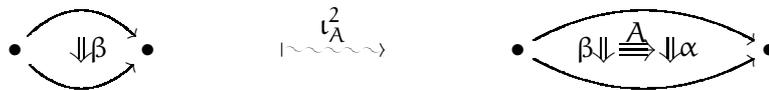
\[
\xymatrix@C=3.5em{
\bullet
	\ar@/^3ex/ [r] ^-{} ^{}="src"
	\ar@/_3ex/ [r] _-{} ^{}="tgt"
	\ar@2 "src"!<0pt,-10pt>;"tgt"!<0pt,10pt> ^-{\beta}
& \bullet
&
   \ar@{|~>} [r] ^{\iota_A^2}
&
&
\bullet
	\ar@/^3ex/ [rr] ^-{} ^{}="src"
	\ar@/_3ex/ [rr] _-{} ^{}="tgt"
	\ar@2 "src"!<-15pt,-10pt>;"tgt"!<-15pt,10pt> _-{\beta} ^-{}="srcA"
	\ar@2 "src"!<10pt,-10pt>;"tgt"!<10pt,10pt> ^-{\alpha} _-{}="tgtA"
	\ar@3 "srcA"!<6pt,0pt> ; "tgtA"!<-6pt,0pt> ^-{A}
&& \bullet
}
\]
\item coherent elimination $\pi_A : \tck{\Sigma}_3 \fl \tck{\Sigma}_3/A$ of a redundant $2$-cell $\alpha$ with its collapsible $3$-cell $A$:
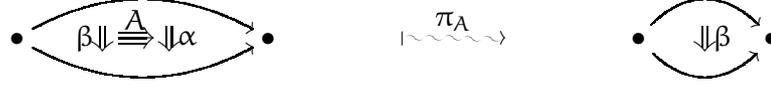
\[
\xymatrix@C=3.5em{
\bullet
	\ar@/^3ex/ [rr] ^-{} ^{}="src"
	\ar@/_3ex/ [rr] _-{} ^{}="tgt"
	\ar@2 "src"!<-15pt,-10pt>;"tgt"!<-15pt,10pt> _-{\beta} ^-{}="srcA"
	\ar@2 "src"!<10pt,-10pt>;"tgt"!<10pt,10pt> ^-{\alpha} _-{}="tgtA"
	\ar@3 "srcA"!<6pt,0pt> ; "tgtA"!<-6pt,0pt> ^-{A}
&& \bullet
&
   \ar@{|~>} [r] ^{\pi_A}
&
&
\bullet
	\ar@/^3ex/ [r] ^-{} ^{}="src"
	\ar@/_3ex/ [r] _-{} ^{}="tgt"
	\ar@2 "src"!<0pt,-10pt>;"tgt"!<0pt,10pt> ^-{\beta}
& \bullet
}
\] 
\item coherent adjunction $\iota_A : \tck{\Sigma}_3 \fl \tck{\Sigma}_3(A)$ of a redundant $3$-cell $A$:
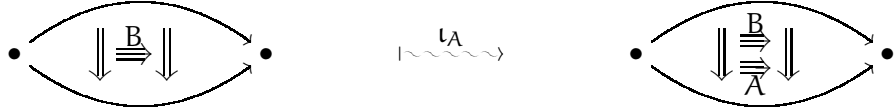
\[
\xymatrix@C=3.5em{
\bullet
	\ar@/^4ex/ [rr] ^-{} ^{}="src"
	\ar@/_4ex/ [rr] _-{} ^{}="tgt"
	\ar@2 "src"!<-15pt,-10pt>;"tgt"!<-15pt,10pt> ^-{}="srcA"
	\ar@2 "src"!<10pt,-10pt>;"tgt"!<10pt,10pt> _-{}="tgtA"
	\ar@3 "srcA"!<6pt,0pt> ; "tgtA"!<-6pt,0pt> ^-{B}
&& \bullet
&
   \ar@{|~>} [r] ^{\iota_A}
&
&
\bullet
	\ar@/^4ex/ [rr] ^-{} ^{}="src"
	\ar@/_4ex/ [rr] _-{} ^{}="tgt"
	\ar@2 "src"!<-15pt,-10pt>;"tgt"!<-15pt,10pt> ^-{}="srcA"
	\ar@2 "src"!<10pt,-10pt>;"tgt"!<10pt,10pt> _-{}="tgtA"
	\ar@3 "srcA"!<7pt,-5pt> ; "tgtA"!<-7pt,-5pt> _-{A}
             \ar@3 "srcA"!<7pt,5pt> ; "tgtA"!<-7pt,5pt> ^-{B}
&& \bullet
}
\]
\item coherent elimination $\pi_{(B,A)} : \tck{\Sigma}_3 \fl \tck{\Sigma}_3/(B,A)$ of a redundant $3$-cell $A$, that maps $A$ to $B$:
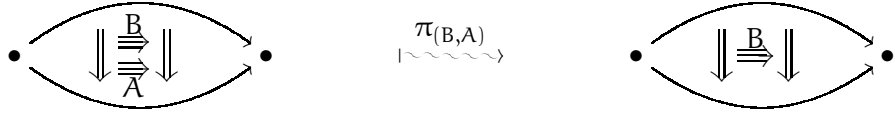
\[
\xymatrix@C=3.5em{
\bullet
	\ar@/^4ex/ [rr] ^-{} ^{}="src"
	\ar@/_4ex/ [rr] _-{} ^{}="tgt"
	\ar@2 "src"!<-15pt,-10pt>;"tgt"!<-15pt,10pt> ^-{}="srcA"
	\ar@2 "src"!<10pt,-10pt>;"tgt"!<10pt,10pt> _-{}="tgtA"
	\ar@3 "srcA"!<7pt,-5pt> ; "tgtA"!<-7pt,-5pt> _-{A}
             \ar@3 "srcA"!<7pt,5pt> ; "tgtA"!<-7pt,5pt> ^-{B}
&& \bullet
&
   \ar@{|~>} [r] ^{\pi_{(B,A)}}
&
&
\bullet
	\ar@/^4ex/ [rr] ^-{} ^{}="src"
	\ar@/_4ex/ [rr] _-{} ^{}="tgt"
	\ar@2 "src"!<-15pt,-10pt>;"tgt"!<-15pt,10pt> ^-{}="srcA"
	\ar@2 "src"!<10pt,-10pt>;"tgt"!<10pt,10pt> _-{}="tgtA"
	\ar@3 "srcA"!<6pt,0pt> ; "tgtA"!<-6pt,0pt> ^-{B}
&& \bullet
}
\]
\end{enumerate}

For $(3,1)$-polygraphs $\Sigma$ and $\Upsilon$, a \emph{Tietze transformation from $\Sigma$ to $\Upsilon$} is a $3$-functor $F:\tck{\Sigma}_3 \fl \tck{\Upsilon}_3$ that decomposes into a sequence of elementary Tietze transformations.
Two $(3,1)$-polygraphs $\Sigma$ and $\Upsilon$ are \emph{Tietze-equivalent} if there exists an equivalence of $2$-categories~\mbox{$F : \tck{\Sigma}_2/\Sigma_3 \fl \tck{\Upsilon}_2/\Upsilon_3$} and the presented monoids $\cl{\Sigma}_2$ and $\cl{\Upsilon}_2$ are isomorphic.
Two $(3,1)$-polygraphs are Tietze equivalent if, and only if, there exists a Tietze transformation between them, {\cite[Theorem 2.1.3.]{GaussentGuiraudMalbos14}}. 

\subsubsection{Homotopical completion procedure}
\label{Subsubsection:HomotopicalCompletion}

Following {\cite[Section~2.2]{GaussentGuiraudMalbos14}}, we recall the homotopical completion procedure that produces a coherent convergent presentation from a terminating presentation.
Given a terminating $2$-polygraph $\Sigma$, equipped with a total termination order $\leq$, the homotopical completion of $\Sigma$ is the $(3,1)$-polygraph obtained from~$\Sigma$ by successive application of the Knuth-Bendix completion procedure,~\cite{KnuthBendix70}, and the Squier construction,~\cite{Squier94}. Explicitly, for any critical branching $(f,g)$ of $\Sigma$, if~$(f,g)$ is confluent one adds a dotted $3$-cell $A$:
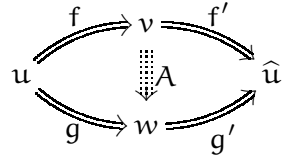
\[
\xymatrix @C=3em @R=0.6em {
& {v}
	\ar@2@/^/ [dr] ^-{f'}
	\ar@3{.>} []!<0pt,-10pt>;[dd]!<0pt,10pt> ^-{A}
\\
{u}
	\ar@2@/^/ [ur] ^-{f}
	\ar@2@/_/ [dr] _-{g}
&& {\rep{u}}
\\
& {w}
	\ar@2@/_/ [ur] _-{g'}
}
\]
where $\rep{u}$ is a common normal form of $v$ and $w$, and if the critical branching $(f,g)$ is not confluent one add a $2$-cell $\beta$ and a $3$-cell $A$:
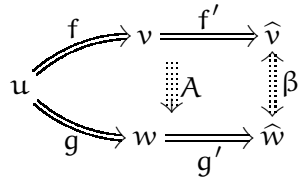
\[
\xymatrix @C=3em @R=0.6em {
& {v}
	\ar@2 [r] ^-{f'}
	\ar@3{.>} []!<10pt,-10pt>;[dd]!<10pt,10pt> ^-{A}
& {\rep{v}}
	\ar@2{<.>} [dd] ^-{\beta}
\\
{u}
	\ar@2@/^/ [ur] ^-{f}
	\ar@2@/_/ [dr] _-{g}
\\
& {w}
	\ar@2 [r] _-{g'}
& {\rep{w}}	
}
\]
where the $2$-cell $\beta$ is directed from a normal form~$\rep{v}$ of $v$ to a normal form $\rep{w}$ of $w$ if $\rep{v}>\rep{w}$ and from~$\rep{w}$ to $\rep{v}$ otherwise.
The adjunction of $2$-cells can create new critical branchings, possibly generating the adjunction of additional $2$-cells and $3$-cells in the same way. This defines an increasing sequence of $(3,1)$-polygraphs, whose union is called a \emph{homotopical completion} of $\Sigma$.
Following~{\cite[Theorem 5.2]{Squier94}}, such a homotopical completion of $\Sigma$ is a coherent convergent presentation of the monoid $\cl{\Sigma}$. We refer the reader to {\cite[Section 2.2]{GaussentGuiraudMalbos14}}.

\subsection{Column coherent presentation}
\label{SubSection:CoherentColumnPresentation}

Using the homotopical completion procedure, we extend the $2$-polygraph $\Colo_2(n)$ into a coherent presentation of the monoid $\P_n$.

\subsubsection{Column coherent presentation}

By definition of the rules $\alpha_{u,v}$ defined in~(\ref{rulesAlpha}), the presentation~$\Colo_2(n)$ has exactly one critical branching of the form
\begin{equation}
\label{CriticalBranchingCuCvCt}
\xymatrix @R=0.6em @C=3em {
& c_ec_{e'}c_t
\\
c_uc_vc_t
	\ar@2@/^/ [ur] ^-{\alpha_{u,v}c_t}
	\ar@2@/_/ [dr] _-{c_u\alpha_{v,t}}
\\
& c_uc_wc_{w'}
}
\end{equation}
for any $u$, $v$, $t$ in $\col(n)$ such that \typetroisbasep{$u$}{$v$}{$t$}{00}, where $e$ and $e'$ (resp. $w$ and $w'$) denote the two columns of the tableau $P(uv)$ (resp. $P(vt)$).
We prove in this section that all of these critical branchings are confluent. This gives an alternative proof of the confluence of the $2$-polygraph $\Colo_2(n)$ given in \ref{Subsubsection:ConfluenceColumnPresentation}.
Moreover, we prove that all the confluence diagrams of these branchings are of the following form:
\begin{equation}
\label{celluleX}
\xymatrix @!C @C=2.3em @R=0.6em {
&
c_{e} c_{e'}c_t
  \ar@2[r] ^-{c_{e}\alpha_{e',t}}
     \ar@3 []!<30pt,-8pt>;[dd]!<30pt,8pt> ^{\mathcal{X}_{u,v,t}}
&
c_{e} c_{b}c_{b'}
  \ar@2[dr] ^-{\;\;\;\;\alpha_{e,b}c_{b'}}
\\
c_uc_vc_t
  \ar@2[ur] ^-{\alpha_{u,v}c_t}
  \ar@2[dr] _-{c_u\alpha_{v,t}}
&&&
c_{a}c_{d} c_{b'}
\\
&
c_uc_{w}c_{w'}
  \ar@2[r] _-{\alpha_{u,w}c_{w'}}
&
c_{a}c_{a'}c_{w'}
 \ar@2[ur] _-{\;\;c_{a}\alpha_{a',w'}}
}	
\end{equation}
where~$a$ and $a'$ (resp. $b$ and $b'$) denote the two columns of the tableau $P(uw)$ (resp. $P(e't)$) and $a$, $d$, $b'$ are the three columns of the tableau $P(uvt)$, which is a normal form for the $2$-polygraph $\Colo_2(n)$. 
Note that in some cases described below, one or further columns $e'$, $w'$, $a'$ and $b'$ can be empty. In those cases some indicated $2$-cells $\alpha$ in the confluence diagram correspond to identities.

Let us denote by $\Colo_3(n)$ the extended presentation of the monoid $\P_n$ obtained from~$\Colo_2(n)$ by adjunction of one $3$-cell $\mathcal{X}_{u,v,t}$ of the form~(\ref{celluleX}), for every columns $u$, $v$ and $t$ such that~$\typetroisbasep{$u$}{$v$}{$t$}{00}.$

\begin{theorem}
\label{MainTheoremA}
For $n>0$, the $(3,1)$-polygraph $\Colo_3(n)$ is a coherent presentation of the monoid $\P_n$.
\end{theorem}

The extended presentation $\Colo_3(n)$ is called the \emph{column coherent presentation} of the monoid~$\P_n$. In~\ref{SubSection:Computations}, we give the values of number of cells of the coherent presentation $\Colo_3(n)$ for plactic monoids of low-dimensional rank~$n$. 
The rest of this section consists in the proof of Theorem~\ref{MainTheoremA}. It is based on the following arguments. The presentation $\Colo_2(n)$ is convergent, thus using the homotopical completion procedure described in~\ref{Subsubsection:HomotopicalCompletion}, it suffices to prove that the $3$-cells $\mathcal{X}_{u,v,t}$ with \typetroisbasep{$u$}{$v$}{$t$}{00} form a family of generating confluences for the presentation $\Colo_2(n)$.
There are four possibilities for the critical branching~(\ref{CriticalBranchingCuCvCt}) depending on the following four cases:
\[
\typetroisbasep{$u$}{$v$}{$t$}{11},
\qquad
\typetroisbasep{$u$}{$v$}{$t$}{21},
\qquad
\typetroisbasep{$u$}{$v$}{$t$}{12},
\qquad
\typetroisbasep{$u$}{$v$}{$t$}{22}.
\]
Each of these cases is examined in the following four lemmas.
In the rest of this section, we will suppose that
\[
u = x_{p}\ldots x_{1},
\qquad
v =y_{q}\ldots y_{1}
\qquad\text{and}\qquad
t = z_{l}\ldots z_{1} 
\]
denote columns of length $p$, $q$ and $l$ respectively.

\begin{lemma}
\label{Lemma1}
If \typetroisbasep{$u$}{$v$}{$t$}{11}, we have the following confluent critical branching:
\[
\xymatrix @!C @C=2.3em @R=0.6em {
&
c_{uv}c_t
   \ar@2[dr] ^-{\alpha_{uv,t}}
  \ar@3 []!<0pt,-10pt>;[dd]!<0pt,+10pt> ^{A_{u,v,t}} 
&
\\
c_uc_vc_t
  \ar@2[ur] ^-{\alpha_{u,v}c_t}
  \ar@2[dr] _-{c_u\alpha_{v,t}}
&&
c_{uvt}
\\
&
c_uc_{vt}
  \ar@2[ur] _-{\alpha_{u,vt}}
&
}	
\] 
\end{lemma}
\begin{proof}
By hypothesis $uv$ and $vt$ are columns, then $uvt$ is a column. Thus \typedeuxbase{$uv$}{$t$}{01} and \typedeuxbase{$u$}{$vt$}{01} and there exist $2$-cells $\alpha_{uv,t}$ and $\alpha_{u,vt}$ in $\Colo_2(n)$ making the critical branching~(\ref{CriticalBranchingCuCvCt}) confluent, where $e=uv$, $w=vt$ and $e'$, $w'$ are the empty column.
\end{proof}

\begin{lemma}
\label{Lemma3}
If \typetroisbasep{$u$}{$v$}{$t$}{21}, we have the following confluent critical branching:
\[
\xymatrix @!C @C=1.4em @R=0.6em {
&
c_{e} c_{e'}c_t
  \ar@2[rr] ^-{c_{e}\alpha_{e',t}}
&
\ar@3 []!<0pt,-10pt>;[dd]!<0pt,10pt> ^-{B_{u,v,t}} 
&
c_{e} c_{e't}
  \ar@2[dr] ^-{\alpha_{e,e't}}
&
\\
c_uc_vc_t
  \ar@2[ur] ^-{\alpha_{u,v}c_t}
  \ar@2[drr] _-{c_u\alpha_{v,t}}
&&&&
c_{s}c_{s'}
\\
&&
c_uc_{vt}
  \ar@2[urr] _-{\alpha_{u,vt}}
&&
}	
\]
where $e$ and $e'$ (resp. $s$ and $s'$) denote the two columns of the tableau $P(uv)$ (resp. $P(uvt)$).
\end{lemma}
\begin{proof}
By hypothesis, $vt$ is a column and~$y_{1}>z_{l}$. 
The tableau $P(uv)$ consists of two columns, that we will denote $e$ and $e'$, then $\lnds(uv)=2$ and~$x_{1}\leqslant y_{q}$.
We have \typedeuxbase{$u$}{$v$}{02}, so that we distinguish the   following possible three cases.

\medskip

\noindent {\bf Case 1:} $p\geq q$ and $x_{i_0}>y_{i_0}$ for some $1\leq i_0\leq q$.

Suppose that $i_0=1$, that is, $x_{1}>y_{1}$. We consider  $y_{j}$ the biggest element of the column $v$ such that~$x_{1}>y_{j}$, then the smallest element of the column~$e'$ is $y_{j+1}$.
By hypothesis, the word $vt$ is a column, in particular $y_{j+1}>z_{l}$. It follows that $e't$ is a column.
Suppose that $i_0>1$, then $x_{1}\leq y_{1}$ and the smallest element of $e'$ is $y_{1}$. 
Since~$y_{1}>z_{l}$ by hypothesis, the word $e't$ is a  column.
Hence, in all cases, $e't$ is a column and there is a $2$-cell $\alpha_{e',t}:c_{e'}c_t \dfl c_{e't}$.

\medskip

\noindent {\bf Case 2:} $p < q$ and $x_{i}\leq y_{i}$ for any $1\leq i\leq p$.

We have $e=y_{q}\ldots y_{p+1}x_{p}\ldots x_{1}$ and $e'=y_{p}\ldots y_{1}$.
By hypothesis, $y_{1}>z_{l}$, hence $e't$ is a column and there is a $2$-cell $\alpha_{e',t}: c_{e'}c_t \dfl c_{e't}$.

\medskip

\noindent {\bf Case 3:} $p < q$ and $x_{i_0}>y_{i_0}$ for some $1\leq i_0\leq p$.

With the same arguments of Case~1, the smallest element of $e'$ is $y_{1}$ or $y_{j+1}$, where~$y_{j}$ is the biggest element of  the column $v$ such that $y_{j}<x_{1}$. Hence, $e't$ is a column and there is a $2$-cell~$\alpha_{e',t} : c_{e'}c_t \dfl c_{e't}$.

\medskip

In case $1$,~$2$ and~$3$,  we have $\lnds(uv)=2$, hence $\lnds(uvt)=2$. Thus the tableau $P(uvt)$ consists of two columns, that we denote $s$ and $s'$ and there is a $2$-cell $\alpha_{u,vt}:c_u c_{vt} \dfl c_{s}c_{s'}$.
Moreover, to compute the tableau $P(uvt)$, one begins by computing $P(uv)$ and after by introducing the elements of the column $t$ on the tableau $P(uv)$. As $C(P(uv))=ee'$, we have $P(uvt) = P(P(uv)t) = P(ee't)$. Hence~$C(P(ee't))= ss'$ and there is a $2$-cell $\alpha_{e,e't}$ which yields the following confluence diagram
\begin{equation}
\label{figure1}
\xymatrix @C=4.2em @R=-4.5em @!C{
& {\begin{tabular}{|c|c|}
\hline
&\\
&\\
\mbox{\tiny$e$}&\mbox{\tiny$e'$}\\
& \\
\cline{2-2}
\\
\cline{1-1}
\end{tabular}
\;
\raisebox{0.35cm}{
\begin{tabular}{|c|}\hline
\mbox{\tiny$z_1$}\\ \hline
$\vdots$\\\hline
\mbox{\tiny$z_l$}\\\hline
\end{tabular}}}
  \ar@2[r] ^-{c_{e}\alpha_{e',t}}
&{\raisebox{0.35cm}{
  \begin{tabular}{|c|}\hline
  \\
  \\
 \mbox{\tiny$e$}
  \\
  \\ 
\\\hline
\end{tabular}}
\quad
\begin{tabular}{|c|}\hline
\mbox{\tiny$z_1$}\\\hline
$\vdots$\\\hline
\mbox{\tiny$z_l$}\\\hline
\\
\mbox{\tiny$e'$}
\\
\\
\hline
\end{tabular}}
\ar@2@/^10ex/[dd] ^-{\alpha_{e,e't}}
\\
{\begin{tabular}{|c|}\hline
\mbox{\tiny$x_1$}\\ \hline
\\
$\vdots$\\
\\
\hline
\mbox{\tiny$x_p$}\\\hline
\end{tabular}
\quad
\begin{tabular}{|c|}\hline
\mbox{\tiny$y_1$}\\ \hline
\\
$\vdots$\\
\\
\hline
\mbox{\tiny$y_q$}\\\hline
\end{tabular}
\quad
\raisebox{0.4cm}{
\begin{tabular}{|c|}\hline
\mbox{\tiny$z_1$}\\ \hline
$\vdots$\\\hline
\mbox{\tiny$z_l$}\\\hline
\end{tabular}}}
\ar@2@/^/ [ur] ^-{\alpha_{u,v}c_{t}}
 \ar@2@/_/ [dr] _-{c_{u}\alpha_{v,t}}
\\
& {\begin{tabular}{|c|}\hline
\mbox{\tiny$x_1$}\\ \hline
$\vdots$\\\hline
\mbox{\tiny$x_p$}\\\hline
\end{tabular}
\quad
\begin{tabular}{|c|}\hline
\mbox{\tiny$z_1$}\\ \hline
\\
$\vdots$\\
\\
\hline
\mbox{\tiny$z_l$}\\\hline
\mbox{\tiny$y_1$}\\ \hline
\\
$\vdots$\\
\\
\hline
\mbox{\tiny$y_q$}\\\hline
\end{tabular}}
 \ar@2[r] _-{\alpha_{u,vt}}
&{\begin{tabular}{|c|c|}
\hline
&\\
&\\
\mbox{\tiny$s$}&\mbox{\tiny$s'$}\\
& \\
\cline{2-2}
\\
\cline{1-1}
\end{tabular}}
}
\end{equation}
\end{proof}

\begin{lemma}
\label{Lemma2}
If \typetroisbasep{$u$}{$v$}{$t$}{12}, we have the following confluent critical branching:
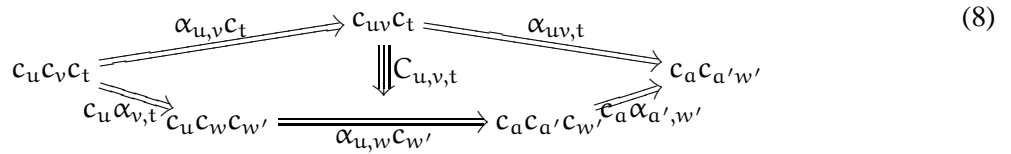
\begin{equation}
\label{CelluleCuvt}
\xymatrix @!C @C=1.4em @R=0.6em {
&&
c_{uv}c_t
     \ar@3 []!<0pt,-10pt>;[dd]!<0pt,10pt> ^{C_{u,v,t}} 
  \ar@2[drr] ^{\alpha_{uv,t}}
&&
\\
c_uc_vc_t
  \ar@2[urr] ^{\alpha_{u,v}c_t}
  \ar@2[dr] _{c_u\alpha_{v,t}}
&&&&
c_{a}c_{a'w'}
\\
&
c_uc_{w}c_{w'}
  \ar@2[rr] _{\alpha_{u,w}c_{w'}}
&&
c_{a}c_{a'}c_{w'}
 \ar@2[ur] _{c_{a}\alpha_{a',w'}}
&
}	
\end{equation}
where $w$ and $w'$ (resp. $a$ and $a'$) denote the two columns of the tableau $P(vt)$ (resp. $P(uw)$).
\end{lemma}

\begin{proof}
By hypothesis, $uv$ is a column hence $x_{1}>y_{q}$. Moreover, the tableau $P(vt)$ consists of two columns $w$ and $w'$, then $\lnds(vt)=2$, hence $y_{1}\leqslant z_{l}$. 
We have \typedeuxbase{$v$}{$t$}{02}, so that we distinguish the three possible following cases.

\medskip

\noindent {\bf Case 1:} $q\geq l$ and $y_{i_{0}}>z_{i_{0}}$  for some  $1\leq i_{0}\leq l$.

Let us denote $w=w_r\ldots w_1$ and $w'=w'_{r'}\ldots w_1'$.
Since $q\geq l$,  we have $w_r=y_{q}$. By hypothesis,~$x_{1}>y_{q}$. Then the word $uw$ is a column. As a consequence, there is a $2$-cell $\alpha_{u,w}:c_uc_w  \dfl c_{uw}$.
In addition, the column $w$ appears to the left of $w'$ in the planar representation of the tableau $P(vt)$, that is,~$\len(w) \geq \len(w')$ and $w_{i}\leq w_{i}'$ for any $i\leq \len(w')$. Then $\len(uw)\geq \len(w')$. We set $uw=\xi_{\len(uw)}\ldots \xi_{1}$ and we have $\xi_{i}\leq w_{i}'$ for any $i\leq \len(w')$. 
Then \typedeuxbase{$uw$}{$w'$}{1} and $c_{uw}c_{w'}$ is a normal form.

On the other hand, the tableau  $P(vt)$ consists of two columns, hence $\lnds(vt)=2$. As a consequence,~$\lnds(uvt)=2$ and the tableau $P(uvt)$ consists of two columns.
Since $q\geq l$, we have $C(P(uvt))=uww'$, hence the two columns of $P(uvt)$ are $uw$ and $w'$. Then there is a $2$-cell $\alpha_{uv,t}:c_{uv}c_{t} \dfl c_{uw}c_{w'}$ which yields the confluence of the critical branching on $c_uc_vc_t$, as follows

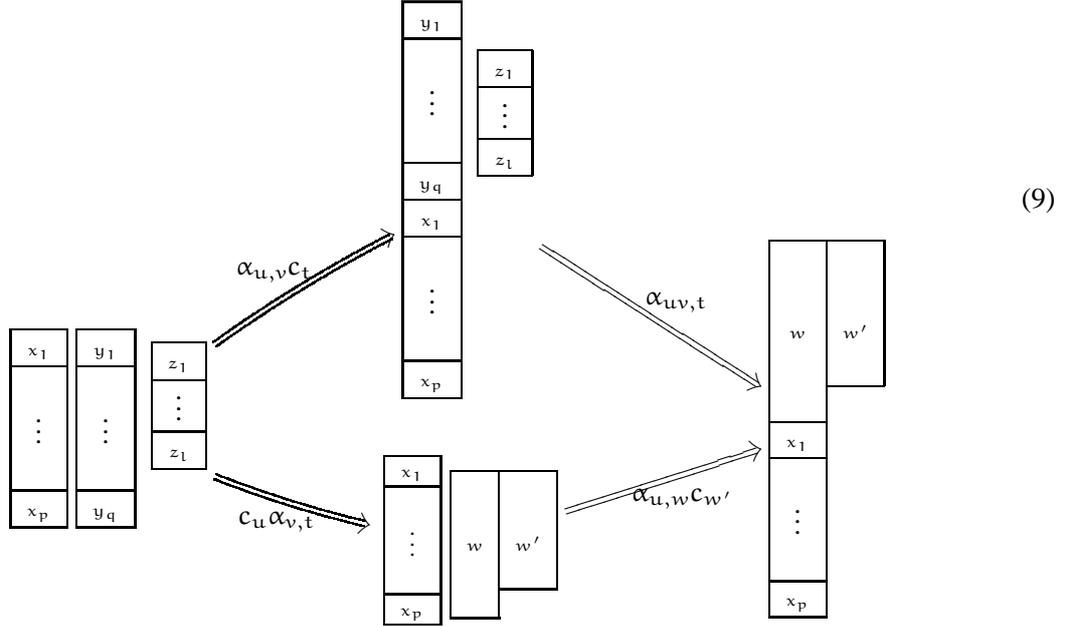
\begin{equation}
\label{figure2}
\xymatrix @C=5em @R=-6em @!C{
& 
{\begin{tabular}{|c|}\hline
\mbox{\tiny$y_1$}\\ \hline
\\
$\vdots$\\
\\
\hline
\mbox{\tiny$y_q$}\\\hline
\mbox{\tiny$x_1$}\\ \hline
\\
$\vdots$\\
\\
\hline
\mbox{\tiny$x_p$}\\\hline
\end{tabular}
\;
\raisebox{1.15cm}{
\begin{tabular}{|c|}\hline
\mbox{\tiny$z_1$}\\ \hline
$\vdots$\\\hline
\mbox{\tiny$z_l$}\\\hline
\end{tabular}}}
  \ar@2[dr] ^-{\mbox{\small$\alpha_{uv,t}$}}
\\
{\begin{tabular}{|c|}\hline
\mbox{\tiny$x_1$}\\ \hline
\\
$\vdots$\\
\\
\hline
\mbox{\tiny$x_p$}\\\hline
\end{tabular}
\;
\begin{tabular}{|c|}\hline
\mbox{\tiny$y_1$}\\ \hline
\\
$\vdots$\\
\\
\hline
\mbox{\tiny$y_q$}\\\hline
\end{tabular}
\;
\raisebox{0.3cm}{
\begin{tabular}{|c|}\hline
\mbox{\tiny$z_1$}\\ \hline
$\vdots$\\\hline
\mbox{\tiny$z_l$}\\\hline
\end{tabular}}}
\ar@2@/^/ [ur] ^-{\mbox{\small$\alpha_{u,v}c_t$}}
 \ar@2@/_/ [dr] _-{\mbox{\small$c_u\alpha_{v,t}$}}
&&{\begin{tabular}{|c|c|}
\hline
&\\
&\\
\mbox{\tiny$w$}&\mbox{\tiny$w'$}\\
& \\
\cline{2-2}
\\ 
\cline{1-1}
\mbox{\tiny$x_1$}\\ \cline{1-1}
\\
$\vdots$\\
\\
\cline{1-1}
\mbox{\tiny$x_p$}\\\cline{1-1}
\end{tabular}}
\\
& {\footnotesize \raisebox{0.05cm}{
\begin{tabular}{|c|}
\hline
\mbox{\tiny$x_1$}\\
\hline
\\
$\vdots$\\
\\
\hline
\mbox{\tiny$x_p$}\\
\hline
\end{tabular}}
\;
\begin{tabular}{|c|c|}
\hline
&\\
&\\
\mbox{\tiny$w$}&\mbox{\tiny$w'$}\\
& \\
\cline{2-2}
\\ 
\cline{1-1}
\end{tabular}}
 \ar@2[ur] _-{\mbox{\small$\alpha_{u,w}c_{w'}$}}
}
\end{equation}

\medskip

\noindent {\bf Case 2:} $q<l$ and  $y_{i} \leq z_{i} $ for any $i\leq q$. 

We have $w=z_{l}\ldots z_{q+1}y_{q}\ldots y_{1}$ and $w'=z_{q}\ldots z_{1}$.
There are two cases along 
\[
uw=x_{p}\ldots x_{1}z_{l}\ldots z_{q+1}y_{q}\ldots y_{1} 
\]
is a column or not.

\medskip

{\bf Case 2. A.} If $x_{1}>z_{l}$, then $uw$ is a column. Hence, there is a $2$-cell $\alpha_{u,w}: c_uc_w \dfl c_{uw}$. Moreover, using Schensted's algorithm we prove that $C(P(uvt)) = uww'$, it follows that the columns of $P(uvt)$ are $uw$ and $w'$. Thus there is a $2$-cell $\alpha_{uv,t}:c_{uv}c_t \dfl c_{uw}c_{w'}$ which yields the confluence diagram~(\ref{figure2}).

\medskip

{\bf Case 2. B.} If $x_{1}\leq z_{l}$, then $\lnds(uw)=2$ and $P(uw)$ consists of two columns, that we denote by~$a$ and $a'$. Then there is a $2$-cell $\alpha_{u,w} : c_uc_w \dfl c_{a}c_{a'}$. 
In addition, by Schensted's algorithm, we deduce  that~ $a'=z_{i_{k}}\ldots z_{i_{1}}$, with $q+1\leq i_1 < \ldots < i_k \leq l$.  We have $a'w' = z_{i_{k}}\ldots z_{i_{1}}z_{q}\ldots z_{1}$. Since all the elements of $a'$ are elements of $t$ and bigger than $z_{q}$,  we have $z_{i_{1}}>z_{q}$. It follows that $a'w'$ is a column and there is a $2$-cell $\alpha_{a',w'}:c_{a'}c_{w'} \dfl c_{a'w'}$.

In the other hand, we have two cases whether \typedeuxbase{$uv$}{$t$}{0} or $\typedeuxbase{$uv$}{$t$}{1}.$
Suppose \typedeuxbase{$uv$}{$t$}{0}. By Schensted's algorithm, we have $C(P(uvt)) = aa'w'$, showing that the two columns of $P(uvt)$ are $a$ and $a'w'$.  Hence there is a $2$-cell $\alpha_{uv,t}:c_{uv}c_t \dfl c_ac_{a'w'}$, which yields the confluence of Diagram~(\ref{CelluleCuvt}).
Suppose~$\typedeuxbase{$uv$}{$t$}{1}.$  Then we obtain $C(P(uw))= uvz_{l}\ldots z_{q+1}$, and $C(P(z_{l}\ldots z_{q+1}w')) = t$.
Hence there is a $2$-cell~$\alpha_{z_{l}\ldots z_{q+1}, w'}$ yielding the confluence diagram
\[
\xymatrix @C=3em @R=0.6em {
& {c_{uv}c_{t}}
\ar@3 []!<0pt,-10pt>;[dd]!<0pt,10pt> ^{C'_{u,v,t}}
\\
{c_{u}c_{v}c_{t}}
	\ar@2@/^/ [ur] ^-{\alpha_{u,v}c_{t}}
	\ar@2@/_/ [dr] _-{c_{u}\alpha_{v,t}}
\\
& {c_{u}c_{w}c_{w'}}
	\ar@2 [r] _-{\alpha_{u,w}c_{w'}}
& {c_{uv}c_{z_{l}\ldots z_{q+1}}c_{w'}}
	\ar@2 [luu] _-{\;\;c_{uv}\alpha_{z_{l}\ldots z_{q+1}, w'}}
}
\]

\medskip

\noindent {\bf Case 3:} $q<l$ and  $y_{i_{0}}>z_{i_{0}}$  for some $1\leq i_{0}\leq q$.

We compute the columns $w$ and $w'$ of the tableau $P(vt)$. If the biggest element of the column $w$ is~$y_{q}$, then we obtain the same confluent branching as in Case 1. If the first element of $w$ is~$z_{l}$, then one obtains the same confluent critical branchings as in Case 2.
\end{proof}

\begin{lemma}
\label{Lemma4}
If \typetroisbasep{$u$}{$v$}{$t$}{22}, we have the following confluent critical branching:
\begin{equation}
\label{celluleDuvt}
\xymatrix @!C @C=2.3em @R=0.6em {
&
c_{e} c_{e'}c_t
  \ar@2[r] ^{c_{e}\alpha_{e',t}}
     \ar@3 []!<40pt,-10pt>;[dd]!<40pt,10pt> ^{D_{u,v,t}} 
&
c_{e} c_{b}c_{b'}
  \ar@2[dr] ^{\alpha_{e,b}c_{b'}}
\\
c_uc_vc_t
  \ar@2[ur] ^{\alpha_{u,v}c_t}
  \ar@2[dr] _{c_u\alpha_{v,t}}
&&&
c_{a}c_{d} c_{b'}
\\
&
c_uc_{w}c_{w'}
  \ar@2[r] _{\alpha_{u,w}c_{w'}}
&
c_{a}c_{a'}c_{w'}
 \ar@2[ur] _{c_{a}\alpha_{a',w'}}
}	
\end{equation}
where $e$, $e'$ (resp. $w$, $w'$) denote the two columns of the tableau $P(uv)$ (resp. $P(vt)$) and $a$, $a'$ (resp. $b$, $b'$) denote the two columns of the tableau $P(uw)$ (resp. $P(e't)$).
\end{lemma}
\begin{proof}
By hypothesis,  $\lnds(uv)=2$ and $\lnds(vt)=2$, hence $x_{1}\leqslant y_{q}$ and $y_{1}\leqslant z_{l}$. 
In addition, since~\typedeuxbase{$u$}{$v$}{02}, the tableau $P(uw)$ consists of two columns, that we denote by $a$ and $a'$.  Thus there is a  $2$-cell $\alpha_{u,w}:c_uc_w \dfl c_ac_{a'}$.
Moreover, as \typedeuxbase{$u$}{$v$}{02} and \typedeuxbase{$v$}{$t$}{02}, we have 

\begin{center}
(($p<q$)  or ($x_{i_0}> y_{i_0}$ for some $i_0\leqslant q$))
\quad 
and 
\quad
(($q<l$)  or ($y_{j_0}> z_{j_0}$ for some $j_0\leqslant l$)), 
\end{center}

\noindent thus we consider the following cases.

\medskip

\noindent {\bf Case 1:}  $p<q < l$ and  $y_{i}\leqslant z_{i}$, for all $i\leqslant q$, and $x_{i}\leqslant y_{i}$, for all $i\leqslant p$.

We have
\[
w=~z_{l}\ldots z_{q+1}y_{q}\ldots y_{1},
\quad
w'=~z_{q}\ldots z_{1},
\quad
e=~y_{q}\ldots y_{p+1}x_{p}\ldots x_{1}
\quad\text{and}\quad
e'=~y_{p}\ldots y_{1}.
\]
Since  $z_{l}\geqslant y_{1}$, the tableau  $P(e't)$ consists of two columns, that we denote by $b$ and $b'$. Thus there is a $2$-cell $\alpha_{e',t}:~c_{e'}c_t\dfl c_bc_{b'}$. In addition, we have
\[
b=~z_{l}\ldots z_{p+1}y_{p}\ldots y_{1},
\quad
b'= z_{p}\ldots z_{1},
\quad
a=~z_{l}\ldots z_{q+1}y_{q}\ldots y_{p+1}x_{p}\ldots x_{1}
\quad\text{and}\quad
a'=~y_{p}\ldots y_{1}.
\] 
Since $z_{q}\geqslant y_{1}$, the tableau $P(a'w')$ consists of two columns, that we denote by $d$ and $d'$. Thus there is a  $2$-cell $\alpha_{a',w'}:c_{a'}c_{w'}\dfl c_dc_{d'}$. Since $z_{l}\geqslant x_{1}$,  the tableau $P(e b)$ consists of two columns, that we denote by $s$ and $s'$. Then  there is a  $2$-cell $\alpha_{e,b}:c_ec_b\dfl~c_sc_{s'}$.
In the other hand, we have
\[
d=z_{q}\ldots z_{p+1}y_{p}\ldots y_{1},
\;
d'=z_{p}\ldots z_{1}, 
\;
s=z_{l}\ldots z_{q+1}y_{q}\ldots y_{p+1}x_{p}\ldots x_{1}
\;\text{and}\;
s'=z_{q}\ldots z_{p+1}y_{p}\ldots y_{1}.
\]
Hence $a = s$, $d = s'$ and  $d' = b'$ which yields the confluence diagram~(\ref{celluleDuvt}).

\medskip

\noindent {\bf Case 2:}
$\left\{
\begin{array}{l}
q < l \;\; \text{and}\; y_{i}\leqslant z_{i}\;\text{for all}\; i\leqslant q\\ 
p\geqslant q \;\text{and}\; x_{i_0}> y_{i_0}\;\text{for some}\;\ i_0\leqslant q
\end{array}\right.$
or
$\left\{\begin{array}{l}
q < l\;\; \text{and} \;y_{i}\leqslant z_{i}\;\text{for all}\; i\leqslant q\\
p < q\;\; \text{and} \; x_{i_0}>y_{i_0}\;\text{for some}\; i_0\leqslant p\end{array}\right.$ 

\medskip

We have $w=z_{l}\ldots z_{q+1}y_{q}\ldots y_{1}$ and $w'=z_{q}\ldots z_{1}$. Using Schensted's algorithm  the smallest element  of the column $a'$ is an element of $v$. Since $z_{q}$ is greater or equal than each element of $v$,  the tableau $P(a'w')$ consists of two columns, that we denote by~$d$ and~$d'$.

On the other hand, all the elements  of $e'$ are elements of  $v$. Since $z_{l}$ is bigger than each element of  $v$, the tableau $P(e't)$ consists of two columns, that we denote by $b$ and $b'$. Thus there is a  $2$-cell~$\alpha_{e',t}:c_{e'}c_t\dfl c_{b}c_{b'}$. Hence, we consider two cases depending on whether or not $c_{e}c_{b}c_{b'}$ is a tableau. Suppose $c_{e}c_{b}c_{b'}$ is a tableau. The column $e$ does not contain elements from the column $t$,  then during inserting the column $w$ into the column $u$, we can only insert some elements of $y_{q}\ldots y_{1}$ into $u$ and we obtain $a=e$. Since~$c_{e}c_{b}c_{b'}$ is the unique tableau obtained from~$c_uc_vc_t$ and~$a=e$, we obtain~$C(P(a'w')) = bb'$.  As a consequence, there is a  $2$-cell $\alpha_{a',w'}:c_{a'}c_{w'}\dfl c_bc_{b'}$ yielding the following confluence diagram:

\begin{equation}
\label{3-cellDprime}
\xymatrix @!C @C=2.3em @R=0.6em {
&
c_{e} c_{e'}c_t
  \ar@2[r] ^{c_{e}\alpha_{e',t}}
     \ar@3 []!<10pt,-10pt>;[dd]!<10pt,10pt> ^{D^{(1)}_{u,v,t}} 
&
c_{e} c_{b}c_{b'}
\\
c_uc_vc_t
  \ar@2[ur] ^{\alpha_{u,v}c_t}
  \ar@2[dr] _{c_u\alpha_{v,t}}
&&&
\\
&
c_uc_{w}c_{w'}
  \ar@2[r] _{\alpha_{u,w}c_{w'}}
&
c_{a}c_{a'}c_{w'}
 \ar@2[uu] _{c_{a}\alpha_{a',w'}}
}
\end{equation}
Suppose $c_{e}c_{b}c_{b'}$ is not a tableau. The first element of the column $b$ is $z_{l}$. The smallest element of the column $e$ is either $x_{1}$ or $y_{j}$, where $y_{j}$ is the biggest element of the column $v$ such that $y_{j}<x_{1}$. By hypothesis the tableau $P(uw)$ consists of two columns, then $x_{1}\leq z_{l}$. In addition, $z_{l}$ is greater than each element of $v$ then $y_{j}\leq z_{l}$. Hence, in all cases, the tableau $P(eb)$ consists of two columns. On the other hand, using Schensted's algorithm,  we have $a'= z_{i_{k}}\ldots z_{i_{1}}y_{j_{k'}}\ldots y_{j_{1}}$ with $q+1\leq i_{1}<\ldots <i_{k}\leq l$, $1\leq j_{1}< \ldots <j_{k'}\leq q$ and we have  $e'= y_{j_{k'}}\ldots y_{j_{1}}$. In addition, we have $b'=d'=z_{i_{k''}}\ldots z_{i_{1}}$ with $1\leq i_{1}<\ldots <i_{k''}\leq q$ and $C(P(eb))=ad$. Hence there is a $2$-cell $\alpha_{e,b}:c_{e}c_{b}\dfl c_{a}c_{d}$ which yields the confluence diagram~(\ref{celluleDuvt}).

\medskip

\noindent {\bf Case 3:}
$\left\{
\begin{array}{l}
 q\geqslant l    \;\; \text{and}\; y_{i_0}> z_{i_0}\;\text{for some }\; i_0\leqslant l\\ 
  p < q     \;\; \text{and}\; x_{i}\leqslant  y_{i}\; \text{for all}\;\ i\leqslant p
\end{array}\right.$
or
\;
$\left\{\begin{array}{l}
q < l\;\; \text{and} \; y_{i_0} > z_{i_0}\;\text{for some}\;\; i_0\leqslant q\\
p < q\;\; \text{and} \; x_{i}\leqslant y_{i}\;\text{for all }\;\; i\leqslant p
\end{array}\right.$ 

We have $e = y_{q}\ldots y_{p+1}x_{p}\ldots x_{1}$ and  $e'= y_{p}\ldots y_{1}$. Since  $y_{1}\leqslant z_{l}$, the tableau $P(e't)$ consists of two columns, that we denote by $b$ and $b'$. The first element of the column $b$ is either $z_{l}$ or~$y_{p}$ which are bigger or equal to $x_{1}$, then the tableau $P(eb)$ consists of two columns, that we denote by $s$ and $s'$. Suppose $l\leq p$. Then by Schensted's insertion algorithm, we have $C(P(e't))= bw'$ and $w=y_{q}\ldots y_{p+1}b$. On the other hand, since $x_{p}<y_{p+1}$, we have $P(uw) = P(u(y_{q}\ldots y_{p+1}b))= P(eb)$. Hence, there is a $2$-cell $\alpha_{e,b}:c_{e}c_{b}\dfl c_{a}c_{a'}$ which yields the confluence diagram: 
\begin{equation}
\label{3-cellD2prime}
\xymatrix @!C @C=2.3em @R=0.6em {
&
c_{e} c_{e'}c_t
  \ar@2[r] ^{c_{e}\alpha_{e',t}}
     \ar@3 []!<10pt,-10pt>;[dd]!<10pt,10pt> ^{D^{(2)}_{u,v,t}} 
&
c_{e} c_{b}c_{w'}
  \ar@2[dd] ^{\alpha_{e,b}c_{w'}}
\\
c_uc_vc_t
  \ar@2[ur] ^{\alpha_{u,v}c_t}
  \ar@2[dr] _{c_u\alpha_{v,t}}
&&&
\\
&
c_uc_{w}c_{w'}
  \ar@2[r] _{\alpha_{u,w}c_{w'}}
&
c_{a}c_{a'}c_{w'}
}
\end{equation}
Suppose~$l>p$, then we consider two cases depending on whether or not the first element of the column~$b$ is $y_p$. If this element is~$y_p$, then when computing the tableau $P(vt)$ no element of the column~$t$ is inserted in~$y_q\ldots y_{p+1}$. Hence we have~$w = y_q\ldots y_{p+1}b$ and~$b'=w'$. On the other hand, by Schensted's insertion procedure we have~$P(uw) = P(eb)$. Hence, there is a $2$-cell $\alpha_{e,b}:c_{e}c_{b}\dfl c_{a}c_{a'}$ which yields the confluence diagram~(\ref{3-cellD2prime}). Suppose that the first element of the column~$b$ is~$z_l$. Then when computing the tableau~$P(vt)$ some elements of the column~$t$ are inserted in~$y_q\ldots y_{p+1}$. In this case, we have that the column~$w'$ contains more elements than~$b'$ and that~$c_{s}c_{s'}c_{b'}$ is a tableau. Moreover, by Schensted's insertion procedure, we have~$a=s$. Since $c_{s}c_{s'}c_{b'}$ is the unique tableau obtained from~$c_uc_vc_t$ and~$a=s$, we obtain that~$C(P(a'w')) = s'b'$. As a consequence, there is a $2$-cell~$\alpha_{a',w'}: c_{a'}c_{w'}\dfl c_{s'}c_{b'}$ which yields the confluence diagram~(\ref{celluleDuvt}).
\medskip

\noindent {\bf Case 4:}
$\left\{
\begin{array}{l}
q\geqslant l    \; \text{and}\; y_{i_0}> z_{i_0} \;\text{for some }\;\; i_0\leqslant l\\ 
p\geqslant q    \; \text{and}\;x_{j_0}>y_{j_0}\; \text{for some}\;\ j_0\leqslant q
\end{array}\right.$
or
\;
$\left\{\begin{array}{l}
q\geqslant l\;\; \text{and} \; y_{i_0} > z_{i_0}\;\text{for some}\;i_0\leqslant q\\
p < q\;\; \text{and} \; x_{j_0}> y_{j_0}\;\text{for some }\; j_0\leqslant p
\end{array}\right.$ 
\\
\\ or \; $\left\{
\begin{array}{l}
q < l    \; \text{and}\; y_{i_0}> z_{i_0} \;\text{for some }\; i_0\leqslant q,\\ 
p\geqslant q    \; \text{and}\;x_{j_0}>y_{j_0}\; \text{for some}\;\ j_0\leqslant q.
\end{array}\right.$
or
\;
$\left\{\begin{array}{l}
q < l\;\; \text{and} \;\; y_{i_0} > z_{i_0}\;\text{for some}\;i_0\leqslant q\\
p < q\;\; \text{and} \;\; x_{j_0}> y_{j_0}\;\text{for some }\;j_0\leqslant p
\end{array}\right.$ 

\medskip

By Lemma~\ref{Lemma3}, the last term of $e'$ is $y_{1}$ or $y_{j+1}$, where $y_{j}$ is the biggest element of $v$ such that~ $y_{j}<x_{1}$. Suppose that the last term of $e'$ is $y_{1}$. Since $z_{l}\geq y_{1}$, the tableau $P(e't)$ consists of two columns.
Furthermore, if the last term of $e'$ is $y_{j+1}$, then we consider two cases: $z_{l}\geqslant y_{j+1} \text{ or } z_{l} < y_{j+1}$. Suppose $z_{l}<y_{j+1}$, then the tableau $P(e't)$ consists of one column $e't$. We consider two cases depending on whether or not $c_{e}c_{e't}$ is a tableau. With the same arguments of Case~2, we obtain a confluence diagram of the following forms: 

\[
\xymatrix @!C @C=2.3em @R=0.6em {
&
c_{e} c_{e'}c_t
  \ar@2[r] ^{c_{e}\alpha_{e',t}}
     \ar@3 []!<10pt,-10pt>;[dd]!<10pt,10pt> ^{D^{(3)}_{u,v,t}} 
&
c_{e} c_{e't}
\\
c_uc_vc_t
  \ar@2[ur] ^{\alpha_{u,v}c_t}
  \ar@2[dr] _{c_u\alpha_{v,t}}
&&&
\\
&
c_uc_{w}c_{w'}
  \ar@2[r] _{\alpha_{u,w}c_{w'}}
&
c_{e}c_{a'}c_{w'}
 \ar@2[uu] _{c_{e}\alpha_{a',w'}}
}
\!\!\!\!\!\!\!\!\!\!\!\!\!\!\!\!\!
\xymatrix @!C @C=2.3em @R=0.6em {
&
c_{e} c_{e'}c_t
  \ar@2[r] ^{c_{e}\alpha_{e',t}}
     \ar@3 []!<40pt,-10pt>;[dd]!<40pt,10pt> ^{D^{(4)}_{u,v,t}} 
&
c_{e} c_{e't}
  \ar@2[dr] ^{\alpha_{e,e't}}
\\
c_uc_vc_t
  \ar@2[ur] ^{\alpha_{u,v}c_t}
  \ar@2[dr] _{c_u\alpha_{v,t}}
&&&
c_{a}c_{a'w'}
\\
&
c_uc_{w}c_{w'}
  \ar@2[r] _{\alpha_{u,w}c_{w'}}
&
c_{a}c_{a'}c_{w'}
 \ar@2[ur] _{c_{a}\alpha_{a',w'}}
}
\]	
Suppose the tableau $P(e't)$ consists of two columns. Using the same arguments of Case~2 and Case~3, we obtain a confluence diagram of the form $D_{u,v,t}$, $D^{(1)}_{u,v,t}$ or $D^{(2)}_{u,v,t}$.
\end{proof}

\begin{remark}
\label{Subsection:QuadraticNormalisation}

In the proof of Theorem~\ref{MainTheoremA}, we don't use the fact that the~$2$-polygraph~$\Colo_2(n)$ is convergent. Using the notion of quadratic normalisation of monoids introduced in~\cite{GuiraudDehornoy15}, our construction allows us to give a new proof of the termination of the $2$-polygraph~$\Colo_2(n)$ without considering the combinatorial properties of tableaux.
Indeed, consider the map $\Phi: \Colo_1(n)^{\ast}\to \Colo_1(n)^{\ast}$  sending a  $1$-cell in $\Colo_1(n)^{\ast}$ to its unique corresponding tableau. Then $(\Colo_1(n), \Phi)$ is a quadratic normalisation of the monoid~$\P_n$ in the sense of~\cite{GuiraudDehornoy15}. That is, the map $\Phi$ satisfies
\begin{enumerate}[{\bf i)}]
\item $\len(\Phi(w)) = \len(w)$,
\item $\len(w) = 1$ implies $\Phi(w) = w$,
\item $\Phi(u\Phi(w)v) = \Phi(uwv)$,  for all $1$-cells $u$, $v$ and $w$ in~$\Colo_1(n)^{\ast}$,
\end{enumerate}
and a quadraticity property, see~{\cite[Definition 3.1.2.]{GuiraudDehornoy15}} for details.
Using the fact that the~$2$-polygraph~$\Colo_2(n)$ has the unique normal form property as proved in Proposition~\ref{normalformproperty}, we show by Theorem~\ref{MainTheoremA} that the quadratic normalisation~$(\Colo_1(n), \Phi)$ is of class~$(3,3)$, that is, one obtains the normal form after at most~$3$ steps when starting from the left and~$3$ steps from the right. 
Hence, we obtain by~{\cite[Proposition~5.1.1]{GuiraudDehornoy15}}, that the corresponding $2$-polygraph~$\Colo_2(n)$ is finite and convergent. As a consequence, we obtain a new proof of the termination of the $2$-polygraph~$\Colo_2(n)$ .
\end{remark}

\section{Reduction of the coherent presentation}
\label{Section:ReductionCoherentPresentation}

In this section, using the homotopical reduction procedure, we reduce the coherent presentation~$\Colo_3(n)$ into a smaller coherent presentation of the monoid $\P_n$.
Let us  begin by recalling the homotopical completion-reduction procedure introduced in~{\cite[2.3.3]{GaussentGuiraudMalbos14}.

\subsection{Homotopical completion-reduction procedure}

\subsubsection{Homotopical reduction procedure} 

Let $\Sigma$ be a $(3,1)$-polygraph. A \emph{$3$-sphere} of the \linebreak $(3,1)$-category $\tck{\Sigma}_3$ is a pair $(f,g)$ of $3$-cells of $\tck{\Sigma}_3$ such that $s_2(f)=s_2(g)$ and $t_2(f)=t_2(g)$.
A \emph{collapsible part of $\Sigma$} is a triple $ \Gamma=(\Gamma_2,\Gamma_3,\Gamma_4)$ made of a
family~$\Gamma_2$ of $2$-cells of $\Sigma$, a family $\Gamma_3$ of
$3$-cells of $\Sigma$ and a family~$\Gamma_4$ of $3$-spheres of
$\tck{\Sigma}_3$, such that the following conditions are satisfied:
\begin{enumerate}[{\bf i)}]
\item every $\gamma$ of every $\Gamma_k$ is collapsible,
that is, $t_{k-1}(\gamma)$ is in $\Sigma_{k-1}$ and $s_{k-1}(\gamma)$
does not contain $t_{k-1}(\gamma)$,
\item no cell of $\Gamma_2$ (resp. $\Gamma_3$) is the target of a
collapsible $3$-cell of $\Gamma_3$ (resp. $3$-sphere of $\Gamma_4$),
\item there exists a well-founded order on the cells of $\Sigma$ such
that, for every~$\gamma$ in every $\Gamma_k$, $t_{k-1}(\gamma)$ is
strictly greater than every generating $(k-1)$-cell that occurs in the
source of~$\gamma$.
\end{enumerate}

The \emph{homotopical reduction} of the $(3,1)$-polygraph $\Sigma$ with respect to a collapsible part $\Gamma$ is the Tietze transformation, denoted by $R_{\Gamma}$, from the $(3,1)$-category
$\tck{\Sigma}_3$ to the $(3,1)$-category freely generated by the
$(3,1)$-polygraph obtained from $\Sigma$ by removing the cells of
$\Gamma$ and all the corresponding redundant cells. We refer the reader to {\cite[2.3.1]{GaussentGuiraudMalbos14}} for details on the definition of the Tietze transformation $R_{\Gamma}$ defined by well-founded induction as follows. For any $\gamma$ in $\Gamma$
\[
R_\Gamma(t(\gamma)) = R_\Gamma(s(\gamma))
\qquad\text{and}\qquad
R_\Gamma(\gamma)=1_{R_\Gamma(s(\gamma))}.
\]
In any other cases, the transformation $R_\Gamma$ acts as an identity.

\subsubsection{Generating triple confluences}

A \emph{local triple branching} of a $2$-polygraph $\Sigma$ is a triple $(f,g,h)$ of rewriting steps of $\Sigma$ with a common source.
An \emph{aspherical} triple branchings have two of their $2$-cells equal. 
A \emph{Peiffer} triple branchings have at least one of their $2$-cells that form a Peiffer branching with the other two.
The \emph{overlap} triple branchings are the remaining local triple branchings.
Local triple branchings are ordered by inclusion of their sources and a minimal overlap triple branching is called \emph{critical}. 
If $\Sigma$ is a coherent and convergent $(3,1)$-polygraph, a \emph{triple generating confluence of $\Sigma$} is a $3$-sphere
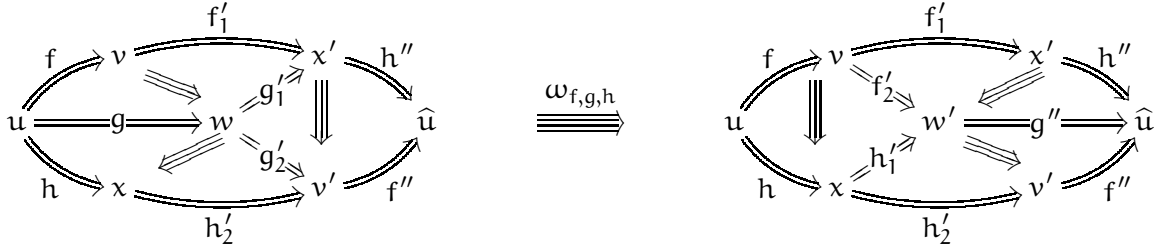
\[
\xymatrix@!C @C=1.8em @R=0.7em{
& v
	\ar@2 @/^/ [rr] ^-{f'_1}
	\ar@3 []!<10pt,-7pt>;[dr]!<-8pt,10pt>
&& x'
	\ar@2 @/^/ [dr] ^-{h''}
		\ar@3 []!<-1pt,-9pt>;[dd]!<-1pt,17pt>
&&&&& v
	\ar@2 @/^/ [rr] ^-{f'_1}
	\ar@2 [dr] |-{f'_2}
	\ar@3 []!<-8pt,-9pt>;[dd]!<-8pt,17pt>
&& x'
	\ar@2 @/^/ [dr] ^-{h''}
	\ar@3 []!<8pt,-2pt>;[dl]!<15pt,8pt>
\\
u
	\ar@2 @/^/ [ur] ^-{f}
	\ar@2 [rr] |-{g}
	\ar@2 @/_/ [dr] _-{h}
&& w
	\ar@2 [ur] |-{g'_1}
	\ar@2 [dr] |-{g'_2}
	\ar@3 []!<8pt,-2pt>;[dl]!<15pt,8pt>
&& {\rep{u}}
& \strut
	\ar@4 [r] ^-*+{\omega_{f,g,h}}
&& u
	\ar@2 @/^/ [ur] ^-{f}
	\ar@2 @/_/ [dr] _-{h}
&& w'
	\ar@2 [rr] |-{g''}
	\ar@3 []!<10pt,-7pt>;[dr]!<-8pt,10pt>
&& {\rep{u}}
\\
& x
	\ar@2 @/_/ [rr] _-{h'_2}
&& v'
	\ar@2 @/_/ [ur] _-{f''}
&&&&& x
	\ar@2 [ur] |-{h'_1}
	\ar@2 @/_/ [rr] _-{h'_2}
&& v'
	\ar@2 @/_/ [ur] _-{f''}
}
\]
where $(f,g,h)$ is a triple critical branching of the $2$-polygraph $\Sigma_2$ and the other cells are obtained by confluence, see~\cite[2.3.2]{GaussentGuiraudMalbos14} for details.

\subsubsection{Homotopical reduction of the polygraph \pdf{\Colo_3(n)}}

In the rest of this section, we apply three steps of homotopical reduction on the $(3,1)$-polygraph $\Colo_3(n)$.
As a first step, we apply in~\ref{Subsection:ReducedColumnPresentation} a homotopical reduction  on the $(3,1)$-polygraph $\Colo_3(n)$ with a collapsible part defined by some of the generating triple confluences of the $2$-polygraph $\Colo_2(n)$. In this way, we reduce the coherent presentation~$\Colo_3(n)$ of the monoid~$\P_n$ into the coherent presentation~$\overline{\Colo}_3(n)$ of~$\P_n$, whose underlying $2$-polygraph is $\Colo_2(n)$ and the $3$-cells $\mathcal{X}_{u,v,t}$ are those of $\Colo_3(n)$, but with $\len(u)=1$.
We reduce in~\ref{Subsection:PreColumnCoherentPresentation} the coherent presentation~$\overline{\Colo}_3(n)$ into a coherent presentation $\PreColo_3(n)$ of $\P_n$, whose underlying $2$-polygraph is $\PreColo_2(n)$. This reduction is given by a collapsible part defined by a set of $3$-cells of $\overline{\Colo}_3(n)$. 
In a final step, we reduce in \ref{Section:KnuthCoherentPresentation} the coherent presentation $\PreColo_3(n)$ into a coherent presentation $\Knuth_3(n)$ of $\P_n$ whose underlying $2$-polygraph is $\Knuth_2(n)$.
By {\cite[Theorem 2.3.4]{GaussentGuiraudMalbos14}}, all these homotopical reductions preserve coherence. That is, the $(3,1)$-polygraph~$\Colo_3(n)$ being a coherent presentation of the monoid~$\P_n$, the $(3,1)$-polygraphs~$\overline{\Colo}_3(n)$ and~$\Knuth_3(n)$ are coherent presentations of $\P_n$.

\subsection{A reduced column presentation}
\label{Subsection:ReducedColumnPresentation}

We apply the homotopical reduction procedure in order to reduce the $(3,1)$-polygraph~$\Colo_3(n)$ using the generating triple confluences.

\subsubsection{Generating triple confluences of $\Colo_2(n)$}
\label{Generating triple confluences}

Consider the homotopical reduction procedure on the $(3,1)$-polygraph $\Colo_3(n)$ defined using the collapsible part made of generating triple confluences.
By Theorem~\ref{MainTheoremA}, the family of $3$-cells $\mathcal{X}_{u,v,t}$ given in~(\ref{celluleX}) and indexed by columns $u$, $v$ and $t$ in~$\col(n)$ such that  \typetroisbasep{$u$}{$v$}{$t$}{00} forms a homotopy basis of the $(2,1)$-category $\tck{\Colo_2(n)}$.
Let us consider such a triple $(u,v,t)$ with $\len(u)\geq 2$. Let $x_p$  be in~$[n]$ such that $u=x_pu_1$ with~$u_1$ in $\col(n)$. There is a critical triple branching with source $c_{x_p}c_{u_1}c_vc_t$. Let us show that the confluence diagram induced by this triple branching is represented by the $3$-sphere $\Omega_{x_p,u_1,v,t}$ whose source is the following $3$-cell
\[
\scalebox{0.9}{
\xymatrix @C=0.65em@R=1.5em{
&
&
c_uc_{v}c_t
\ar@2[rr] ^-{\alpha_{u,v}} _-{}
\ar@3 []!<-8pt,-12pt>;[dd]!<-8pt,17pt> ^{\mathcal{X}_{x_p,u_1,v}c_{t}} 
&
&
c_ec_{e'}c_{t}
\ar@2[drrr] ^-{\alpha_{e',t}} _-{}
 \ar@3 []!<-9pt,-12pt>;[dd]!<-9pt,20pt> ^{c_e\mathcal{X}_{y,s',t}} 
&
&
\\
&
&
&
c_{e}c_{y}c_{s'}c_{t}
\ar@2[ur] ^-{\alpha_{y,s'}} _-{}
\ar@2[dr] ^-{\alpha_{s',t}} _-{}
&&&&
c_{e}c_{b}c_{b'}
\ar@2[dr] ^-{\alpha_{e,b}} _-{}
\\
c_{x_p}c_{u_1}c_vc_t
\ar@2[rr] ^-{\alpha_{u_1,v}} _-{}
\ar@2[uurr] ^-{\alpha_{x_p,u_1}} _-{}
\ar@2[dr] ^-{} _-{\alpha_{v,t}}
&
&
c_{x_{p}}c_{s}c_{s'}c_{t}
\ar@2[ur] ^-{\alpha_{x_{p},s}} _-{}
\ar@2[dr] ^-{\alpha_{s',t}} _-{}
\ar@3 []!<-5pt,-12pt>;[dd]!<-5pt,15pt> ^{c_{x_p}\mathcal{X}_{u_1,v,t}} 
&
\equiv
&
c_{e}c_{y}c_{d_{1}}c_{d_{1}'}
\ar@2[rr] ^-{\alpha_{y,d_{1}}} _-{}
\ar@3 []!<15pt,-15pt>;[dd]!<15pt,15pt> ^{\mathcal{X}_{x_p,s,d_1}c_{d'_1}} 
&
&
c_{e}c_{b}c_{s_{2}}c_{d_{1}'}
\ar@2[ur] ^-{\alpha_{s_{2},d_{1}'}} _-{}
\ar@2[dr] ^-{\alpha_{e,b}} _-{}
&
\equiv
&
c_{a}c_{d}c_{b'}
\\
&
c_{x_{p}}c_{u_{1}}c_{w}c_{w'}
\ar@2[dr] _-{\alpha_{u_{1},w}} _-{}
&
&
c_{x_{p}}c_{s}c_{d_{1}}c_{d_{1}'}
\ar@2[dr] ^-{\alpha_{s,d_{1}}} _-{}
\ar@2[ur] ^-{\alpha_{x_{p},s}} _-{}
&
&
&
&
c_{a}c_{d}c_{s_{2}}c_{d_{1}'}
\ar@2[ur] _-{\alpha_{\alpha_{s_{2},d_{1}'}}} _-{}
\\
&
&
c_{x_{p}}c_{a_{1}}c_{a_{1}'}c_{w'}
\ar@2[rr] ^-{} ^-{\alpha_{a_{1}',w'}}
\ar@2[ddrr] ^-{} _-{\alpha_{x_{p},a_{1}}}
&
&
c_{x_{p}}c_{a_{1}}c_{s_{3}}c_{d_{1}'}
\ar@2[rr] ^-{\alpha_{x_{p},a_{1}}} _-{}
&
&
c_{a}c_{z}c_{s_{3}}c_{d_{1}'}
\ar@2[ur] _-{\alpha_{z,s_{3}}} _-{}
\\
&&&& \equiv
\\
&&&&
c_{a}c_{z}c_{a_{1}'}c_{w'}
\ar@2[uurr] _-{\alpha_{a_{1}',w'}} _-{}
}
}
\]
and whose target is the following $3$-cell
\[
\scalebox{0.9}{
\xymatrix @C=0.65em@R=1.5em{
&
&
c_ec_{e'}c_t
\ar@2[rr] ^-{\alpha_{e',t}} _-{}
 \ar@3 []!<20pt,-15pt>;[dd]!<20pt,15pt> ^{\mathcal{X}_{u,v,t}} 
&
&
c_ec_bc_{b'}
\ar@2[dr] ^-{\alpha_{e,b}} _-{}
&
&
\\
&
c_uc_vc_t
\ar@2[ur] ^-{\alpha_{u,v}} _-{}
\ar@2[dr] ^-{} _-{\alpha_{v,t}}
&
&
&
&
c_ac_dc_{b'}
&
&
\\
c_{x_p}c_{u_1}c_vc_t
\ar@2[ur] ^-{\alpha_{x_p,u_1}} _-{}
\ar@2[dr] ^-{} _-{\alpha_{v,t}}
&
\equiv
&
c_uc_wc_{w'}
\ar@2[rr] ^-{\alpha_{u,w}} _-{}
\ar@3 []!<10pt,-15pt>;[dd]!<10pt,15pt> ^{\mathcal{X}_{x_p,u_1,w}c_{w'}} 
&
&
c_ac_{a'}c_{w'}
\ar@2[ur] ^-{\alpha_{a',w'}} _-{}
\ar@3 []!<15pt,-15pt>;[dr]!<15pt,15pt> ^{c_{a}\mathcal{X}_{z,a'_1,w'}}
&
&
c_{a}c_{d}c_{s_{2}}c_{d_{1}'}
\ar@2[ul] _-{\alpha_{s_{2},d_{1}'}} _-{}
&
&
\\
&
c_{x_{p}}c_{u_{1}}c_{w}c_{w'}
\ar@2[ur] ^-{\alpha_{x_{p},u_{1}}} _-{}
\ar@2[dr] ^-{} _-{\alpha_{u_{1},w}}
&
&
&
&
c_{a}c_{z}c_{s_{3}}c_{d_{1}'}
\ar@2[ur] _-{\alpha_{z,s_{3}}} _-{}
&
&
\\
&
&
c_{x_{p}}c_{a_{1}}c_{a_{1}'}c_{w'}
\ar@2[rr] ^-{\alpha_{x_{p},a_{1}}} _-{}
&
&
c_{a}c_{z}c_{a_{1}'}c_{w'}
\ar@2[ur] _-{\alpha_{a_{1}',w'}} _-{}
\ar@2[uu] _-{\alpha_{z,a_{1}'}} _-{}
&
&
\\
}
}
\]
In the generating triple confluence, some columns may be empty and thus the indicated $2$-cells $\alpha$ may be identities. To facilitate the reading of the diagram, we have omitted the context of the $2$-cells $\alpha$.

The $3$-sphere $\Omega_{x_p,u_1,v,t}$ is constructed as follows.
We have \typedeuxbase{$x_p$}{$u_1$}{01} and \typedeuxbase{$u_1$}{$w$}{0}, thus $\mathcal{X}_{x_p,u_1,w}$ is either of the form $A_{x_p,u_1,w}$ or $C_{x_p,u_1,w}$. Let us denote by $a_{1}$ and $a_{1}'$ the two columns of the tableau $P(u_{1}w)$.  The $3$-cell $\mathcal{X}_{x_{p},u_1,w}$ being confluent, we have~$C(P(x_{p}a_{1})) = az$ with $z$ in $[n]$ and $C(P(za_{1}^{'})) = a'$. 
In addition, from \typedeuxbase{$z$}{$a'_1$}{01} and \typedeuxbase{$a'_1$}{$w'$}{0}, we deduce that $\mathcal{X}_{z,a'_{1},w'}$ is either of the form $A_{z,a'_{1},w'}$ or $C_{z,a'_{1},w'}$. 
From \typedeuxbase{$x_p$}{$u_1$}{01} and \typedeuxbase{$u_1$}{$v$}{0}, we deduce that
$\mathcal{X}_{x_p,u_1,v}$ is either of the form~$A_{x_p,u_{1},v}$ or~$C_{x_p,u_{1},v}$. Let us denote by $s$ and $s'$ the two columns of the tableau $P(u_1v)$. The $3$-cell $\mathcal{X}_{x_p,u_1,v}$ being confluent, we obtain that~$C(P(x_{p}s)) = ey$ with $y$ in $[n]$ and $C(P(ys')) = e'$. From \typedeuxbase{$y$}{$s'$}{01} and \typedeuxbase{$s'$}{$t$}{0}, we deduce that $\mathcal{X}_{y,s',t}$ is either of the form~$A_{y,s',t}$ or~$C_{y,s',t}$. Denote by $d_1$ and $d_1'$ the two columns of the tableau $P(s't)$. The $3$-cell~$\mathcal{X}_{y,s',t}$ being confluent and  $C(P(e't))= bb'$,  we have  $C(P(yd_{1})) = bs_{2}$ and $C(P(s_{2}d_{1}')) = b'$.
On the other hand, the $3$-cell $\mathcal{X}_{u_{1},v,t}$ is confluent, then we have  $C(P(sd_{1})) = a_{1}s_{3}$ and $C(P(a_{1}'w')) = s_{3}d_{1}'$. 
Finally, since the $3$-cell $\mathcal{X}_{x_{p},s,d_{1}}$ is confluent, we obtain $C(P(zs_{3})) = ds_{2}$.

\subsubsection{Reduced coherent column presentation}

Let us define by $\overline{\Colo}_3(n)$ the extended presentation of the monoid $\P_n$ obtained from $\Colo_2(n)$ by adjunction of one  family of $3$-cells~$\mathcal{X}_{x,v,t}$ of the form~(\ref{celluleX}), for every $1$-cell $x$ in $[n]$ and columns $v$ and $t$ in $\col(n)$ such that \typetroisbasep{$x$}{$v$}{$t$}{00}. The following result shows that this reduced presentation is also coherent.

\begin{proposition}
\label{Proposition:Colo3Reduced}
For $n > 0$, the $(3,1)$-polygraph $\overline{\Colo}_3(n)$ is a coherent presentation of the monoid~$\P_n$. 
\end{proposition}

\begin{proof}
Let $\Gamma_4$ be the collapsible part made of the family of $3$-sphere $\Omega_{x_p,u_1,v,t}$, indexed by $x_p$ in $[n]$ and~$u_1,v,t$ in $\col(n)$ such that \typetroisbasep{$u$}{$v$}{$t$}{00} and $u=x_pu_1$.
On the $3$-cells of $\Colo_3(n)$, we define a well-founded order~$\vartriangleleft$ by
\begin{enumerate}[{\bf i)}]
\item $A_{u,v,t} \vartriangleleft C_{u,v,t} \vartriangleleft B_{u,v,t} \vartriangleleft D_{u,v,t}$, 
\item if $\mathcal{X}_{u,v,t}\in\{A_{u,v,t}, B_{u,v,t}, C_{u,v,t}, D_{u,v,t}\}$ and $u'\ordrecoldeglex u$, then $\mathcal{X}_{u',v',t'} \vartriangleleft \mathcal{X}_{u,v,t}$,
\end{enumerate}
for any $u,v,t$ in $\col(n)$ such that  \typetroisbasep{$u$}{$v$}{$t$}{00}.
By construction of the $3$-sphere $\Omega_{x_p,u_1,v,t}$, its source contains the $3$-cell $\mathcal{X}_{u_{1},v,t}$ and its target contains the $3$-cell $\mathcal{X}_{u,v,t}$ with $\len(u_1)<\len(u)$. Up to a Nielsen transformation, the homotopical reduction  $R_{\Gamma_4}$ applied on the $(3,1)$-polygraph $\Colo_{3}(n)$ with respect to $\Gamma_{4}$ and the order~$\vartriangleleft$ give us the $(3,1)$-polygraph~$\overline{\Colo}_3(n)$. In this way, the presentation~$\overline{\Colo}_3(n)$ is a coherent presentation of the monoid~$\P_n$.
\end{proof} 
 
\subsection{Pre-column coherent presentation}
\label{Subsection:PreColumnCoherentPresentation}

We reduce the coherent presentation $\overline{\Colo}_3(n)$ into a coherent presentation whose underlying $2$-polygraph is~$\PreColo_2(n)$.
This reduction is obtained using the homotopical reduction~$R_{\Gamma_3}$ on the $(3,1)$-polygraph~$\overline{\Colo}_3(n)$ whose collapsible part $\Gamma_3$ is defined by
\begin{align*}
\Gamma_3 = 
\{\;
&A_{x,v,t}\;|\;x\in [n],\;v,t\in \col(n)\;\text{such that}\; \typetroisbasep{$x$}{$v$}{$t$}{11}
\;\}
\\
&\qquad\cup
\{\;
B_{x,v,t}\;|\;x\in [n],\;v,t\in \col(n)\;\text{such that}\; \typetroisbasep{$x$}{$v$}{$t$}{21}
\;\}
\\
&\qquad\qquad\cup
\{\;
C_{x,v,t}\;|\;x\in [n],\;v,t\in \col(n)\;\text{such that}\; \typetroisbasep{$x$}{$v$}{$t$}{12}
\;\},
\end{align*}
and the well-founded order defined as follows.

\subsubsection{A well-founded order on $2$-cells}
\label{WellFoundedOrderOn2Cells}

Consider two columns $u$ and $v$ in $\col(n)$ such that \typedeuxbase{$u$}{$v$}{0}. Let denote by $C_{r}(P(uv))$ the reading of the right column of the tableau $P(uv)$. We define a well-founded order~$\vartriangleleft$ on the $2$-cells of~$\Colo_2(n)$ as follows
\[
\alpha_{u',v'}\vartriangleleft \alpha_{u,v} 
\quad\text{if}\quad
\begin{cases}
\len(uv) >  \len(u'v') & \text{or} \\
\len(uv)= \len(u'v')  & \text{and \quad} 
\begin{cases}
\len(u)>\len(C_{r}(P(u'v')))\;\;\text{or} \\
\len(u)\leq \len(C_{r}(P(u'v'))) \text{ and } u'\preccurlyeq _{\text{rev}} u
\end{cases}
\end{cases}
\]
for any columns $u$, $v$, $u'$ and $v'$ in $\col(n)$ such that \typedeuxbase{$u$}{$v$}{0} and \typedeuxbase{$u'$}{$v'$}{0}.

\subsubsection{The homotopical reduction $R_{\Gamma_3}$}
\label{SubSubSectionHomotopicalReduction}

Consider the well-founded order $\vartriangleleft$ on the $2$-cells of $\Colo_2(n)$ defined in \ref{WellFoundedOrderOn2Cells} and the well-founded order $\vartriangleleft$ on $3$-cells defined in the proof of Proposition~\ref{Proposition:Colo3Reduced}.
The reduction $R_{\Gamma_3}$ induced by these orders can be decomposed as follows. 
For any $x$ in $[n]$ and columns $v$, $t$ such that \typetroisbasep{$x$}{$v$}{$t$}{11}, 
we have $\alpha_{x,v}\vartriangleleft \alpha_{xv,t}$, $\alpha_{v,t} \vartriangleleft \alpha_{xv,t}$ and~\mbox{$\alpha_{x,vt} \vartriangleleft \alpha_{xv,t}$}. The reduction $R_{\Gamma_3}$ removes the $2$-cell $\alpha_{xv,t}$ together with the following $3$-cell:
\[
\xymatrix @!C @C=2.3em @R=0.6em {
&
c_{xv}c_t
   \ar@2[dr] ^{\alpha_{xv,t}}
  \ar@3 []!<0pt,-10pt>;[dd]!<0pt,+10pt> ^{A_{x,v,t}} 
&
\\
c_xc_vc_t
  \ar@2[ur] ^{\alpha_{x,v}c_t}
  \ar@2[dr] _{c_x\alpha_{v,t}}
&&
c_{xvt}
\\
&
c_xc_{vt}
  \ar@2[ur] _{\alpha_{x,vt}}
&
}	
\] 
By iterating this reduction on the length of the column~$v$, we reduce all the $2$-cells $\alpha_{u,v}$ of $\Colo_2(n)$ to the following set of $2$-cells 
\begin{equation}
\label{Equations:twoCellsA}
\{
\;\alpha_{u,v} 
\;|\;  
\len(u)\geq 1,\; \len(v)\geq 2 \;\text{and}\; \typedeux{02} 
\; 
\}
\,\cup\,
\{
\;\alpha_{u,v}
\;|\;  
\len(u)= 1,\; \len(v)\geq 1 \;\text{and}\; \typedeux{01}
\;\}.
\end{equation}

For any $x$ in $[n]$ and columns $v$, $t$ such that \typetroisbasep{$x$}{$v$}{$t$}{12}, consider the following $3$-cell:
\[
\xymatrix @!C @C=1.4em @R=0.6em {
&&
c_{xv}c_t
     \ar@3 []!<0pt,-10pt>;[dd]!<0pt,10pt> ^{C_{x,v,t}} 
  \ar@2[drr] ^{\alpha_{xv,t}}
&&
\\
c_xc_vc_t
  \ar@2[urr] ^{\alpha_{x,v}c_t}
  \ar@2[dr] _{c_x\alpha_{v,t}}
&&&&
c_{a}c_{a'w'}
\\
&
c_xc_{w}c_{w'}
  \ar@2[rr] _{\alpha_{x,w}c_{w'}}
&&
c_{a}c_{a'}c_{w'}
 \ar@2[ur] _{c_{a}\alpha_{a',w'}}
&
}	
\]
where $w$, $w'$, $a$ and $a'$ are defined in Lemma~\ref{Lemma2}.
The $2$-cells $\alpha_{x,v}$, $\alpha_{v,t}$, $\alpha_{x,w}$ and $\alpha_{a',w'}$ are smaller than $\alpha_{xv,t}$ for the order $\vartriangleleft$.
The reduction $R_{\Gamma_3}$ removes the $2$-cell $\alpha_{xv,t}$ together with the $3$-cell $C_{x,v,t}$.
By iterating this reduction on the length of $v$, we reduce the set of $2$-cells given in (\ref{Equations:twoCellsA}) to the following set: 
\begin{equation}
\label{Equations:twoCellsB}
\{
\;\alpha_{u,v} 
\;|\;  
\len(u)= 1,\; \len(v)\geq 2 \;\text{and}\; \typedeux{02} 
\; 
\}
\,\cup\,
\{
\;\alpha_{u,v}
\;|\;  
\len(u)= 1,\; \len(v)\geq 1 \;\text{and}\; \typedeux{01}
\;\}.
\end{equation}

For any $x$ in $[n]$ and columns $v$, $t$ such that \typetroisbasep{$x$}{$v$}{$t$}{21}, consider the following $3$-cell:
\[
\xymatrix @!C @C=1.4em @R=0.6em {
&
c_{e} c_{e'}c_t
  \ar@2[rr] ^{c_{e}\alpha_{e',t}}
&
\ar@3 []!<0pt,-10pt>;[dd]!<0pt,10pt> ^{B_{x,v,t}} 
&
c_{e} c_{e't}
  \ar@2[dr] ^{\widetilde{\alpha}_{e,e't}}
&
\\
c_xc_vc_t
  \ar@2[ur] ^{\alpha_{x,v}c_t}
  \ar@2[drr] _{c_x\alpha_{v,t}}
&&&&
c_{s}c_{s'}
\\
&&
c_xc_{vt}
  \ar@2[urr] _{\alpha_{x,vt}}
&&
}	
\]
where~$e$,~$e'$,~$s$ and~$s'$ are defined in Lemma~\ref{Lemma3}. Note that~$\widetilde{\alpha}_{e,e't}$ is the $2$-cell in (\ref{Equations:twoCellsB}) obtained from the $2$-cell $\alpha_{e,e't}$ by the previous step of the homotopical reduction by the $3$-cell $C_{x,v,t}$.
Having $x$ in $[n]$, by definition of $\alpha$ we have $e'$ in $[n]$.
The $2$-cells $\alpha_{x,v}$, $\alpha_{e',t}$, $\alpha_{v,t}$ and~$\widetilde{\alpha}_{e,e't}$ being smaller than $\alpha_{x,vt}$ for the order $\vartriangleleft$, we can remove the $2$-cells $\alpha_{x,vt}$ together with the $3$-cell $B_{x,v,t}$. By iterating this reduction on the length of the column~$t$, we reduce the set (\ref{Equations:twoCellsB}) to the following set
\begin{equation}
\label{Equations:twoCellsC}
\{
\;\alpha_{u,v} 
\;|\;  
\len(u)= 1,\; \len(v)= 2 \;\text{and}\; \typedeux{02} 
\; 
\}
\,\cup\,
\{
\;\alpha_{u,v}
\;|\;  
\len(u)= 1,\; \len(v)\geq 1 \;\text{and}\; \typedeux{01}
\;\}.
\end{equation}

Let us recall from Section~\ref{Section:Pre-columnPresentation} that~$\PC_2(n)$ is the cellular extension of $\Colo_1^\ast(n)$ whose set of $2$-cells is 
\[
\big\{\;
c_{x}c_{zy}\odfll{\alpha'_{x,zy}}c_{zx}c_{y}
\;|\; 1 \leq x\leq y<z \leq n
\;\big\}\,\cup\,\big\{\; c_{y}c_{zx}\odfll{\alpha'_{y,zx}}c_{yx}c_{z}
\;|\; 1 \leq x< y\leq z \leq n
\;\big\}.
\]

\begin{lemma}
\label{Lemma:Precol2}
We have
\[
\PC_2(n) = 
\{\;\alpha_{u,v} : c_uc_v \dfl c_wc_{w'} \;|\;  \len(u)=1,\; \len(v)= 2 \;\text{and}\; \typedeux{02} \;\}.
\]
\end{lemma}

\begin{proof}
Consider the $2$-cells  $\alpha_{u,v}$ in $\Colo_2(n)$ such that $\len(u)=1$, $\len(v)= 2$ and \typedeux{02}. Suppose that~$v=xx'$ with $x>x'$ in $[n]$. Since \typedeux{02}, we obtain  that $u\leq x$. Hence,  we have two cases to consider.
If $u\leq x'$, then $C(P(uv)) = (xu)x'$. Hence, the $2$-cell $\alpha_{u,v}$ is equal to the $2$-cell~$\alpha'_{u,xx'} : c_{u}c_{xx'}\dfl c_{xu}c_{x'}$.
In the other case, if $x'<u$, then $C(P(uv)) = (ux')x$. Hence the $2$-cell~$\alpha_{u,v}$ is equal to~\mbox{$\alpha'_{u,xx'}:c_{u}c_{xx'}\dfl~c_{ux'}c_{x}$.}
\end{proof}

Recall from~\ref{Subsubsection:Pre-columnPresentation} that the set of $2$-cells $\PreColo_2(n)$ is given by 
\[
\PreColo_2(n) = \PC_2(n)
\,\cup\,
\big\{
c_{x}c_{u} \odfl{\alpha'_{x,u}} c_{xu}
\;|\; xu\in \col(n) \;\;\text{and}\;\; 1\leq x \leq n
\big\}.
\]
Thus, by Lemma~\ref{Lemma:Precol2}, the set of $2$-cells defined in (\ref{Equations:twoCellsC}) is equal to $\PreColo_2(n)$.

\subsubsection{Pre-column coherent presentation}

The homotopical reduction $R_{\Gamma_3}$, defined in \ref{SubSubSectionHomotopicalReduction},
 reduces the coherent presentation $\overline{\Colo}_3(n)$ into a coherent presentation of the monoid $\P_n$.
The set of $2$-cells of this coherent presentation is given by~(\ref{Equations:twoCellsC}), which is~$\PreColo_2(n)$ by Lemma~\ref{Lemma:Precol2}.
Let us denote by~$\PreColo_3(n)$ the extended presentation of the monoid $\P_n$ obtained from $\PreColo_2(n)$ by adjunction of  the $3$-cell $R_{\Gamma_3}(C'_{x,v,t})$ where
\[
\xymatrix @C=3em @R=0.6em {
& {c_{xv}c_{t}}
\ar@3 []!<0pt,-10pt>;[dd]!<0pt,+10pt> ^{C'_{x,v,t}}
\\
{c_{x}c_{v}c_{t}}
	\ar@2@/^/ [ur] ^-{\alpha_{x,v}c_{t}}
	\ar@2@/_/ [dr] _-{c_{x}\alpha_{v,t}}
\\
& {c_{x}c_{w}c_{w'}}
	\ar@2 [r] _-{\alpha_{x,w}c_{w'}}
& {c_{xv}c_{z_{l}\ldots z_{q+1}}c_{w'}}
	\ar@2 [luu] _-{c_{xv}\alpha_{z_{l}\ldots z_{q+1}, w'}}
}
\]
with \typetroisbasep{$x$}{$v$}{$t$}{12}, and the $3$-cell $R_{\Gamma_3}(D_{x,v,t})$ where
\[
\xymatrix @!C @C=2.3em @R=0.6em {
&
c_{e} c_{e'}c_t
  \ar@2[r] ^{c_{e}\alpha_{e',t}}
     \ar@3 []!<40pt,-10pt>;[dd]!<40pt,10pt> ^{D_{x,v,t}} 
&
c_{e} c_{b}c_{b'}
  \ar@2[dr] ^{\alpha_{e,b}c_{b'}}
\\
c_xc_vc_t
  \ar@2[ur] ^{\alpha_{x,v}c_t}
  \ar@2[dr] _{c_x\alpha_{v,t}}
&&&
c_{a}c_{d} c_{b'}
\\
&
c_xc_{w}c_{w'}
  \ar@2[r] _{\alpha_{x,w}c_{w'}}
&
c_{a}c_{a'}c_{w'}
 \ar@2[ur] _{c_{a}\alpha_{a',w'}}
}
\]
with \typetroisbasep{$x$}{$v$}{$t$}{22}.
The homotopical reduction $R_{\Gamma_3}$ eliminates the $3$-cells of~$\overline{\Colo}_3(n)$ of the form $A_{x,v,t}$, $B_{x,v,t}$ and $C_{x,v,t}$, which are not of the form $C'_{x,v,t}$. We have then proved the following result.

\begin{theorem}
\label{Theorem:PreColo3Coherent}
For $n > 0$, the $(3,1)$-polygraph $\PreColo_3(n)$ is a coherent presentation of the monoid~$\P_n$.
\end{theorem}

\subsubsection{Example: coherent presentation of monoid \pdf{\P_2}} 
\label{P2Example}

The Knuth presentation $\Knuth_2(2)$ has generators $1$ and $2$ subject to the Knuth relations $\eta_{1,1,2} : 211 \dfl 121$ and $\epsilon_{1,2,2} : 221 \dfl 212$. This presentation is convergent with only one critical branching with source the $1$-cell $2211$. This critical branching is confluent:
\[
\xymatrix@!C@C=3em{
2211
  \ar@2@/^3ex/ [r] ^{2\eta_{1,1,2}}="src"
  \ar@2@/_3ex/ [r] _{\epsilon_{1,2,2}1}="tgt"
& 2121
  \ar@3 "src"!<-5pt,-15pt>;"tgt"!<-5pt,15pt>  ^-{\, C''}
}
\]
Following the homotopical completion procedure given in \ref{Subsubsection:HomotopicalCompletion}, the $2$-polygraph extended by the previous $3$-cell is a coherent presentation of the monoid $\P_2$.
Consider the column presentation~$\Colo_2(2)$ of the monoid $\P_2$   with $1$-cells $c_1$, $c_2$ and $c_{21}$ and $2$-cells~$\alpha_{2,1}$,~$\alpha_{1,21}$ and~$\alpha_{2,21}$. The coherent presentation $\Colo_3(2)$ has only one $3$-cell
\[
\xymatrix @C=3em @R=0.6em {
& {c_{21}c_{21}}
\ar@3 []!<-10pt,-10pt>;[dd]!<-10pt,+10pt> ^{C'_{2,1,21}}
\\
{c_2c_1c_{21}}
	\ar@2@/^/ [ur] ^-{\alpha_{2,1}c_{21}}
	\ar@2@/_/ [dr] _-{c_{2}\alpha_{1,21}}
\\
& {c_2c_{21}c_1}
	\ar@2 [r] _-{\alpha_{2,21}c_1}
& {c_{21}c_2c_1}
	\ar@2 [luu] _-{c_{21}\alpha_{2,1}}
}
\] 
It follows that the $(3,1)$-polygraphs $\overline{\Colo}_3(2)$ and $\Colo_3(2)$ coincide.
Moreover, in this case the set $\Gamma_3$ is  empty and the homotopical reduction $R_{\Gamma_3}$ is the identity and thus $\PreColo_3(2)$ is also equal to $\Colo_3(2)$.

In next section, we will show how to relate the coherent presentations $\Colo_3(2)$ and $\langle\, \Knuth_2(2)\;|\; C''\,\rangle$. 

\subsubsection{Example: coherent presentation of monoid \pdf{\P_3}} 

For the monoid $\P_3$, the Knuth presentation has~$3$ generators and $8$ relations. It is not convergent, but it can be completed by adding $3$ relations. The obtained presentation has $27$ $3$-cells corresponding to the $27$ critical branchings.
The column coherent presentation~$\Colo_3(3)$ of~$\P_3$ has $7$ generators, $22$ relations and $42$ $3$-cells. The coherent presentation~$\overline{\Colo}_3(3)$ has $7$ generators, $22$ relations and $34$ $3$-cells. After applying the homotopical reduction $R_{\Gamma_3}$, the coherent presentation~$\PreColo_3(3)$ admits~$7$~generators,~$22$~relations and $24$ $3$-cells.
We give in~\ref{SubSection:Computations} the values of number of cells of the $(3,1)$-polygraphs~$\overline{\Colo}_3(n)$ and~$\PreColo_3(n)$  for plactic monoids of rank~$n\leq 10$. 

\subsection{Knuth's coherent presentation}
\label{Section:KnuthCoherentPresentation}

We reduce the coherent presentation $\PreColo_3(n)$ into a coherent presentation of the monoid $\P_n$ whose underlying $2$-polygraph is $\Knuth_2(n)$. We proceed in three steps developed in the next sections.
\begin{enumerate}[{\bf Step 1.}]
\item We apply the inverse of the Tietze transformation~$T_{\gamma \leftarrow \alpha'}$, that coherently replaces the $2$-cells~$\gamma_{x_{p}\ldots x_{1}}$ by the $2$-cells~$\alpha'_{x_p,x_{p-1}\ldots x_{1}}$, for each column~$x_p\ldots x_1$  such that~ $\len(x_{p}\ldots x_{1})>2$.
\item We apply the inverse of the Tietze transformation~$T_{\eta,\epsilon\leftarrow \alpha'}$, that coherently replaces the $2$-cells~$\alpha'_{x,zy}$ by~$\eta_{x,y,z}^{c}$ for~$1\leq x\leq y < z\leq n$ and the $2$-cells~$\alpha'_{y,zx}$ by~$\epsilon_{x,y,z}^{c}$ for~\mbox{$1\leq x<y \leq z \leq n$}.
\item Finally for each column~$x_p\ldots x_1$, we coherently eliminate the generator~$c_{x_p\ldots x_1}$ together with  the $2$-cell~$\gamma_{x_p\ldots x_1}$ with respect to the order~$\ordrecoldeglex$.  
\end{enumerate}

\subsubsection{Step 1}

The Tietze transformation
$T_{\gamma \leftarrow \alpha'}: \tck{\CPC_2(n)}\to \tck{\PreColo_2(n)}$ 
defined in~Lemma~\ref{Precolo2nLemma} substitutes a $2$-cell 
$\alpha'_{x_p,x_{p-1}\ldots x_{1}}: c_{x_{p}}c_{x_{p-1}\ldots x_{1}}\odfl{} c_{x_p\ldots x_1}$ to the $2$-cell \linebreak $\gamma_{x_{p}\ldots x_{1}}:c_{x_p}\ldots c_{x_{1}} \odfl{} c_{x_p\ldots x_1}$, for each column $x_{p}\ldots x_{1}$ such that $\len(x_{p}\ldots x_{1})>2$, from the bigger column to the smaller one with respect to the total order~$\ordrecoldeglex$.

We consider the inverse of this Tietze transformation~$T_{\gamma \leftarrow \alpha'}^{-1}: \tck{\PreColo_2(n)}\to\tck{\CPC_2(n)}$
that substitutes the~$2$-cell~$\gamma_{x_{p}\ldots x_{1}}:c_{x_p}\ldots c_{x_{1}} \odfl{} c_{x_p\ldots x_1}$ to the $2$-cell~$\alpha'_{x_p,x_{p-1}\ldots x_{1}}: c_{x_{p}}c_{x_{p-1}\ldots x_{1}}\odfl{} c_{x_p\ldots x_1}$ 
\[
\xymatrix @C=1em @R=0,6em {
{c_{x_{p}} c_{x_{p-1}\ldots x_{1}}}
  \ar@2 [rr] ^-{\alpha'_{x_{p},x_{p-1}\ldots x_{1}}}
   &&{c_{x_{p}\ldots x_{1}}}
\\
   \\
&{c_{x_{p}}\ldots c_{x_{1}}}
\ar@2@/_/ [uur] _-{\gamma_{x_{p}\ldots x_{1}}}
\ar@2@/^/ [uul] ^-{c_{x_{p}}\gamma_{x_{p-1}\ldots x_{1}}}
}
\]
for each column~$x_p\ldots x_1$ such that~$\len(x_{p}\ldots x_{1})>2$ with respect to the order~$\ordrecoldeglex$.

Let us denote by~$\CPC_3(n)$ the $(3,1)$-polygraph whose underlying $2$-polygraph is~$\CPC_2(n)$, and the set of $3$-cells is defined by 
\[
\{
\;
T_{\gamma \leftarrow \alpha'}^{-1}(R_{\Gamma_3}(C'_{x,v,t}))\quad \text{for \typetroisbasep{$x$}{$v$}{$t$}{12}}
\;
\}
\;\cup\;
\{
\;
T_{\gamma \leftarrow \alpha'}^{-1}(R_{\Gamma_3}(D_{x,v,t}))\quad \text{for \typetroisbasep{$x$}{$v$}{$t$}{22}}
\;
\}
.
\]
In this way, we extend the Tietze transformation~$T_{\gamma \leftarrow \alpha'}^{-1}$  into a Tietze transformation between the $(3,1)$-polygraphs~$\PreColo_3(n)$ and~$\CPC_3(n)$.
The $(3,1)$-polygraph $\PreColo_3(n)$ being a coherent presentation of the monoid $\P_n$ and the Tietze transformation $T_{\gamma \leftarrow \alpha'}^{-1}$ preserves the coherence property, hence we have the following result.

\begin{lemma}
For $n > 0$, the monoid $\P_{n}$ admits $\CPC_3(n)$ as a coherent presentation.
\end{lemma}

\subsubsection{Step 2}

The Tietze transformation~$T_{\eta,\epsilon\leftarrow \alpha'}$ from~$\tck{\Knuthcc_2(n)}$ into~$\tck{\CPC_2(n)}$ defined in the proof of Lemma~\ref{Gamma2(n)Lemma} replaces the $2$-cells $\eta_{x,y,z}^{c}$ and $\epsilon_{x,y,z}^{c}$ in  $\Knuthcc_2(n)$ by composite of $2$-cells in $\CPC_2(n)$.

Let us consider the inverse of this Tietze transformation 
$T_{\eta,\epsilon\leftarrow \alpha'}^{-1} :
\tck{\CPC_2(n)} \fll \tck{\Knuthcc_2(n)}$.
making the following transformations. For every $1\leq x\leq y<z \leq n$, $T_{\eta,\epsilon\leftarrow \alpha'}^{-1}$ substitutes the $2$-cell~$\eta_{x,y,z}^{c} : c_zc_xc_y \dfl c_xc_zc_y$ to the $2$-cell~$\alpha'_{x,zy}$: 

\[
\xymatrix @C=1.5em @R=0.6em {
& {c_{x}c_{z}c_{y}}
  \ar@2[r] ^-{c_{x}\gamma_{zy}}
& {c_{x}c_{zy}}
  \ar@2@/^/ [ddl] ^-{\alpha'_{x,zy}}
\\
{c_{z}c_{x}c_{y}}
	\ar@2@/_/ [dr] _-{\gamma_{zx}c_{y}}
\\
& {c_{zx}c_{y}}
}
\]
For every $1\leq x<y\leq z \leq n$,  $T_{\eta,\epsilon\leftarrow \alpha'}^{-1}$ substitutes the $2$-cell $\epsilon_{x,y,z}^{c} : c_yc_zc_x \dfl c_yc_xc_z$ to the $2$-cell~$\alpha'_{y,zx}$: 
\[
\xymatrix @C=1.5em @R=0,6em {
& {c_{y}c_{x}c_{z}}
  \ar@2[r] ^-{\gamma_{yx}c_{z}}
& {c_{yx}c_{z}}
\\
{c_{y}c_{z}c_{x}}
	\ar@2@/_/ [dr] _-{c_{y}\gamma_{zx}}
\\
& {c_{y}c_{zx}}
\ar@2@/_/ [uur] _-{\alpha'_{y,zx}}
}
\]
Let us denote by $\Knuthcc_3(n)$ the $(3,1)$-polygraph whose 
underlying $2$-polygraph is $\Knuthcc_{2}(n)$ and whose set of $3$-cells is  
\[
\{
\;
T_{\eta,\epsilon\leftarrow \alpha'}^{-1}(T_{\gamma \leftarrow \alpha'}^{-1}(R_{\Gamma_3}(C'_{x,v,t})))\quad \text{for \typetroisbasep{$x$}{$v$}{$t$}{12}}
\;
\}
\;\cup\;
\{
\;
T_{\eta,\epsilon\leftarrow \alpha'}^{-1}(T_{\gamma \leftarrow \alpha'}^{-1}(R_{\Gamma_3}(D_{x,v,t})))\quad \text{for \typetroisbasep{$x$}{$v$}{$t$}{22}}
\;
\}.
\]
We extend the Tietze transformation $T_{\eta,\epsilon\leftarrow \alpha'}^{-1}$ into a Tietze transformation between $(3,1)$-polygraphs
\[
T_{\eta,\epsilon\leftarrow \alpha'}^{-1} : \tck{\CPC_3(n)} \fll \tck{\Knuthcc_3(n)},
\] 
where~
the $(3,1)$-polygraph~$\CPC_3(n)$ is a coherent presentation of the monoid $\P_n$ and the Tietze transformation~$T_{\eta,\epsilon\leftarrow \alpha'}^{-1}$ preserves the coherence property, hence we have the following result.

\begin{lemma}
For $n > 0$, the monoid $\P_{n}$ admits $\Knuthcc_3(n)$ as a coherent presentation.
\end{lemma}

\subsubsection{Step 3}

Finally, in order to obtain the Knuth coherent presentation, we perform an homotopical reduction, obtained using the homotopical reduction~$R_{\Gamma_2}$ on the~$(3,1)$-polygraph~$\Knuthcc_3(n)$ whose collapsible part~$\Gamma_2$ is defined by the $2$-cells~$\gamma_u$ of~$\CC_2(n)$ and the well-founded order~$\ordrecoldeglex$. 
Thus, for every $2$-cell $\gamma_{x_p\ldots x_1}:c_{x_{p}}\ldots c_{x_{1}}\odfl{} c_{x_{p}\ldots x_{1}}$ in~$\CC_2(n)$, we eliminate the generator $c_{x_{p}\ldots x_{1}}$ together with the $2$-cell~$\gamma_{x_p\ldots x_1}$, from the bigger column to the smaller one with respect to the order~$\ordrecoldeglex$. 

\subsubsection{Knuth coherent presentation}

Using the Tietze transformations constructed in the previous sections,  we consider the following composite of Tietze transformations
\[
\mathcal{R}\; :=\;   R_{\Gamma_2}\circ T_{\eta,\epsilon\leftarrow \alpha'}^{-1}\circ T_{\gamma \leftarrow \alpha'}^{-1}\circ R_{\Gamma_3}
\]
defined from $\tck{\overline{\Colo}_3(n)}$ to~$\tck{\Knuthcc_3(n)}$ as follows. Firstly, the transformation~$\mathcal{R}$ eliminates the $3$-cells of~$\overline{\Colo}_3(n)$ of the form $A_{x,v,t}$, $B_{x,v,t}$ and $C_{x,v,t}$ which are not of the form $C'_{x,v,t}$ and reduced its set of $2$-cells to~$\PreColo_2(n)$. Secondly, this transformation coherently replaces the $2$-cells~$\gamma_{x_{p}\ldots x_{1}}$ by the $2$-cells~$\alpha'_{x_p,x_{p-1}\ldots x_{1}}$, for each column~$x_p\ldots x_1$  such that~ $\len(x_{p}\ldots x_{1})>2$,
the $2$-cells~$\alpha'_{x,zy}$ by~$\eta_{x,y,z}^{c}$ for~$1\leq x\leq y < z\leq n$ and the $2$-cells~$\alpha'_{y,zx}$ by~$\epsilon_{x,y,z}^{c}$ for~\mbox{$1\leq x<y \leq z \leq n$}. Finally, for each column~$x_p\ldots x_1$, the transformation~$\mathcal{R}$ eliminates the generator~$c_{x_p\ldots x_1}$ together with  the $2$-cell~$\gamma_{x_p\ldots x_1}$ with respect to the order~$\ordrecoldeglex$.
  
Let us denote by $\Knuth_{3}(n)$ the extended presentation of the monoid $\P_{n}$ obtained from $\Knuth_{2}(n)$ by adjunction of  the following set of $3$-cells
\[
\{
\;
\mathcal{R}(C'_{x,v,t})\quad \text{for \typetroisbasep{$x$}{$v$}{$t$}{12}}
\;
\}
\;\cup\;
\{
\;
\mathcal{R}(D_{x,v,t})\quad \text{for \typetroisbasep{$x$}{$v$}{$t$}{22}}
\;
\}
.\]

The transformation~$\mathcal{R}$ being a composite of Tietze transformations, it follows the following result.

%

\begin{theorem}
\label{KnuthcoherentTheorem}
For $n > 0$, the $(3,1)$-polygraph $\Knuth_3(n)$ is a coherent presentation of the monoid~$\P_n$.
\end{theorem}

\subsubsection{Example: Knuth's coherent presentation of the monoid \pdf{\P_2}}

We have seen in Example~\ref{P2Example} that the $(3,1)$-polygraphs~$\Colo_3(2)$, $\overline{\Colo}_3(2)$ and $\PreColo_3(2)$ are equal. The coherent presentation~$\PreColo_3(2)$ is given by
\[
\PreColo_1(2)=\{c_1,c_2,c_{21}\},
\qquad
\PreColo_2(2)=\{\alpha_{2,1}, \alpha_{1,21}, \alpha_{2,21}\},
\qquad
\PreColo_3(2)=\{C'_{2,1,21}\},
\]
where $C'_{2,1,21}$ is the following $3$-cell:
\[
\xymatrix @C=3em @R=0.6em {
& {c_{21}c_{21}}
\ar@3 []!<-10pt,-10pt>;[dd]!<-10pt,+10pt> ^{C'_{2,1,21}}
\\
{c_2c_1c_{21}}
	\ar@2@/^/ [ur] ^-{\alpha_{2,1}c_{21}}
	\ar@2@/_/ [dr] _-{c_{2}\alpha_{1,21}}
\\
& {c_2c_{21}c_1}
	\ar@2 [r] _-{\alpha_{2,21}c_1}
& {c_{21}c_2c_1}
	\ar@2 [luu] _-{c_{21}\alpha_{2,1}}
}
\] 
By definition of the $2$-cells of $C_2(2)$, we have $\gamma_{21}:= \alpha_{2,1}$. Thus we obtain that~\mbox{$T_{\gamma \leftarrow \alpha'}^{-1}(C'_{2,1,21}) = C'_{2,1,21}$} up to replace all the $2$-cells~$\alpha_{2,1}$ in~$C'_{2,1,21}$ by~$\gamma_{21}$. Hence, the coherent presentation~$CPC_3(2)$ is equal to~$\PreColo_3(2)$.
In order to compute the $3$-cell~$T_{\eta,\epsilon\leftarrow \alpha'}^{-1}(T_{\gamma \leftarrow \alpha'}^{-1}(C'_{2,1,21}))$, the $2$-cells $\alpha_{1,21}$ and~$\alpha_{2,21}$ in~$C'_{2,1,21}$ are respectively replaced  by the $2$-cells~$\eta_{1,1,2}^{\text{c}}$ and~$\epsilon_{1,2,2}^{\text{c}}$  as in the following diagram
\begin{equation}
\label{T2transformation}
\xymatrix @C=3em @R=0.6em{
&&{c_{21}c_{21}}
    \ar@3 []!<-10pt,-10pt>;[dd]!<-10pt,+10pt> ^{C'_{2,1,21}}
\\
&c_2c_1c_{21}
    \ar@2@/^/ [ur] ^-{\gamma_{21}c_{21}}
   \ar@{..>}@2 @/_/ [dr] _-{c_2\xcancel{\alpha_{1,21}}}
\\
 {c_2c_1c_2c_1}   
  \ar@2@/^/ [ur] ^-{c_2c_1\gamma_{21}}
&& c_2c_{21}c_1
    \ar@{..>}@2  [r] _-{\xcancel{\alpha_{2,21}}c_1}
& c_{21}c_2c_1
   \ar@2 [luu] _-{c_{21}\gamma_{21}}
\\
\\
&&{c_2c_2c_1c_1}
\ar@2@/^/ [uull] ^-{c_2\eta_{1,1,2}^{\text{c}}}
 \ar@2 [r] _-{\epsilon_{1,2,2}^{\text{c}}c_1}
 \ar@2 [uu] _-{c_2\gamma_{21}c_1}
&{c_2c_1c_2c_1}
\ar@2 [uu] _-{\gamma_{21}c_2c_1}
}
\end{equation}
where the cancel symbol means that the corresponding $2$-cell is removed. Hence the coherent presentation~$\Knuthcc_3(2)$ of~$\P_2$ has for $1$-cells $c_1$, $c_2$ and $c_{21}$,  for $2$-cells~$\alpha_{2,1}$, $\alpha_{1,21}$ and $\alpha_{2,21}$ and the only $3$-cell~(\ref{T2transformation}).
Let us compute the Knuth coherent presentation $\Knuth_3(2)$. The $3$-cell~$R_{\Gamma_2}(T_{\eta,\epsilon\leftarrow \alpha'}^{-1}(T_{\gamma \leftarrow \alpha'}^{-1}(C'_{2,1,21})))$ is obtained from~(\ref{T2transformation}) by removing the $2$-cell~$\gamma_{21}$ together with the~$1$-cell~$c_{21}$. Thus we obtain the following $3$-cell, where the cancel symbol means that the corresponding element is removed,
\[
\xymatrix @C=3em @R=0.6em{
&&{\xcancel{c_{21}}\xcancel{c_{21}}}
\\
&c_2c_1\xcancel{c_{21}}
   \ar@{..>}@2 @/^/ [ur] ^-{\xcancel{\gamma_{21}c_{21}}}
\\
 {c_2c_1c_2c_1}   
 \ar@{..>}@2  @/^/ [ur] ^-{c_2c_1\xcancel{\gamma_{21}}}
&& c_2\xcancel{c_{21}}c_1
& \xcancel{c_{21}}c_2c_1
   \ar@{..>}@2 [luu] _-{c_{21}\xcancel{\gamma_{21}}}
\\
\\
&&{c_2c_2c_1c_1}
\ar@2@/^/ [uull] ^-{c_2\eta_{1,1,2}^{\text{c}}}
 \ar@2 [r] _-{\epsilon_{1,2,2}^{\text{c}}c_1}
\ar@{..>}@2[uu] _-{c_2\xcancel{\gamma_{21}}c_1}
&{c_2c_1c_2c_1}
\ar@{..>}@2[uu] _-{\xcancel{\gamma_{21}}c_2c_1}
}
\]
Hence, the Knuth coherent presentation $\Knuth_3(2)$ of the monoid $\P_2$ has generators $c_1$ and $c_2$ subject to the Knuth relations $\eta_{1,1,2}^{\text{c}} : c_2c_1c_1 \dfl c_1c_2c_1$ and $\epsilon_{1,2,2}^{\text{c}} : c_2c_2c_1 \dfl c_2c_1c_2$ and the following~$3$-cell
\[
\xymatrix@!C@C=3em{
c_2c_2c_1c_1
  \ar@2@/^3ex/ [r] ^{2\eta_{1,1,2}^{\text{c}}}="src"
  \ar@2@/_3ex/ [r] _{\epsilon_{1,2,2}^{\text{c}}1}="tgt"
& c_2c_1c_2c_1
    \ar@3 "src"!<-5pt,-15pt>;"tgt"!<-5pt,15pt>  ^-{\, C''}
}
\]
In this way, we obtain the Knuth coherent presentation of the monoid $\P_2$ that we obtain in Example~\ref{P2Example} as a consequence of the fact that the $2$-polygraph $\Knuth_2(2)$ is convergent.

\subsubsection{Coherent presentations in small ranks}
\label{SubSection:Computations}

Let us denote by $\Knuthkb_2(n)$ the convergent $2$-polygraph obtained from $\Knuth_2(n)$ by the Knuth-Bendix completion using the lexicographic order.
For $n=3$, the polygraph $\Knuthkb_2(3)$ is finite, but $\Knuthkb_2(n)$ is infinite for $n\geq 4$, \cite{KubatOkninski14}.
Let us denote by $\Knuthkb_3(n)$ the Squier completion of $\Knuthkb_2(n)$. For $n\geq 4$, the polygraph $\Knuthkb_2(n)$ having an infinite set of critical branching, the set of $3$-cells of $\Knuthkb_3(n)$ is infinite. However, the $(3,1)$-polygraph $\Knuth_3(n)$ constructed in this section is a finite coherent convergent presentation of $\P_n$.
Table~\ref{Table:NombreCellules} presents the number of cells of the coherent presentations~$\Knuth_3(n)$, $\overline{\Colo}_3(n)$ and~$\Colo_3(n)$ of the monoid~$\P_n$.

\medskip

\begin{table}[H]
\begin{small}
\hspace{-2cm}
\begin{tabular}{|r|r|r|r|r|r|r|r|r|r|}
  \hline
$n$  & $\Knuth_1(n)$ & $\Colo_1(n)$ & $\Knuth_2(n)$& $\Knuthkb_2(n)$&$\Colo_2(n)$&$\Knuthkb_3(n)$& $\Knuth_3(n)$& $\overline{\Colo}_3(n)$ & $\Colo_3(n)$ \\
  \hline
  $1$ &1 &1  &0 & 0& 0&0& 0 &0  & 0 \\
  $2$ &2 &3  &2 & 2& 3&1& 1 &1  & 1 \\
  $3$ &3 &7  &8 & 11& 22&27& 24 &34 &42 \\
  $4$ &4 &15  &20  &$\infty$& 115&$\infty$&242 &330  &621 \\
  $5$ &5 &31  &40  & $\infty$& 531&$\infty$& 1726&2225 &6893\\
  $6$ &6 &63  &70 & $\infty$& 2317&$\infty$& 10273&12635&67635\\
  $7$ &7 &127 &112 &$\infty$& 9822&$\infty$&  55016&65282 &623010\\
  $8$ &8 &255  &168 &$\infty$&40971&$\infty$&  275868&318708&5534197\\
  $9$ &9 &511  &240  &$\infty$&169255&$\infty$&   1324970&1500465&48052953\\
  $10$ &10 &1023  &330&$\infty$& 694837&$\infty$&   6178939&6892325 &410881483\\
  \hline
\end{tabular}
\end{small}
\caption{Number of cells of $(3,1)$-polygraphs~$\Knuth_3(n)$, $\overline{\Colo}_3(n)$ and~$\Colo_3(n)$, for $1\leq n \leq 10$.}
\label{Table:NombreCellules}
\end{table}

\subsubsection{Actions of plactic monoids on categories}

In~\cite{GaussentGuiraudMalbos14}, the authors give a description of the category of actions of a monoid on categories in terms of coherent presentations.
Using this description, Theorem~\ref{KnuthcoherentTheorem} allows to present actions of plactic monoids on categories as follows. The category~$\mathrm{Act}(\P_n)$ of actions of the monoid $\P_n$ on categories is equivalent to the category of $2$-functors from the $(2,1)$-category~$\tck{\Knuth_2(n)}$ to the category $\Cat$ of categories, that sends the $3$-cells of $\Knuth_3(n)$ to commutative diagrams in~$\Cat$. 

\subsubsection{Higher syzygies for the plactic monoid}

In \cite{GuiraudMalbos12advances}, the authors show how to extend a convergent presentation of a monoid into a polygraphic resolution of the monoid, that is, a cofibrant replacement of the monoid in the category of $(\infty,1)$-categories.
The column presentation $\Colo_2(n)$ of the monoid~$\P_n$ can then be extended into a polygraphic resolution whose $n$-cells, for every $n\geq 3$, are indexed by~\mbox{$(n-1)$-fold} branching of $\Colo_2(n)$.
We can explicit the $4$-cells of this resolution, which correspond to the confluence diagrams induced by critical triple branchings. That is,  for columns~$u$, $v$, $t$ and~$e$ in~$\col(n)$ such that\mbox{~\typedeuxbase{$u$}{$v$}{0},}~\typedeuxbase{$v$}{$t$}{0} and~\typedeuxbase{$t$}{$e$}{0}, there is a critical triple branching with source $c_{u}c_{v}c_tc_e$. Using the same arguments of Section~\ref{Generating triple confluences}, we can show that the confluence diagram induced by this triple branching is represented by the $3$-sphere $\Omega_{u,v,t,e}$ whose the source is the $3$-cell  
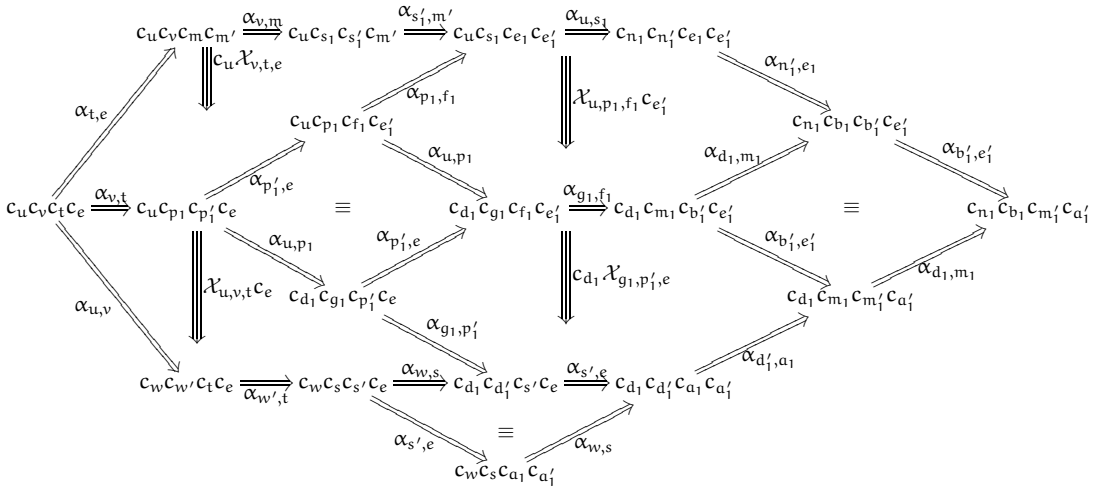
\[
\scalebox{0.75}{
\xymatrix @C=0.6em @R=1em{
&
&
{c_uc_vc_{m}c_{m'}}
\ar@2[rr] ^-{\alpha_{v,m}} _-{}
\ar@3 []!<10pt,12pt>;[dd]!<10pt,5pt> ^{c_u\mathcal{X}_{v,t,e}} 
&
&
{c_uc_{s_1}c_{s'_1}c_{m'}}
\ar@2[rr] ^-{\alpha_{s'_1,m'}} _-{}
&
&
{c_uc_{s_1}c_{e_1}c_{e'_1}}
\ar@2[rr] ^-{} ^-{\alpha_{u,s_1}} 
\ar@3 []!<30pt,-12pt>;[dddd]!<30pt,30pt> ^{\mathcal{X}_{u,p_1,f_1}c_{e'_1}} 
&
&
{c_{n_1}c_{n'_1}c_{e_1}c_{e'_1}}
\ar@2[ddrr] ^-{\alpha_{n'_1,e_1}} _-{}
\\
\\
&
&
&
&
{c_{u}c_{p_1}c_{f_1}c_{e'_1}}
\ar@2[uurr] _-{\alpha_{p_1,f_1}} _-{}
\ar@2[ddrr] ^-{\alpha_{u,p_1}} _-{}
&&&&&&
{c_{n_1}c_{b_1}c_{b'_1}c_{e'_1}}
\ar@2[ddrr] ^-{\alpha_{b'_1,e'_1}} _-{}
\\
\\
{c_{u}c_{v}c_{t}c_{e}}
\ar@2[uuuurr] ^-{\alpha_{t,e}} _-{}
\ar@2[rr] ^-{\alpha_{v,t}} _-{}
\ar@2[ddddrr] _-{\alpha_{u,v}} _-{}
&&
{c_{u}c_{p_1}c_{p'_1}c_{e}}
\ar@2[uurr] _-{\alpha_{p'_1,e}} _-{}
\ar@2[ddrr] ^-{} ^-{\alpha_{u,p_1}}
\ar@3 []!<5pt,-12pt>;[dddd]!<5pt,20pt> ^{\mathcal{X}_{u,v,t}c_e} 
&&
\equiv
&&
{c_{d_1}c_{g_1}c_{f_1}c_{e'_1}}
\ar@2[rr] ^-{\alpha_{g_1,f_1}} _-{}
\ar@3 []!<30pt,-12pt>;[dddd]!<30pt,30pt> ^{c_{d_1}\mathcal{X}_{g_1,p'_1,e}} 
&&
{c_{d_1}c_{m_1}c_{b'_1}c_{e'_1}}
\ar@2[uurr] ^-{\alpha_{d_1,m_1}} _-{}
\ar@2[ddrr] ^-{\alpha_{b'_1,e'_1}} _-{}
&&
\equiv
&&
{c_{n_1}c_{b_1}c_{m'_1}c_{a'_1}}
\\
\\
&&&&
{c_{d_1}c_{g_1}c_{p'_1}c_{e}}
\ar@2[uurr] ^-{\alpha_{p'_1,e}} _-{}
\ar@2[ddrr] ^-{\alpha_{g_1,p'_1}} _-{}
&&&&&&
{c_{d_1}c_{m_1}c_{m'_1}c_{a'_1}}
\ar@2[uurr] _-{\alpha_{d_1,m_1}} _-{}
\\
\\
&&
{c_{w}c_{w'}c_{t}c_{e}}
\ar@2[rr] _-{\alpha_{w',t}} _-{}
&&
{c_{w}c_{s}c_{s'}c_{e}}
\ar@2[rr] ^-{\alpha_{w,s}} _-{}
\ar@2[ddrr] _-{\alpha_{s',e}} _-{}
&&
{c_{d_1}c_{d'_1}c_{s'}c_{e}}
\ar@2[rr] ^-{\alpha_{s',e}} _-{}
&&
{c_{d_1}c_{d'_1}c_{a_1}c_{a'_1}}
\ar@2[uurr] ^-{} _-{\alpha_{d'_1,a_1}}
\\
\\
&&&&&&
{c_{w}c_{s}c_{a_1}c_{a'_1}}
\ar@2[uurr] _-{\alpha_{w,s}} _-{}
\ar@{} [uu] |-{\displaystyle \equiv}
}}
\]
and the target is the $3$-cell
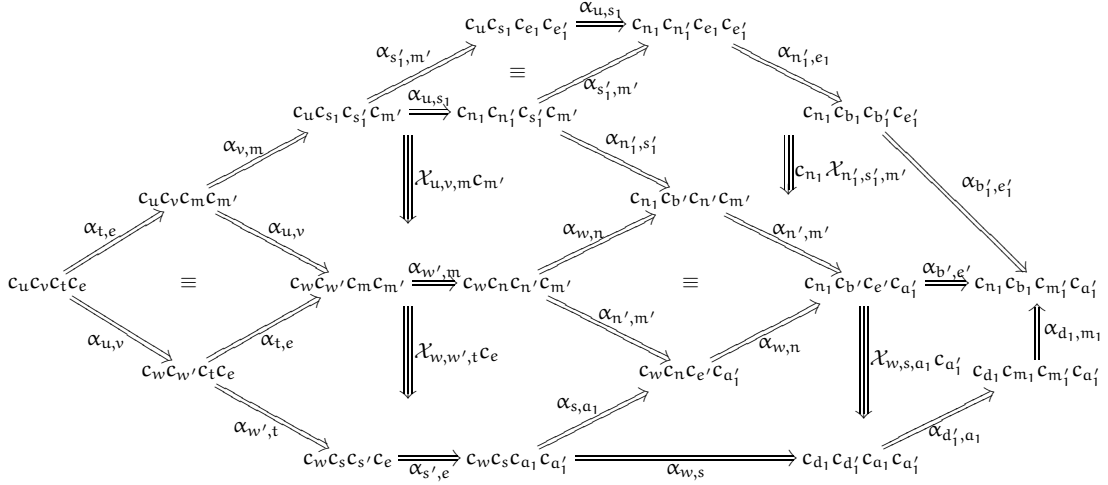
\[
\scalebox{0.75}{
\xymatrix @C=0.6em @R=1em{
&&&&&&
{c_{u}c_{s_1}c_{e_1}c_{e'_1}}
\ar@{} [dd] |-{\displaystyle \equiv}
\ar@2[rr] ^-{\alpha_{u,s_1}} _-{}
&&
{c_{n_1}c_{n'_1}c_{e_1}c_{e'_1}}
\ar@2[ddrr] ^-{\alpha_{n'_1,e_1}} _-{}
\\
\\
&&&&
{c_{u}c_{s_1}c_{s'_1}c_{m'}}
\ar@2[rr] ^-{\alpha_{u,s_1}} _-{}
\ar@2[uurr] ^-{\alpha_{s'_1,m'}} _-{}
\ar@3 []!<30pt,-12pt>;[dddd]!<30pt,30pt> ^{\mathcal{X}_{u,v,m}c_{m'}} 
&&
{c_{n_1}c_{n'_1}c_{s'_1}c_{m'}}
\ar@2[uurr] _-{\alpha_{s'_1,m'}} _-{}
\ar@2[ddrr] ^-{\alpha_{n'_1,s'_1}} _-{}
\ar@3 []!<135pt,-12pt>;[dd]!<135pt,-7pt> ^{c_{n_1}\mathcal{X}_{n'_1,s'_1,m'}} 
&&&&
{c_{n_1}c_{b_1}c_{b'_1}c_{e'_1}}
\ar@2[ddddrr] ^-{\alpha_{b'_1,e'_1}} _-{}
\\
\\
&&
{c_{u}c_{v}c_{m}c_{m'}}
\ar@2[uurr] ^-{\alpha_{v,m}} _-{}
\ar@2[ddrr] ^-{\alpha_{u,v}} _-{}
&&&&&&
{c_{n_1}c_{b'}c_{n'}c_{m'}}
\ar@2[ddrr] ^-{\alpha_{n',m'}} _-{}
\\
\\
{c_{u}c_{v}c_{t}c_{e}}
\ar@2[uurr] ^-{\alpha_{t,e}} _-{}
\ar@2[ddrr] _-{\alpha_{u,v}} _-{}
&&
\equiv
&&
{c_{w}c_{w'}c_{m}c_{m'}}
\ar@2[rr] ^-{\alpha_{w',m}} _-{}
\ar@3 []!<30pt,-12pt>;[dddd]!<30pt,30pt> ^{\mathcal{X}_{w,w',t}c_e} 
&&
{c_{w}c_{n}c_{n'}c_{m'}}
\ar@2[uurr] ^-{\alpha_{w,n}} _-{}
\ar@2[ddrr] ^-{\alpha_{n',m'}} _-{}
&&
\equiv
&&
{c_{n_1}c_{b'}c_{e'}c_{a'_1}}
\ar@2[rr] ^-{\alpha_{b',e'}} _-{}
\ar@3 []!<1pt,-12pt>;[dddd]!<1pt,20pt> ^{\mathcal{X}_{w,s,a_1}c_{a'_1}} 
&&
{c_{n_1}c_{b_1}c_{m'_1}c_{a'_1}}
\\
\\
&&
{c_{w}c_{w'}c_{t}c_{e}}
\ar@2[uurr] _-{\alpha_{t,e}} _-{}
\ar@2[ddrr] _-{\alpha_{w',t}} _-{}
&&&&&&
{c_{w}c_{n}c_{e'}c_{a'_1}}
\ar@2[uurr] _-{\alpha_{w,n}} _-{}
&&&&
{c_{d_1}c_{m_1}c_{m'_1}c_{a'_1}}
\ar@2[uu] _-{\alpha_{d_1,m_1}} _-{}
\\
\\
&&&&
{c_{w}c_{s}c_{s'}c_{e}}
\ar@2[rr] _-{\alpha_{s',e}} _-{}
&&
{c_{w}c_{s}c_{a_1}c_{a'_1}}
\ar@2[rrrr] _-{\alpha_{w,s}} _-{}
\ar@2[uurr] ^-{\alpha_{s,a_1}} _-{}
&&&&
{c_{d_1}c_{d'_1}c_{a_1}c_{a'_1}}
\ar@2[uurr] _-{\alpha_{d'_1,a_1}} _-{}
}}
\]
In the generating triple confluence, some columns may be empty and thus the indicated $2$-cells $\alpha$ may be identities. To facilitate the reading of the diagram, we have omitted the context of the $2$-cells $\alpha$.
More generally, we expect that the generating $n$-cell of the resolution has the form of the permutohedron of dimension $n$.

\section{Coherence and Lakshmibai-Seshadri's paths}
\label{CoherenceLakshmibaiSeshadriPaths}

In this section, we construct a coherent presentation of the monoid~$\P_n$ in term of Lakshmibai-Seshadri's paths. After recalling the notions of paths and crystal graphs, we briefly recall in~\ref{Subsection:TableauxLSpaths} the notion of Lakshmibai-Seshadri's paths and we refer the reader to~\cite{Littelmann94,Littelmann96} for more informations. Finally, we construct in~\ref{Subsection:CoherenceOfPaths} a convergent presentation of the monoid~$\P_n$ using Yamanouchi paths and a coherent presentation of it in terms of Lakshmibai-Seshadri's paths.

\subsection{Paths and crystal graphs}
\label{Paths and crystal graphs} 

Denote by $\mathfrak{gl}_{n}$ the general linear Lie algebra of $n$ by $n$ matrices.
Consider $\mathbb{R}^{n}$ with its canonical basis~$(\epsilon_1,\ldots, \epsilon_n)$.
The set of \emph{weights} of $\mathfrak{gl}_{n}$, denoted by $X$, is the lattice $\mathbb{Z}\epsilon_1\oplus\ldots\oplus\mathbb{Z}\epsilon_n$. The \emph{simple roots} of~$\mathfrak{gl}_{n}$ are the weight $\alpha_i = \epsilon_i - \epsilon_{i+1}$, for $1\leq i \leq n$.
Its \emph{fundamental weights} are~$\omega_i=\epsilon_1+\ldots+\epsilon_i$, for $1\leq i \leq n$. We will denote by~$\mathbb{F}$  the set of the fundamental weights. The \emph{dominant weights} are of the form~$a_1\omega_1 + \ldots+a_n\omega_n$ where~$a_1\geq \ldots\geq a_n\geq 0$. A dominant weight can be also written on the following form~$p_1\epsilon_1 + \ldots+p_n\epsilon_n$, with~$p_1\geq \ldots\geq p_n\geq 0$.

\subsubsection{Paths}

We will denote by $X_{\mathbb{R}}$ the real vector space~$X\otimes_{\mathbb{Z}}\mathbb{R}$.
A \emph{path} is a piecewise linear continuous map
$\pi: [0,1]\fll X_{\mathbb{R}}$.
We will consider paths up to a reparametrization, that is, a path $\pi$ is equal to any path $\pi\circ \varphi$, where $\varphi:[0,1]\fll [0,1]$ is a piecewise linear non-decreasing surjective continuous map. The target $\pi(1)$ of a path $\pi$ is called the \emph{weight} of $\pi$ and denoted by $\textrm{wt}(\pi)$.
We denote by 
\[
\Pi= \big\{\; \pi: [0,1]\fll X_{\mathbb{R}} \; \big| \; \pi(0) = 0 \text{ and } \pi(1)\in X \; \big\}
\]
the set of all paths whose source is $0$ and weight lies in $X$. We will denote by $\theta : [0,1]\fll X_{\mathbb{R}}$ the trivial path defined by $\theta(t)=0$, for any $t\in[0,1]$.
Given two paths  $\pi_1$ and $\pi_2$ in $\Pi$, the concatenation~$\pi_1\star\pi_2$ is defined  by:
\[
\pi_1\star\pi_2(t) :=\left \{
\begin{array}{ll}
\pi_{1}(2t)  &  \text{ for } 0\leq t\leq \frac{1}{2}, \\
\pi_{1}(1)+\pi_{2}(2t-1)&   \text{ for }\frac{1}{2}\leq t\leq 1. \\
\end{array}
\right. 
\]
With the concatenation $\star$ the set $\Pi$ forms a monoid whose unity is the trivial path and called the \emph{monoid of paths}.

\subsubsection{Words and paths}
\label{WordsandPaths}

For $\lambda$ in $X_{\mathbb{R}}$, consider  the path 
$\pi_{\lambda} : [0,1]\fll X_{\mathbb{R}}$
that connects the origin with~$\lambda$ by a straight line, that is $\pi_\lambda(t)=t\lambda$, for any $t\in[0,1]$. The path $\pi_{\lambda}$ is in $\Pi$ if and only if $\lambda$ is in~$X$.

Any $1$-cell in the free monoid $X_{\mathbb{R}}^\ast$ on $X_{\mathbb{R}}$ is a finite sequence of weights.
We define a map~$X_{\mathbb{R}}^\ast \fll \Pi$
sending any $1$-cell $w=\lambda_{1} \ldots \lambda_{r}$, with $\lambda_i$ in $X_{\mathbb{R}}$, to a path $\pi_w= \pi_{\lambda_1}\star \ldots\star \pi_{\lambda_r}$. The path~$\pi_w$ is in $\Pi$ if and only if $\lambda_1+\ldots+\lambda_r$ is in $X$. In addition, if we identify every path $\pi_{\epsilon_i}$ with the integer~$i$, for $1\leq i \leq n$, then the set of paths $\{\pi_{\epsilon_i}\;|\; 1\leq i \leq n \}$ is identified with the set~$[n]$. Hence, for every $1$-cell~$w=x_1\ldots x_r$ in the free monoid $[n]^{\ast}$, with $x_i$ in $[n]$, we  associate a path $\pi_w = \pi_{\epsilon_{x_1}}\star\ldots\star \pi_{\epsilon_{x_r}}$. We will denote by~$\Pi_W$ the free monoid over~$\{\pi_{\epsilon_i}\;|\; 1\leq i \leq n \}$.

\subsubsection{Root operators}
\label{Rootoperators}

Let  $\pi_w$ be a path in~$\Pi_W$. For each $i$ in $[n]$ and each simple root $\alpha_i$  of~$\mathfrak{gl}_n$, one defines the root operators
\[
e_{\alpha_i},f_{\alpha_i} : \Pi_W \fll \Pi_W\cup\{0\}
\]
as follows.
First, one considers the path~$\pi_{w}^{i}$ obtained by deleting all the paths other that $\pi_{\epsilon_i}$ and  $\pi_{\epsilon_{i+1}}$ from~$\pi_w$. 
Second, one removes the concatenation~$\pi_{\epsilon_i}\star\pi_{\epsilon_{i+1}}$ of adjacent paths, that is \linebreak with~\mbox{$\pi_{\epsilon_i}(1)=\pi_{\epsilon_{i+1}}(0)$}. After these two operations we obtain a new path. The second step of the  process is repeated until it is impossible to remove adjacent paths. Let $r$ and $s$ be respectively the number of paths~$\pi_{\epsilon_{i+1}}$ and $\pi_{\epsilon_{i}}$  in the final path. 
\begin{itemize}
\item If $r>0$, then $e_{\alpha_i}(\pi_w)$ is obtained by replacing in $\pi_w$ the rightmost path $\pi_{\epsilon_{i+1}}$ of the final path by~$\pi_{\epsilon_i}$ and the others paths of $\pi_w$ stay unchanged. If $r = 0$, then $e_{\alpha_i}(w) = 0$.
\item If $s>0$, then $f_{\alpha_i}(\pi_w)$ is obtained by replacing in $\pi_w$ the leftmost path $\pi_{\epsilon_i}$ of  the final path by~$\pi_{\epsilon_{i+1}}$ and the others paths of $\pi_w$ stay unchanged. If $s = 0$, we set $f_{\alpha_i}(w) = 0$.
\end{itemize}
These operators preserve the length of the paths. We have also that if $f_{\alpha}(\pi)=\pi^{'}\neq 0$ then~$e_{\alpha}(\pi^{'})=\pi$. 

\begin{example}
Consider the path $\pi_w= \pi_{\epsilon_3}\star\pi_{\epsilon_1}\star  \pi_{\epsilon_2}\star  \pi_{\epsilon_2}\star  \pi_{\epsilon_1}\star  \pi_{\epsilon_3}\star \pi_{\epsilon_3}\star\pi_{\epsilon_1}\star  \pi_{\epsilon_3}$. Let us compute~$f_{\alpha_1}(\pi_w)$ and~$e_{\alpha_1}(\pi_w)$. We have~$\pi_{w}^{1}= \pi_{\epsilon_1}\star  \pi_{\epsilon_2}\star  \pi_{\epsilon_2}\star  \pi_{\epsilon_1}\star  \pi_{\epsilon_1}$. After removing the concatenation~$\pi_{\epsilon_1}\star\pi_{\epsilon_2}$ of the adjacent paths~$\pi_{\epsilon_1}$ and~$\pi_{\epsilon_2}$, we can not eliminate more paths. 
Then the final path is   $\pi_{\epsilon_2}\star  \pi_{\epsilon_1}\star  \pi_{\epsilon_1}$, with~$r=1$ and~$s=2$. Hence we obtain that
\begin{align*}
f_{\alpha_1}(\pi_w) &= \pi_{\epsilon_3}\star\pi_{\epsilon_1}\star  \pi_{\epsilon_2}\star  \pi_{\epsilon_2}\star  \pi_{\epsilon_2}\star  \pi_{\epsilon_3}\star \pi_{\epsilon_3}\star\pi_{\epsilon_1}\star  \pi_{\epsilon_3},\\
e_{\alpha_1}(\pi_w) &= \pi_{\epsilon_3}\star\pi_{\epsilon_1}\star  \pi_{\epsilon_2}\star  \pi_{\epsilon_1}\star  \pi_{\epsilon_1}\star  \pi_{\epsilon_3}\star \pi_{\epsilon_3}\star\pi_{\epsilon_1}\star  \pi_{\epsilon_3}.
\end{align*}
\end{example}

\subsubsection{Crystal graphs}

A \emph{crystal graph} is a $1$-polygraph~$\mathcal{G}$ whose set of $0$-cells is~$\Pi_{W}$ and whose set of $1$-cells is
\[
\mathcal{G}_1\; := \; \big\{\; \pi\overset{i_\pi}{\rightarrow} f_{\alpha_i}(\pi) \;\big|\; i\in\{1,\ldots,n\} \;\big\}
.\]
Note that $\pi' = f_{\alpha_i}(\pi)$ if and only if~$\pi = e_{\alpha_i}(\pi')$,~\cite{Kashiwara94}. If there is no confusion, we will denote~$i_\pi$ by~$i$. 

\subsubsection{Connected components of crystal graphs}

For any path~$\pi$ in~$\Pi_{W}$, we denote by~$B(\pi)$ the connected component of the crystal graph containing~$\pi$. Every connected component contains a path~$\pi$ that satisfies the following property:
\[
e_{\alpha_i}(\pi) = 0,
\]
for any $1\leq i \leq n$ and called a \emph{highest weight path}. We will denote by~$\Pi_{W}^{+}$ the set of highest weight paths in~$\Pi_{W}$.
An \emph{isomorphism} between two connected components~$B(\pi)$ and~$B(\pi')$ is a bijective \break map~$\psi:B(\pi)\to B(\pi')$ that satisfies the following conditions:
\begin{enumerate}[{\bf i)}]
\item it is \emph{weight-preserving}, that is~$\textrm{wt}(\pi_w) = \textrm{wt}(\psi(\pi_w))$, for all~$\pi_w$ in~$B(\pi)$,
\item for all~$\pi_w$ and~$\pi_{w'}$ in~$B(\pi)$, if there is a $1$-cell~$\pi_w\overset{i}{\rightarrow}\pi_{w'}$, then there is a $1$-cell~~$\psi(\pi_w)\overset{i}{\rightarrow}\psi(\pi_{w'})$.
\end{enumerate}
Recall that for two paths $\pi_1$ and $\pi_2$ in $\Pi_W^{+}$,~$B(\pi_1)$ and~$B(\pi_2)$ are isomorphic if and only if their highest weight paths $\pi_1$ and $\pi_2$ have the same weight,~{\cite[Theorem~1]{Littelmann96}}.

\begin{example}
\label{CrysalExample}
For~$n=3$, the connected component~$B(\pi_{\epsilon_3}\star\pi_{\epsilon_1}\star\pi_{\epsilon_3})$ containing the path~$\pi_{\epsilon_3}\star\pi_{\epsilon_1}\star\pi_{\epsilon_3}$ has the following form
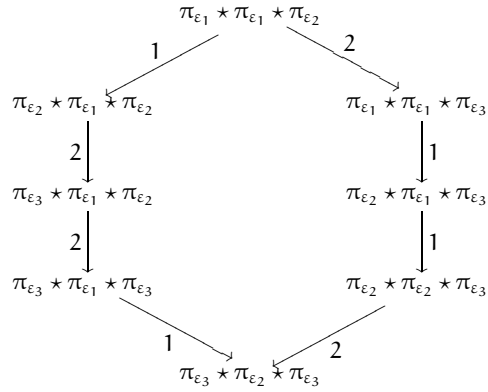
\[\scalebox{0.8}{
\xymatrix@C=0.5em @R=1cm{
&& {\pi_{\epsilon_1}\star\pi_{\epsilon_1}\star\pi_{\epsilon_2}}
 	\ar@<+0.4ex>[dl] _-{1}
 	\ar@<+0.4ex>[dr]^-{2}
 \\
&{\pi_{\epsilon_2}\star\pi_{\epsilon_1}\star\pi_{\epsilon_2}}
	\ar@<+0.4ex>[d]_-{2}
&& {\pi_{\epsilon_1}\star\pi_{\epsilon_1}\star\pi_{\epsilon_3}}
\ar@<+0.4ex>[d]^-{1}
\\
&{\pi_{\epsilon_3}\star\pi_{\epsilon_1}\star\pi_{\epsilon_2}}
\ar@<+0.4ex>[d]_-{2}
&& {\pi_{\epsilon_2}\star\pi_{\epsilon_1}\star\pi_{\epsilon_3}}
\ar@<+0.4ex>[d]^-{1}
\\
&{\pi_{\epsilon_3}\star\pi_{\epsilon_1}\star\pi_{\epsilon_3}}
\ar@<+0.4ex>[dr]_-{1}
&&{\pi_{\epsilon_2}\star\pi_{\epsilon_2}\star\pi_{\epsilon_3}}
\ar@<+0.4ex>[dl]^-{2}
\\
&&{\pi_{\epsilon_3}\star\pi_{\epsilon_2}\star\pi_{\epsilon_3}}
 } 
}
\]
The highest weight path of this connected component is~$\pi_{\epsilon_1}\star\pi_{\epsilon_1}\star\pi_{\epsilon_2}$.
\end{example}

\subsection{Tableaux and Lakshmibai-Seshadri's paths}
\label{Subsection:TableauxLSpaths}

\subsubsection{Tableaux}

Let~$\lambda=p_1\epsilon_1+ \ldots+p_k\epsilon_k$ be a dominant weight.  A \emph{Young tableau of shape $\lambda$} is a collection of boxes in left-justified rows filled by elements in~$[n]$ strictly increasing in the columns, such that the $i$th row contains $p_i$ boxes, for~$1\leq i \leq k$. 
For instance, a Young tableau of shape~$\lambda = 4\epsilon_1 + 3\epsilon_2 + \epsilon_3$ is the following diagram
\[\small
\young(1211,233,3)
\]
A \emph{tableau of shape}~$\lambda$, or \emph{tableau} for short, is a Young tableau of shape~$\lambda$ where the entries are non-decreasing in the rows. 
For example, a tableau of shape~$\lambda = 4\epsilon_1 + 3\epsilon_2 + \epsilon_3$ is the following diagram
\[\small
\raisebox{2.5ex}{\tableau{T} \;=\;}  \young(1112,233,3)
\]
The \emph{Japanese reading} of a tableau~$\tableau{T}$, denoted by~$\textrm{J}(\tableau T)$, is the $1$-cell obtained by reading the tableau~$\tableau T$ column-wise from top to bottom and from right to left. We will denote by~$\pi_{\textrm{J}(\tableau T)}$ the path in~$\Pi_{W}$ corresponding to the $1$-cell~$\textrm{J}(\tableau T)$, as presented  in~\ref{WordsandPaths}.  
For example, the Japanese reading of the previous tableau~$\tableau T$ is~$\textrm{J}(\tableau T) = 21313123$ and its corresponding path in~$\Pi_{W}$ is~$\pi_{\textrm{J}(\tableau T)} = \pi_{\epsilon_2}\star\pi_{\epsilon_1}\star\pi_{\epsilon_3}\star\pi_{\epsilon_1}\star\pi_{\epsilon_3}\star\pi_{\epsilon_1}\star\pi_{\epsilon_2}\star\pi_{\epsilon_3}$.

\subsubsection{Lakshmibai-Seshadri's paths}

By definition, a tableau of shape~$\omega_i$, for~$1\leq i \leq n$,   consists of one column with~$i$ elements satisfying~$x_1<\ldots<x_i$ from top to bottom. For each tableau of shape~$\omega_i$,  we will associate the path~$t\mapsto t(\epsilon_{x_1}+\ldots+\epsilon_{x_i})$ that connects the origin with the weight~$\epsilon_{x_1}+\ldots+\epsilon_{x_i}$ by a straight line. In this way, every column of a tableau will be represented by a path.
For a fundamental weight~$\omega_i$, the  \emph{Lakshmibai-Seshadri paths}, or \emph{L-S paths} for short, of shape~$\omega_i$ are the paths obtained from  all the columns of length~$i$.
 
\begin{example}
For~$n=3$, let us compute the L-S paths of shape~$\omega_1,\omega_2$ and~$\omega_3$. The only three columns of length~$1$ contains respectively the elements~$1$,~$2$ and~$3$, then the L-S paths of shape~$\omega_1$ are~$\pi_{\epsilon_1},\pi_{\epsilon_2}$ and~$\pi_{\epsilon_3}$.
The columns of length~$2$ are 
\[
\small
 \young(1,2)\;,\;\;\;\;\young(1,3)\;\;\;\; \text{ and } \;\;\;\;\young(2,3)
. \] 
Hence the L-S paths of~$\omega_2$ are the paths~$\pi_{\epsilon_1+\epsilon_2}$,~$\pi_{\epsilon_1+\epsilon_3}$ and~$\pi_{\epsilon_2+\epsilon_3}$.
 Moreover, the only column of length~$3$ is
 \[
\small \young(1,2,3) 
\] 
 Hence the only L-S path of shape~$\omega_3$ is~$\pi_{\epsilon_1+\epsilon_2+\epsilon_3}$.
\end{example} 
 
\subsubsection{Tableaux and L-S paths}

An \emph{L-S monomial of shape} $(\omega_1,\ldots, \omega_k)$ is a concatenation~$\pi_{1}\star\ldots\star\pi_{k}$,
where the path~$\pi_{i}$ is an L-S path of shape $\omega_{i}$, for every $1\leq i \leq k$.
A Young tableau of shape \linebreak $\lambda= a_{1}\omega_{1}+\ldots+a_{n}\omega_{n}$ is represented by the L-S monomial 
\[
\underset{1 \leq i \leq n}{\bigstar}{\pi_{1,\omega_i}\star\ldots\star\pi_{a_i,\omega_i}}
\]
where $\pi_{i,\omega_i}$ is an L-S path of shape $\omega_i$. That is, the first $a_{1}$ paths are of shape $\omega_{1}$, the next $a_{2}$ paths are of shape $\omega_{2}$ and so on until the final $a_{n}$ paths are of shape $\omega_{n}$. 
In the sequel, if there is no confusion we will identify Young tableaux with their corresponding L-S monomials.

\begin{example}
For $n=3$, the L-S monomial $\pi_{\epsilon_1}\star\pi_{\epsilon_1+\epsilon_3}\star\pi_{\epsilon_2+\epsilon_3}\star \pi_{\epsilon_1+\epsilon_2+\epsilon_3}$ of shape $\omega_1+2\omega_2+\omega_3$ corresponds to the following Young tableau
\[\small
\young(1211,233,3)
\]
The Japanese reading of this tableau is represented by the path~$\pi_{\epsilon_1}\star\pi_{\epsilon_1}\star\pi_{\epsilon_3}\star\pi_{\epsilon_2}\star\pi_{\epsilon_3}\star\pi_{\epsilon_1}\star\pi_{\epsilon_2}\star\pi_{\epsilon_3}$ in~$\Pi_{W}$. 
\end{example}

\subsubsection{Tableaux and crystal graphs}
\label{TableauxCrystal}

Let~$\tableau T$ be a Young tableau of shape~$\lambda$. 
To compute the root operators on~$\tableau T$, it is sufficient to compute them on the path~$\pi_{\textrm{J(\tableau T)}}$ and then to transform the resulted paths on  Young tableaux.
For example, to compute the operator~$f_{\alpha_1}$ on the tableau
\[
\small
\raisebox{1.5ex}{{\tableau{T}}\;=\;} \young(12,3)
\]
it is sufficient to compute~$f_{\alpha_1}(\pi_{\textrm{J(\tableau{T})}})$.  We have~$f_{\alpha_1}(\pi_{\textrm{J(\tableau{T})}}) = f_{\alpha_1}(\pi_{\epsilon_2}\star\pi_{\epsilon_1}\star\pi_{\epsilon_3}) = \pi_{\epsilon_2}\star\pi_{\epsilon_2}\star\pi_{\epsilon_3}$ and the path~$\pi_{\epsilon_2}\star\pi_{\epsilon_2}\star\pi_{\epsilon_3}$ corresponds to the following tableau
 \[
\small
\raisebox{1.5ex}{\tableau{T'}\;=\;} \young(22,3)
\]
Hence~$f_{\alpha_1}(\tableau{T}) = \tableau{T'}$.
We will denote by~$B(\tableau{T})$ the connected component of the crystal graph containing a Young tableau~$\tableau{T}$. Note that a tableau of shape~$\omega_1+\ldots+\omega_k$ is a vertex of the connected component~$B(\pi_{\omega_{1}}\star \ldots\star \pi_{\omega_{k}})$. Moreover, the  highest weight tableau of $B(\pi_{\omega_{1}}\star \ldots\star \pi_{\omega_{k}})$ has only $i$'s in the $i$-th row, for~\mbox{$1\leq i \leq k$}. In particular, the L-S paths of shape~$\omega_i$, for~$i$ in~$[n]$, are the vertices of the connected component~$B(\pi_{\omega_i})$.

\begin{example}
For~$n=3$, the L-S monomial $\tableau{T}=\pi_{\epsilon_2}\star\pi_{\epsilon_1+\epsilon_3}\star\pi_{\epsilon_1+\epsilon_3}\star \pi_{\epsilon_1+\epsilon_2+\epsilon_3}$ of \linebreak  shape~$\omega_1+2\omega_2+\omega_3$ corresponds to the following tableau
\[\small
\young(1112,233,3)
\]
The path $\pi_{\epsilon_2}\star\pi_{\epsilon_1+\epsilon_3}\star\pi_{\epsilon_1+\epsilon_3}\star \pi_{\epsilon_1+\epsilon_2+\epsilon_3}$ is a vertex of $B(\pi_{\omega_1}\star\pi_{\omega_2}\star\pi_{\omega_2}\star\pi_{\omega_3}),$ 
 with
\[\pi_{\epsilon_2}\star\pi_{\epsilon_1+\epsilon_3}\star\pi_{\epsilon_1+\epsilon_3}\star \pi_{\epsilon_1+\epsilon_2+\epsilon_3} = f_{\alpha_2}(f_{\alpha_1}(f_{\alpha_2}(\pi_{\omega_1}\star\pi_{\omega_2}\star\pi_{\omega_2}\star\pi_{\omega_3}))),\]
where the path $\pi_{\omega_1}\star\pi_{\omega_2}\star\pi_{\omega_2}\star\pi_{\omega_3}$ corresponds to the following tableau
\[\small
\young(1111,222,3)
\]
The tableau~$\tableau{T}$ is represented by the path~$\pi_{\textrm{J}(\tableau{T})}=\pi_{\epsilon_2}\star\pi_{\epsilon_1}\star\pi_{\epsilon_3}\star\pi_{\epsilon_1}\star\pi_{\epsilon_3}\star\pi_{\epsilon_1}\star\pi_{\epsilon_2}\star\pi_{\epsilon_3}$ in~$\Pi_{W}$.
\end{example}

\begin{example}
For~$n=3$, the tableaux of shape $\omega_1+\omega_2$ on the set $[3]$ are the vertices of the following connected component~$B(\pi_{\omega_1}\star\pi_{\omega_2})$
\[\scalebox{0.8}{
\xymatrix@C=0.01em @R=1.2cm{
&& {\pi_{\omega_1}\star\pi_{\omega_2} = \raisebox{-0.3cm}{\young(11,2)}}
 	\ar@<+0.4ex>[dl] _-{1}
 	\ar@<+0.4ex>[dr]^-{2}
 \\
&{\pi_{\epsilon_2}\star\pi_{\epsilon_1+\epsilon_2} = \raisebox{-0.3cm}{\young(12,2)}}
	\ar@<+0.4ex>[d]_-{2}
&& {\pi_{\epsilon_1}\star\pi_{\epsilon_1+\epsilon_3} =\raisebox{-0.3cm}{\young(11,3)}}
\ar@<+0.4ex>[d]^-{1}
\\
&{\pi_{\epsilon_3}\star\pi_{\epsilon_1+\epsilon_2} = \raisebox{-0.3cm}{\young(13,2)}}
\ar@<+0.4ex>[d]_-{2}
&& {\pi_{\epsilon_2}\star\pi_{\epsilon_1+\epsilon_2} =\raisebox{-0.3cm}{\young(12,3)}}
\ar@<+0.4ex>[d]^-{1}
\\
&{ \pi_{\epsilon_3}\star\pi_{\epsilon_1+\epsilon_3} =\raisebox{-0.3cm}{\young(13,3)}}
\ar@<+0.4ex>[dr]_-{1}
&&{ \pi_{\epsilon_2}\star\pi_{\epsilon_2+\epsilon_3} =\raisebox{-0.3cm}{\young(22,3)}}
\ar@<+0.4ex>[dl]^-{2}
\\
&&{\pi_{\epsilon_3}\star\pi_{\epsilon_2+\epsilon_3} =\raisebox{-0.3cm}{\young(23,3)}}
 } 
}
\]
The paths corresponding to the Japanese readings of the vertices of this connected component are the vertices of the connected component of Example~\ref{CrysalExample}.
\end{example}

\subsubsection{Yamanouchi path tableau}

A \emph{Yamanouchi path} is a path~$\pi$ in~$\Pi_{W}$ such that any of its left factor path~$\pi'$  satisfies 
\[
|\pi'|_{1}\geq\ldots \geq |\pi'|_{n}
\]
where~$|\pi'|_{i}$ denotes the number of occurrences of the path~$\pi_{\epsilon_i}$ in~$\pi'$. A path is a Yamanouchi path if and only if it is  a highest weight path,~{\cite[Proposition~2.6.1]{Shimozono05}}. As a consequence, all the paths of~$\Pi_{W}^{+}$ are Yamanouchi paths.
As we have seen previously, the highest weight tableau of the connected component containing tableaux of shape~$a_1\omega_1+\ldots+a_k\omega_k$ has only $i$'s in the $i$-th row for~$1\leq i\leq k$. Then this highest weight tableau is represented by the following Yamanouchi path
\[
\underbrace{(\pi_{\epsilon_{1}}\star\ldots\star\pi_{\epsilon_1})}_\textrm{$a_1$ times }\star\underbrace{(\pi_{\epsilon_1}\star\pi_{\epsilon_{2}})\star\ldots\star (\pi_{\epsilon_1}\star\pi_{\epsilon_{2}})}_\textrm{$a_2$ times }\star\ldots\star
\underbrace{(\pi_{\epsilon_1}\star\ldots\star\pi_{\epsilon_{k}})\star\ldots\star (\pi_{\epsilon_1}\star\ldots\star\pi_{\epsilon_{k}})}_\textrm{$a_k$ times}.
\]
A Yamanouchi path that represents a tableau is called a \emph{Yamanouchi path tableau}. Yamanouchi paths form a single plactic class whose representative path is a unique Yamanouchi path tableau,~{\cite[Lemma~5.4.7]{Lothaire02}}.

\subsubsection{Yamanouchi's map}

Let us define a map
\[
Y:\Pi_{W}^{+} \to\Pi_{W}^{+}
\] 
that transforms a non-Yamanouchi path tableau to a Yamanouchi path tableau as follows. Let~$\pi_w$ be a non-Yamanouchi path tableau, then~$Y(\pi_w)$ is equal to the path~$\pi_{\textrm{J(\tableau{T})}}$, where~$\tableau{T}$ is the tableau obtained from~$\pi_w$ by putting for every~$\pi_{\epsilon_i}$ in~$\pi_w$ an element~$i$ in the~$i$-th row of~$\tableau{T}$.

\begin{example}
For~$n=3$, the path~$\pi_w=\pi_{\epsilon_1}\star\pi_{\epsilon_1}\star\pi_{\epsilon_2}\star\pi_{\epsilon_3}\star\pi_{\epsilon_1}\star\pi_{\epsilon_2}\star\pi_{\epsilon_1}\star\pi_{\epsilon_2}\star\pi_{\epsilon_3}$ is a Yamanouchi path that is not a Yanamouchi path tableau. Moreover, this path can be transformed to the following tableau
\[
\small
\raisebox{3,5ex}{\tableau{T}\;=\;}\young(1111,222,33)
\]
after replacing each~$\pi_{\epsilon_1}$ in~$\pi_w$ by the element~$1$ in the first row of~$\tableau{T}$, each~$\pi_{\epsilon_2}$ in~$\pi_w$ by the element~$2$ in its second row and each~$\pi_{\epsilon_3}$ in~$\pi_w$ by the element~$3$ in its third row. Hence we obtain
\[
Y(\pi_w) = \pi_{\textrm{J(\tableau{T})}} = \pi_{\epsilon_1}\star\pi_{\epsilon_1}\star\pi_{\epsilon_2}\star\pi_{\epsilon_1}\star\pi_{\epsilon_2}\star\pi_{\epsilon_3}\star\pi_{\epsilon_1}\star\pi_{\epsilon_2}\star\pi_{\epsilon_3}.
\]
\end{example}

\subsection{Coherence of paths}
\label{Subsection:CoherenceOfPaths} 

In the free monoid~$\Pi_W$ over the set~$\{\pi_{\epsilon_i}\;|\; 1\leq i \leq n \}$, the  Knuth relations~(\ref{KnuthRelations}) can  be written  in the following form
\begin{equation}
\label{KnuthRelationsCrys}
\begin{array}{rl}
\big\{ \pi_{\epsilon_z}\star\pi_{\epsilon_x}\star\pi_{\epsilon_y} \odfll{\eta_{x,y,z}^{\pi}}& \pi_{\epsilon_x}\star\pi_{\epsilon_z}\star\pi_{\epsilon_y} \;\big|\;  1 \leq x\leq y<z  \leq n \big\}\\
&\;\cup\;
\big\{ \pi_{\epsilon_y}\star\pi_{\epsilon_z}\star\pi_{\epsilon_x}\odfll{\epsilon_{x,y,z}^{\pi}}\pi_{\epsilon_y}\star\pi_{\epsilon_x}\star\pi_{\epsilon_z} \;\big|\;  1 \leq x< y\leq z \leq n \big\}.
\end{array}
\end{equation}
We denote by $\Knuthpath_2(n)$ the $2$-polygraph whose set of $1$-cells is~$\{\pi_{\epsilon_i}\;|\; 1\leq i \leq n \}$ and whose set of $2$-cells is given by~(\ref{KnuthRelationsCrys}). The $2$-polygraphs~$\Knuthpath_2(n)$ and~$\Knuth_2(n)$ are Tietze equivalent, by the mapping~$i\mapsto\pi_{\epsilon_i}$ that induces an isomorphism between the two presented monoids.

\subsubsection{Equivalence on paths}
\label{Equivalenceonpaths}

Let $\pi_w$ and $\pi_{w'}$ be two paths in $\Pi_W$. One can define a relation~$\sim_{\textrm{path(n)}}$ on~$\Pi_W$ by : $\pi_w\sim_{\textrm{path(n)}} \pi_{w'}$ if, and only if, the two following conditions are satisfied:
\begin{enumerate}[{\bf i)}]
\item the connected components $B(\pi_w)$ and $B(\pi_{w'})$ are isomorphic, that is $\textrm{wt}(\pi_{w}^{+}) = \textrm{wt}(\pi_{w'}^{+})$, where $\pi_{w}^{+}$ and $\pi_{w'}^{+}$ are the highest weight  paths of $B(\pi_w)$ and $B(\pi_w')$.
\item $\pi_w$ and $\pi_{w'}$ have the same position in the components $B(\pi_w)$ and $B(\pi_{w'})$, that is, there exist~$i_{1},\ldots, i_{r}$ such that~$\pi_w=f_{\alpha_{i_1}}\cdot\cdot\cdot f_{\alpha_{i_r}}(\pi_{w}^{+})$ and~$\pi_{w'}=f_{\alpha_{i_1}}\cdot\cdot\cdot f_{\alpha_{i_r}}(\pi_{w'}^{+})$.
\end{enumerate}

\subsubsection{$2$-polygraph of crystals}
\label{2polygraphofcrystals}

Let~$\Crys^{0}_2(n)$ be the $2$-polygraph whose set of $1$-cells is~\mbox{$\{\pi_{\epsilon_i}\;|\; 1\leq i \leq n \}$} and whose set of $2$-cells is
\[
\big\{\; \pi_{w} \odfll{\vartheta_{\pi_w}} Y(\pi_{w}) \; \big| \; \pi_{w}\in \Pi_{W}^{+}
\;\;\text{and}\; \textrm{wt}(\pi_{w}) = \textrm{wt}(Y(\pi_{w})) \; \big\}
.
\]
For~$\pi_w$ in $\Pi_w^{+}$, the path~$f_{\alpha_{j_{k}}}\circ f_{\alpha_{j_{k-1}}}\circ\ldots\circ f_{\alpha_{j_{1}}}(\pi_w)$ will be also denoted by~$f_{\alpha_{j_{k}}}(\pi_w)$,
where for~\mbox{$i=1,\ldots,k$}, every~$j_i$ is an element of~$[n]$,~$\alpha_{j_i}$ is a simple root and~$f_{\alpha_{j_i}}$ is the corresponding root operator.
For~$k\geq 0$, let us define the $2$-polygraph~$\Crys^{k}_2(n)$ whose set of $1$-cells is~$\{\pi_{\epsilon_i}\;|\; 1\leq i \leq n \}$ and whose set of $2$-cells is
\[
\big\{\;f_{\alpha_{j_{k}}}( \pi_{w}) \odfll{\vartheta_{\pi_w}^{\alpha_{j_{k}}}} f_{\alpha_{j_{k}}}(Y(\pi_{w})) \; \big| \; \pi_{w}\in \Pi_{W}^{+}
\;\;\text{and}\; \textrm{wt}(\pi_{w}) = \textrm{wt}(Y(\pi_{w})) \; \big\}
.
\]

The \emph{$2$-polygraph of crystals} is the $2$-polygraph denoted by~$\Crys_2(n)$, whose $1$-cells are $\pi_{\epsilon_1},\ldots,\pi_{\epsilon_n}$ and whose set of $2$-cells is
\[
\underset{i\geq 0}{\cup}\Crys^{i}_2(n)
.
\]
By construction, the monoid presented by the $2$-polygraph~$\Crys_2(n)$ is isomorphic to the quotient of~$\Pi_W$ by the equivalence~$\sim_{\mathrm{path(n)}}$. 

\begin{theorem}
\label{Crystaltheorem}
For $n>0$, the $2$-polygraph~$\Crys_2(n)$ is a convergent presentation of the monoid~$\P_n$.
\end{theorem}

\begin{proof}
Thanks to the mapping~$i\mapsto \pi_{\epsilon_i}$, for any~$i$ in~$[n]$, the free monoid~$\Pi_W$ is identified to the monoid~$[n]^{\ast}$. By~\cite{DateJimboMiwa90,LascouxLeclercThibon95}, the equivalence~$\sim_{\text{plax(n)}}$ defined in~\ref{Placticcongruence} coincides with the equivalence~$\sim_{\text{path(n)}}$, taking into account that the column reading of Schensted's tableaux obtained by the row insertion~(\ref{Algorithm:Schensted}) is replaced by the Japanese reading of Schensted's tableaux which are obtained by a similar column insertion,~\cite{Schensted61}. Thus, the monoid~$\P_n$ is isomorphic to the quotient of~$\Pi_W$ by the equivalence~$\sim_{\textrm{path(n)}}$. Hence, the $2$-polygraph~$\Crys_2(n)$ is a  presentation of the monoid~$\P_n$.

Prove the convergence of the $2$-polygraph~$\Crys_2(n)$. The termination is proved by showing that~$\Crys_2(n)$ is compatible with  a total order~$\ordrenweight$ defined on the set~$\mathbb{F}^{n}$ as follows.
First, we fix an ordering~$\ordreweight$ on the set of fundamental weights~$\mathbb{F}$ of the Lie algebra~$\mathfrak{gl}_n$ by
\[
\omega_1 \ordreweight\omega_2\ordreweight\ldots\ordreweight\omega_n.
\] 
Let~$\ordrenweight$ be the lexicographic order on the set~$\mathbb{F}^{n}$ induced by the order~$\ordreweight$, that is,\linebreak $(\omega_{i_1},\ldots,\omega_{i_n})\ordrenweight (\omega_{i'_{1}},\cdots,\omega_{i'_{n}})$ if
\[
\omega_{i_1} \ordreweight \omega_{i'_{1}} \; \text{ or } \; [\omega_{i_1} = \omega_{i'_{1}} \; \text{ and } \; (\omega_{i_2},\ldots,\omega_{i_n})\ordrenweight (\omega_{i'_{2}},\ldots,\omega_{i'_{n}})],
\]
where for every~$1\leq k\leq n$,~$\omega_{i_k}$ and~$\omega_{i'_k}$ are fundamental weights in~$\mathbb{F}$. Then~$\ordrenweight$ is a well-ordering on the set~$\mathbb{F}^{n}$.
Since the root operators preserve the lengths of paths and the shapes of tableaux, we will suppose that all the paths are Yamanouchi paths. Note also that any path in~$\Pi_W$ has a unique decomposition as an L-S monomial~$\pi_1\star\ldots\star\pi_k$ of shape~$(\omega_{j_1},\ldots,\omega_{j_k})$, where the path~$\pi_i$ is an L-S path of maximal shape~$\omega_{j_i}$, for every~$1\leq i \leq k$ and~$1\leq j_i \leq n$. In this way, we will consider this unique decomposition for all the Yamanouchi paths.  
By construction of the Yamanouchi map~$Y$, every non-Yamanouchi  L-S monomial tableau is transformed to a Yamanouchi path tableau by beginning with the concatenation of its L-S paths of shape~$\omega_1$, after by the concatenation of its paths of shape~$\omega_2$ and so on until the concatenation of its L-S paths of maximal shape with respect to the order~$\ordreweight$. Then, for every $2$-cell~$\vartheta_{\pi_w}:\pi_{w} \dfl Y(\pi_{w})$ in~$\Crys_2(n)$, we have~$Y(\pi_{w})\ordernweight \pi_{w}$. Hence, the $2$-polygraph~$\Crys_2(n)$ is compatible with the order~$\ordrenweight$. Hence, rewriting an L-S monomial that is not a Yamanouchi path tableau always decreases it with respect to the order $\ordrenweight$.  Since every application of a $2$-cell in~$\Crys_2(n)$ yields a $\ordrenweight$-preceding L-S monomial, it follows that any sequence of rewriting using~$\Crys_2(n)$ must terminate.

Let us show that $\Crys_2(n)$  is confluent. Let~$\pi_w$ be a path in~$\Pi_{W}$ and  $\pi_{w'}$, $\pi_{w"}$ be two normal forms such that  $\pi_w\dfl \pi_{w'}$ and $\pi_w\dfl \pi_{w"}$. It is sufficient to prove that $\pi_{w'}=\pi_{w"}$. We have that~$\pi_{w'}$ is a Yamanouchi path tableau such that~$\pi_w\sim_{\textrm{path(n)}}\pi_{w'}$. 
Similarly, the path~$\pi_{w"}$ is a Yamanouchi path tableau such that~$\pi_w\sim_{\textrm{path(n)}}\pi_{w"}$. 
Since~$\pi_w\sim_{\textrm{path(n)}}\pi_{w'}\sim_{\textrm{path(n)}}\pi_{w"}$  and each plactic congruence contains exactly one Yamanouchi path tableau, we obtain that~$\pi_{w'}=\pi_{w"}$. Since the $2$-polygraph~$\Crys_2(n)$ is terminating, and rewriting any non-Yamanouchi path tableau must terminate with a unique normal form,~$\Crys_2(n)$ is confluent.
\end{proof}

As a consequence, we obtain that the $2$-polygraphs~$\Knuthpath_2(n)$ and~$\Crys_2(n)$ are Tietze equivalent.

\subsubsection{$2$-polygraph of paths}
\label{2polygraphofpaths}

Let denote by $\Path_1(n)$ the $1$-polygraph with only one $0$-cell and whose $1$-cells are all L-S paths of shape~$\omega_1,\ldots, w_n$.
For each pair $(\pi_u, \pi_{v})$ in $\Path_1(n)$  such that $\pi_u\star\pi_{v}$ is not a tableau, we define the $2$-cell
\[
\alpha_{\pi_u,\pi_{v}} : \pi_u\star\pi_{v}\dfl \pi_w\star\pi_{w'},
\]
where $\pi_w\star\pi_{w'}$ is the unique tableau such that $\pi_u\star\pi_{v}\sim_{\textrm{path(n)}} \pi_w\star\pi_{w'}$.
The \emph{$2$-polygraph of paths}, denoted by $\Path_2(n)$, is the  $1$-polygraph $\Path_1(n)$ extended by the set of $2$-cells  $\alpha_{\pi_u,\pi_{v}}$, where  $\pi_u$ and $\pi_{v}$ are in~$\Path_1(n)$  such that $\pi_u\star\pi_{v}$ is not a tableau. As shown by Littelmann, the $2$-polygraph~$\Path_2(n)$ is a presentation of the monoid~$\P_n$,~{\cite[Theorem~B]{Littelmann96}}.  Indeed, Littelmann showed that any L-S monomial is equivalent modulo relations in~$\Path_2(n)$ to a tableau. Note that we can also prove that the $2$-polygraph is a presentation of~$\P_n$ as follows. For every L-S path~$m$ in~$\Path_1(n)$, we consider its column reading~$C(m)$ as defined in Section~\ref{placticmonoids}. By the following composite of mappings
\[
\begin{array}{cccc}
\Path_1(n)&\longrightarrow\Pi_{W}&\longrightarrow\Knuth_1(n)^{\ast}&\longrightarrow\Colo_1(n)\\
m\quad&\longmapsto \pi_{C(m)}&\longmapsto C(m)\quad &\longmapsto c_{C(m)}\quad
\end{array}
\]
we transform each L-S path in~$\Path_1(n)$ into an element of~$\Colo_1(n)$.
Thus, the set~$\Path_1(n)$ is identified to the set~$\Colo_1(n)$. Similarly,  we transform through the previous mapping the left and right hands of the $2$-cells of~$\Path_2(n)$ into elements in~$\Colo_1(n)^{\ast}$. In this way, we identify the sets~$\Colo_2(n)$ and~$\Path_2(n)$. Hence the $2$-polygraph~$\Path_2(n)$ is Tietze equivalent to~$\Colo_2(n)$.

\subsubsection{Path coherent presentation}
\label{Pathcoherentpresentation}

Let us denote by~$\Path_3(n)$ the extended presentation of the monoid~$\P_n$ obtained from~$\Path_2(n)$ by adjunction of the following $3$-cell 
\[
\xymatrix @!C @C=2.3em @R=1em {
&
\pi_{e}\star \pi_{e'}\star\pi_t
  \ar@2[r] ^{\pi_{e}\alpha_{\pi_{e'},\pi_t}}
     \ar@3 []!<45pt,-8pt>;[dd]!<45pt,8pt> ^{}
&
\pi_{e}\star \pi_{b}\star \pi_{b'}
  \ar@2[dr] ^{\alpha_{\pi_e,\pi_b}\pi_{b'}}
\\
\pi_u\star \pi_v\star \pi_t
  \ar@2[ur] ^{\alpha_{\pi_u,\pi_v}\pi_t}
  \ar@2[dr] _{\pi_u\alpha_{\pi_v,\pi_t}}
&&&
\pi_{a}\star\pi_{d}\star \pi_{b'}
\\
&
\pi_u\star\pi_{w}\star\pi_{w'}
  \ar@2[r] _{\alpha_{\pi_u,\pi_w}\pi_{w'}}
&
\pi_{a}\star \pi_{a'}\star\pi_{w'}
 \ar@2[ur] _{\pi_{a}\alpha_{\pi_{a'},\pi_{w'}}}
}	
\]
where the paths $\pi_u\star\pi_v$ and~$\pi_v\star\pi_t$ are not tableaux and the paths $\pi_e\star \pi_{e'}$, \; $\pi_w\star \pi_{w'}$,\; 
$\pi_a\star \pi_{a'}$,\; $\pi_b\star \pi_{b'}$ and  $\pi_a\star\pi_d\star\pi_{b'}$ are tableaux such that
\[
\pi_u\star\pi_v \sim_{\textrm{path(n)}} \pi_e\star \pi_{e'},\;
\pi_v\star\pi_t \sim_{\textrm{path(n)}} \pi_w\star \pi_{w'},\;
\pi_u\star\pi_w \sim_{\textrm{path(n)}} \pi_a\star \pi_{a'},\;
\pi_{e'}\star\pi_t \sim_{\textrm{path(n)}} \pi_b\star \pi_{b'},\]
\[\pi_e\star\pi_b \sim_{\textrm{path(n)}} \pi_a\star \pi_d,\;
\pi_{a'}\star\pi_{w'} \sim_{\textrm{path(n)}} \pi_d\star \pi_{b'},
\; \text{ and }\;
\pi_u\star\pi_v\star\pi_t \sim_{\textrm{path(n)}}\pi_a\star\pi_d\star\pi_{b'}
.
\]
The $2$-polygraphs~$\Colo_2(n)$ and~$\Path_2(n)$ have the same properties. In particular, they have the same critical branchings and the same confluence diagrams. Hence, we obtain the following result as a direct consequence of Theorem~\ref{MainTheoremA}.

\begin{corollary}
\label{Path2Theorem}
For $n>0$, the $(3,1)$-polygraph $\Path_3(n)$ is a coherent presentation of the monoid~$\P_n$.
\end{corollary}

\subsubsection{Example: $2$-polygraph $\Path_2(3)$}

Let us compute the elements of the $2$-polygraph of paths~$\Path_2(3)$ of the monoid~$\P_3$.
The set of $1$-cells is
\[
\Path_1(3)\;=\; \big\{\;\pi_{\epsilon_1}, \pi_{\epsilon_2}, \pi_{\epsilon_3}, \pi_{\epsilon_1+ \epsilon_2}, \pi_{\epsilon_1+ \epsilon_3}, \pi_{\epsilon_2+ \epsilon_3}, \pi_{\epsilon_1+ \epsilon_2+\epsilon_3}  \;\big\}
.
\]
The left and right sides of the $2$-cells of~$\Path_2(3)$ are the paths corresponding to the vertices appearing at the same place in the following crystal isomorphisms
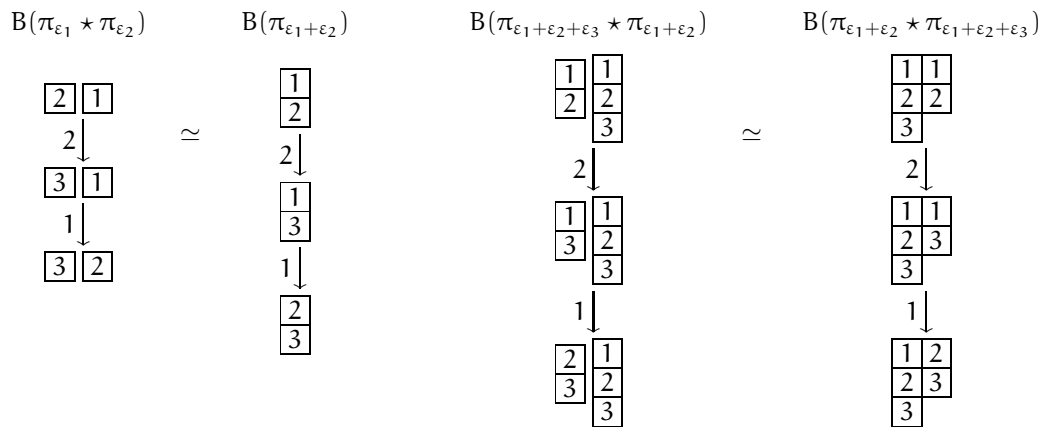
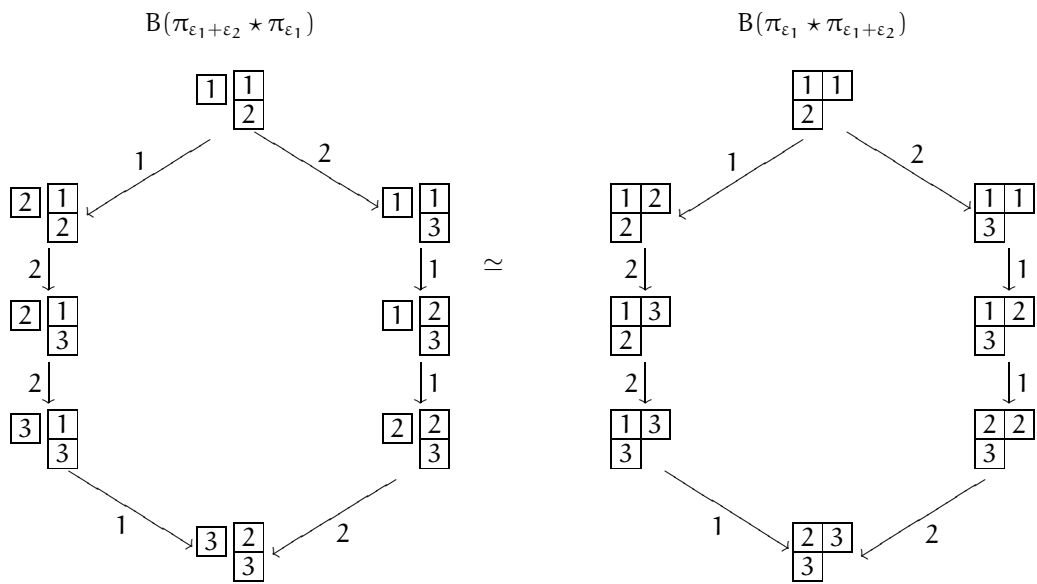
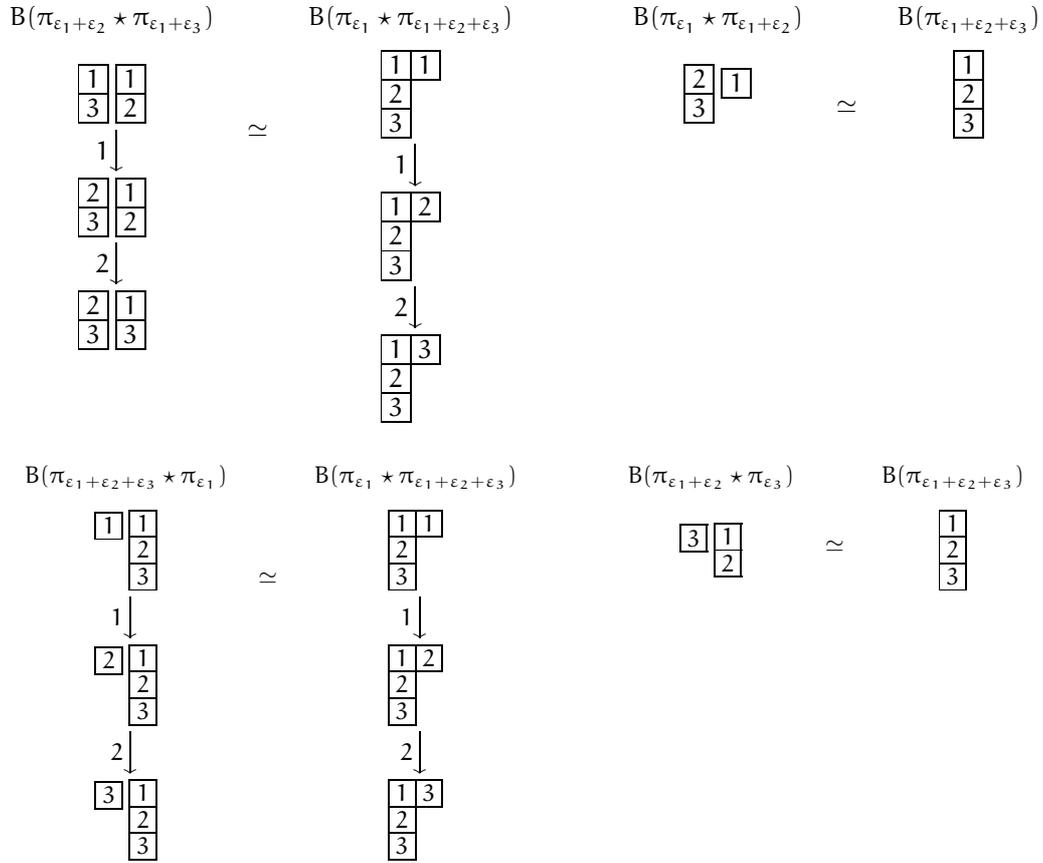
\[
\scalebox{0.9}{
\xymatrix@C=2em @R=0.6cm{
B(\pi_{\epsilon_1}\star \pi_{\epsilon_2})
\\
{\young(2)\; \young(1)}
	\ar@<+0.4ex>[d]_-{2}
\\
{\young(3)\; \young(1)}
\ar@<+0.4ex>[d]_-{1}
\\
{\young(3)\; \young(2)}
} 
\quad
\raisebox{-10ex}{$ \simeq $}
\quad
\xymatrix@C=3em @R=0.6cm{
B(\pi_{\epsilon_1+\epsilon_2})
\\
{ \young(1,2)}
	\ar@<+0.4ex>[d]_-{2}
\\
{ \young(1,3)}
\ar@<+0.4ex>[d]_-{1}
\\
{\young(2,3)}
}
\qquad
\qquad
\xymatrix@C=3em @R=0.6cm{
B(\pi_{\epsilon_1+\epsilon_2+\epsilon_3}\star\pi_{\epsilon_1+\epsilon_2})
\\
{\raisebox{3.5mm}{\young(1,2)}\; \young(1,2,3)}
	\ar@<+0.4ex>[d]_-{2}
\\
{\raisebox{3.5mm}{\young(1,3)}\; \young(1,2,3)}
\ar@<+0.4ex>[d]_-{1}
\\
{\raisebox{3.5mm}{\young(2,3)}\; \young(1,2,3)}
} 
\quad
\raisebox{-10ex}{$\simeq$}
\quad
\xymatrix@C=3em @R=0.6cm{
B(\pi_{\epsilon_1+\epsilon_2}\star\pi_{\epsilon_1+\epsilon_2+\epsilon_3})
\\
{ \young(11,22,3)}
	\ar@<+0.4ex>[d]_-{2}
\\
{ \young(11,23,3)}
\ar@<+0.4ex>[d]_-{1}
\\
{\young(12,23,3)}
}
}\]
\[
\scalebox{0.9}{
\xymatrix@C=2em @R=0.6cm{
&&B(\pi_{\epsilon_1+\epsilon_2}\star \pi_{\epsilon_1})
\\
&& {\raisebox{3.7mm}{\young(1)}\; \young(1,2)}
 	\ar@<+0.4ex>[dl] _-{1}
 	\ar@<+0.4ex>[dr]^-{2}
 \\
&{\raisebox{3.7mm}{\young(2)}\; \young(1,2)}
	\ar@<+0.4ex>[d]_-{2}
&& {\raisebox{3.7mm}{\young(1)}\; \young(1,3)}
\ar@<+0.4ex>[d]^-{1}
\\
&{\raisebox{3.7mm}{\young(2)}\; \young(1,3)}
\ar@<+0.4ex>[d]_-{2}
&& {\raisebox{3.7mm}{\young(1)}\; \young(2,3)}
\ar@<+0.4ex>[d]^-{1}
\\
&{\raisebox{3.7mm}{\young(3)}\; \young(1,3)}
\ar@<+0.4ex>[dr]_-{1}
&&{\raisebox{3.7mm}{\young(2)}\; \young(2,3)}
\ar@<+0.4ex>[dl]^-{2}
\\
&&{\raisebox{3.7mm}{\young(3)}\; \young(2,3)}
 } 
\quad
 \raisebox{-21ex}{$\simeq$}
\quad
 \xymatrix@C=2em @R=0.6cm{
&&B(\pi_{\epsilon_1} \star\pi_{\epsilon_1+\epsilon_2})
\\
&& {\young(11,2)}
 	\ar@<+0.4ex>[dl] _-{1}
 	\ar@<+0.4ex>[dr]^-{2}
 \\
&{ \young(12,2)}
	\ar@<+0.4ex>[d]_-{2}
&& {\young(11,3)}
\ar@<+0.4ex>[d]^-{1}
\\
&{ \young(13,2)}
\ar@<+0.4ex>[d]_-{2}
&& {\young(12,3)}
\ar@<+0.4ex>[d]^-{1}
\\
&{ \young(13,3)}
\ar@<+0.4ex>[dr]_-{1}
&&{ \young(22,3)}
\ar@<+0.4ex>[dl]^-{2}
\\
&&{\young(23,3)}
 } 
}\]
\[
\scalebox{0.9}{
\xymatrix@C=3em @R=0.6cm{
B(\pi_{\epsilon_1+\epsilon_2}\star\pi_{\epsilon_1+\epsilon_3})
\\
{\young(1,3)\; \young(1,2)}
	\ar@<+0.4ex>[d]_-{1}
\\
{\young(2,3)\; \young(1,2)}
\ar@<+0.4ex>[d]_-{2}
\\
{\young(2,3)\; \young(1,3)}
} 
\quad
\raisebox{-10ex}{$\simeq$}
\quad
\xymatrix@C=3em @R=0.6cm{
B(\pi_{\epsilon_1}\star\pi_{\epsilon_1+\epsilon_2+\epsilon_3})
\\
{ \young(11,2,3)}
	\ar@<+0.4ex>[d]_-{1}
\\
{ \young(12,2,3)}
\ar@<+0.4ex>[d]_-{2}
\\
{\young(13,2,3)}
} 
\qquad
\qquad
\xymatrix@C=3em @R=0.6cm{
B(\pi_{\epsilon_1}\star \pi_{\epsilon_1+\epsilon_2})
\\
{\young(2,3)\; \raisebox{3.7mm}{\young(1)}}
}
\quad
\raisebox{-8ex}{$ \simeq $}
\quad
\xymatrix@C=3em @R=0.6cm{
B(\pi_{\epsilon_1+\epsilon_2+\epsilon_3})
\\
{\young(1,2,3)}
}
}
\]
\[\small
\scalebox{0.9}{
\xymatrix@C=3em @R=0.6cm{
B(\pi_{\epsilon_1+\epsilon_2+\epsilon_3}\star \pi_{\epsilon_1})
\\
{\raisebox{7.5mm}{\young(1)}\; \young(1,2,3)}
	\ar@<+0.4ex>[d]_-{1}
\\
{\raisebox{7.5mm}{\young(2)}\; \young(1,2,3)}
\ar@<+0.4ex>[d]_-{2}
\\
{\raisebox{7.5mm}{\young(3)}\; \young(1,2,3)}
} 
\quad
\raisebox{-10ex}{$\simeq$}
\quad
\xymatrix@C=3em @R=0.6cm{
B(\pi_{\epsilon_1}\star \pi_{\epsilon_1+\epsilon_2+\epsilon_3})
\\
{ \young(11,2,3)}
	\ar@<+0.4ex>[d]_-{1}
\\
{ \young(12,2,3)}
\ar@<+0.4ex>[d]_-{2}
\\
{\young(13,2,3)}
}
\qquad\qquad
\xymatrix@C=3em @R=0.6cm{
B(\pi_{\epsilon_1+\epsilon_2}\star \pi_{\epsilon_3})
\\
{{\raisebox{3.7mm}{\young(3)}\; \young(1,2)}}
}
\quad
\raisebox{-7ex}{$\simeq$}
\quad
\xymatrix@C=3em @R=0.6cm{
B(\pi_{\epsilon_1+\epsilon_2+\epsilon_3})
\\
{\young(1,2,3)}
}}
\]
This presentation of the monoid~$\P_3$ can be extending into a coherent one by adding $42$ $3$-cells as mentioned in~\ref{SubSection:Computations}.

\begin{small}
\renewcommand{\refname}{\Large\textsc{References}}
\bibliographystyle{abbrv}
\bibliography{biblioPLAXIQUE}

\def\cprime{$'$}
\begin{thebibliography}{10}

\bibitem{Baker00}
T.~H. Baker.
\newblock An insertion scheme for {$C_n$} crystals.
\newblock In {\em Physical combinatorics ({K}yoto, 1999)}, volume 191 of {\em
  Progr. Math.}, pages 1--48. Birkh\"auser Boston, Boston, MA, 2000.

\bibitem{Berelle86}
A.~Berele.
\newblock A {S}chensted-type correspondence for the symplectic group.
\newblock {\em J. Combin. Theory Ser. A}, 43(2):320--328, 1986.

\bibitem{BokutChenChenLi15}
L.~Bokut, Y.~Chen, W.~Chen, and J.~Li.
\newblock New approaches to plactic monoid via {G}r\"obner--{S}hirshov bases.
\newblock {\em J. Algebra}, 423:301--317, 2015.

\bibitem{CainGrayMalheiro14}
A.~J. Cain, R.~D. Gray, and A.~Malheiro.
\newblock {Crystal monoids \& crystal bases: rewriting systems and biautomatic
  structures for plactic monoids of types $A_{n}$, $B_{n}$, $C_{n}$, $D_{n}$,
  and $G_2$ }.
\newblock arXiv:1412.7040, 2015.

\bibitem{CainGrayMalheiro15}
A.~J. Cain, R.~D. Gray, and A.~Malheiro.
\newblock Finite {G}r\"obner--{S}hirshov bases for {P}lactic algebras and
  biautomatic structures for {P}lactic monoids.
\newblock {\em J. Algebra}, 423:37--53, 2015.

\bibitem{DateJimboMiwa90}
E.~Date, M.~Jimbo, and T.~Miwa.
\newblock Representations of {$U_q(gl(n,{C}))$} at {$q=0$} and the
  {R}obinson-{S}chensted correspondence.
\newblock In {\em Physics and mathematics of strings}, pages 185--211. World
  Sci. Publ., Teaneck, NJ, 1990.

\bibitem{GuiraudDehornoy15}
P.~Dehornoy and Y.~Guiraud.
\newblock Quadratic normalization in monoids.
\newblock {\em Internat. J. Algebra Comput.}, 26(5):935--972, 2016.

\bibitem{Fulton97}
W.~Fulton.
\newblock {\em Young tableaux}, volume~35 of {\em London Mathematical Society
  Student Texts}.
\newblock Cambridge University Press, Cambridge, 1997.
\newblock With applications to representation theory and geometry.

\bibitem{GaussentGuiraudMalbos14}
S.~Gaussent, Y.~Guiraud, and P.~Malbos.
\newblock Coherent presentations of {A}rtin monoids.
\newblock {\em Compos. Math.}, 151(5):957--998, 2015.

\bibitem{Gordon83}
B.~Gordon.
\newblock A proof of the {B}ender-{K}nuth conjecture.
\newblock {\em Pacific J. Math.}, 108(1):99--113, 1983.

\bibitem{GuiraudMalbos12advances}
Y.~Guiraud and P.~Malbos.
\newblock Higher-dimensional normalisation strategies for acyclicity.
\newblock {\em Adv. Math.}, 231(3-4):2294--2351, 2012.

\bibitem{GuiraudMalbos14}
Y.~Guiraud and P.~Malbos.
\newblock Polygraphs of finite derivation type.
\newblock arXiv:1402.2587, Math. Struct. in Comp. Science, to appear, 2016.

\bibitem{GuiraudMalbosMimram13}
Y.~Guiraud, P.~Malbos, and S.~Mimram.
\newblock A homotopical completion procedure with applications to coherence of
  monoids.
\newblock In {\em 24th {I}nternational {C}onference on {R}ewriting {T}echniques
  and {A}pplications}, volume~21 of {\em LIPIcs. Leibniz Int. Proc. Inform.},
  pages 223--238. Schloss Dagstuhl. Leibniz-Zent. Inform., Wadern, 2013.

\bibitem{Hage14}
N.~Hage.
\newblock Finite convergent presentation of plactic monoid for type {C}.
\newblock {\em Internat. J. Algebra Comput.}, 25(8):1239--1263, 2015.

\bibitem{Hage15}
N.~Hage.
\newblock Finite convergent presentations of plactic monoids for semisimple lie
  algebras.
\newblock arXiv:1512.07813, 2015.

\bibitem{Humphreys08}
J.~E. Humphreys.
\newblock {\em Representations of semisimple {L}ie algebras in the
  {BGG}category {$\scr{O}$}}, volume~94 of {\em Graduate Studies in
  Mathematics}.
\newblock American Mathematical Society, Providence, RI, 2008.

\bibitem{Kashiwara91}
M.~Kashiwara.
\newblock Crystallizing the {$q$}-analogue of universal enveloping algebras.
\newblock In {\em Proceedings of the {I}nternational {C}ongress of
  {M}athematicians, {V}ol.\ {I}, {II} ({K}yoto, 1990)}, pages 791--797. Math.
  Soc. Japan, Tokyo, 1991.

\bibitem{Kashiwara94}
M.~Kashiwara.
\newblock On crystal bases.
\newblock In {\em Representations of groups ({B}anff, {AB}, 1994)}, volume~16
  of {\em CMS Conf. Proc.}, pages 155--197. Amer. Math. Soc., Providence, RI,
  1995.

\bibitem{KashiwaraNakashima94}
M.~Kashiwara and T.~Nakashima.
\newblock Crystal graphs for representations of the {$q$}-analogue of classical
  {L}ie algebras.
\newblock {\em J. Algebra}, 165(2):295--345, 1994.

\bibitem{KnuthBendix70}
D.~Knuth and P.~Bendix.
\newblock Simple word problems in universal algebras.
\newblock In {\em Computational {P}roblems in {A}bstract {A}lgebra ({P}roc.
  {C}onf., {O}xford, 1967)}, pages 263--297. Pergamon, Oxford, 1970.

\bibitem{Knuth70}
D.~E. Knuth.
\newblock Permutations, matrices, and generalized {Y}oung tableaux.
\newblock {\em Pacific J. Math.}, 34:709--727, 1970.

\bibitem{KubatOkninski14}
{\L}.~Kubat and J.~Okni{\'n}ski.
\newblock Gr\"obner-{S}hirshov bases for plactic algebras.
\newblock {\em Algebra Colloq.}, 21(4):591--596, 2014.

\bibitem{LascouxLeclercThibon95}
A.~Lascoux, B.~Leclerc, and J.-Y. Thibon.
\newblock Crystal graphs and {$q$}-analogues of weight multiplicities for the
  root system {$A_n$}.
\newblock {\em Lett. Math. Phys.}, 35(4):359--374, 1995.

\bibitem{LascouxSchutsenberger78}
A.~Lascoux and M.-P. Sch{\"u}tzenberger.
\newblock Sur une conjecture de {H}. {O}. {F}oulkes.
\newblock {\em C. R. Acad. Sci. Paris S\'er. A-B}, 286(7):A323--A324, 1978.

\bibitem{LascouxSchutzenberger81}
A.~Lascoux and M.-P. Sch{\"u}tzenberger.
\newblock Le mono\"\i de plaxique.
\newblock In {\em Noncommutative structures in algebra and geometric
  combinatorics ({N}aples, 1978)}, volume 109 of {\em Quad. ``Ricerca Sci.''},
  pages 129--156. CNR, Rome, 1981.

\bibitem{Lecouvey02}
C.~Lecouvey.
\newblock Schensted-type correspondence, plactic monoid, and jeu de taquin for
  type {$C_n$}.
\newblock {\em J. Algebra}, 247(2):295--331, 2002.

\bibitem{Lecouvey03}
C.~Lecouvey.
\newblock Schensted-type correspondences and plactic monoids for types {$B_n$}
  and {$D_n$}.
\newblock {\em J. Algebraic Combin.}, 18(2):99--133, 2003.

\bibitem{Littelmann94}
P.~Littelmann.
\newblock The path model for representations of symmetrizable {K}ac-{M}oody
  algebras.
\newblock In {\em Proceedings of the {I}nternational {C}ongress of
  {M}athematicians, {V}ol.\ 1, 2 ({Z}\"urich, 1994)}, pages 298--308.
  Birkh\"auser, Basel, 1995.

\bibitem{Littelmann96}
P.~Littelmann.
\newblock A plactic algebra for semisimple {L}ie algebras.
\newblock {\em Adv. Math.}, 124(2):312--331, 1996.

\bibitem{Lopatkin16}
V.~Lopatkin.
\newblock Cohomology rings of the plactic monoid algebra via a
  {G}r\"obner--{S}hirshov basis.
\newblock {\em J. Algebra Appl.}, 15(5):1650082, 30, 2016.

\bibitem{Lothaire02}
M.~Lothaire.
\newblock {\em Algebraic combinatorics on words}, volume~90 of {\em
  Encyclopedia of Mathematics and its Applications}.
\newblock Cambridge University Press, Cambridge, 2002.

\bibitem{Schensted61}
C.~Schensted.
\newblock Longest increasing and decreasing subsequences.
\newblock {\em Canad. J. Math.}, 13:179--191, 1961.

\bibitem{Schutzenberger77}
M.-P. Sch{\"u}tzenberger.
\newblock La correspondance de {R}obinson.
\newblock In {\em Combinatoire et repr\'esentation du groupe sym\'etrique
  ({A}ctes {T}able {R}onde {CNRS}, {U}niv. {L}ouis-{P}asteur {S}trasbourg,
  {S}trasbourg, 1976)}, pages 59--113. Lecture Notes in Math., Vol. 579.
  Springer, Berlin, 1977.

\bibitem{Sheats99}
J.~T. Sheats.
\newblock A symplectic jeu de taquin bijection between the tableaux of {K}ing
  and of {D}e {C}oncini.
\newblock {\em Trans. Amer. Math. Soc.}, 351(9):3569--3607, 1999.

\bibitem{Shimozono05}
M.~Shimozono.
\newblock Crystals for dummies.
\newblock Notes, URL:www.aimath. org/WWN/kostka/crysdumb.pdf,2005.

\bibitem{Squier94}
C.~C. Squier, F.~Otto, and Y.~Kobayashi.
\newblock A finiteness condition for rewriting systems.
\newblock {\em Theoret. Comput. Sci.}, 131(2):271--294, 1994.

\bibitem{Sundaram90}
S.~Sundaram.
\newblock Orthogonal tableaux and an insertion algorithm for {${\rm
  SO}(2n+1)$}.
\newblock {\em J. Combin. Theory Ser. A}, 53(2):239--256, 1990.

\bibitem{Tietze08}
H.~Tietze.
\newblock \"{U}ber die topologischen {I}nvarianten mehrdimensionaler
  {M}annigfaltigkeiten.
\newblock {\em Monatsh. Math. Phys.}, 19(1):1--118, 1908.

\end{thebibliography}
\end{small}

\vfill
\begin{flushright}
\begin{small}
\noindent \textsc{Nohra Hage} \\
\url{nohra.hage@univ-st-etienne.fr} \\
Univ Lyon, Universit\'e Jean Monnet\\ 
CNRS UMR 5208, Institut Camille Jordan\\
Maison de l'Universit\'e, 10 rue Tr\'efilerie, CS 82301\\
F-42023 Saint-\'Etienne Cedex 2, France

\bigskip
\noindent \textsc{Philippe Malbos} \\
\url{malbos@math.univ-lyon1.fr} \\
Univ Lyon, Universit\'e Claude Bernard Lyon 1\\
CNRS UMR 5208, Institut Camille Jordan\\
43 blvd. du 11 novembre 1918\\
F-69622 Villeurbanne cedex, France
\end{small}
\end{flushright}
\end{document}